\documentclass[11pt,article]{amsart}
\usepackage{color}
\usepackage[numbers,sort&compress]{natbib}
\usepackage{pifont,latexsym,ifthen,amsthm,rotating,calc,textcase,booktabs}
\usepackage{amssymb, amsthm, amsmath, mathtools, bm, bbm}
\usepackage{amsthm}
\numberwithin{equation}{section}
\usepackage[dvipsnames, svgnames, x11names]{xcolor}
\usepackage{graphicx}
\usepackage{tikz}
\usepackage{stmaryrd}
\usepackage{graphicx}
\usepackage[latin1]{inputenc}
\usepackage[T1]{fontenc}
\usepackage[english]{babel}
\usepackage{listings}
\usepackage{xcolor,mathrsfs,url}
\usepackage{ifthen}
\usepackage{fancyhdr}
\usepackage{verbatim}
\usepackage{enumitem}
\usepackage{todonotes}
\usepackage{float}
\usepackage{subfigure}
\usepackage{hyperref}
\usepackage{appendix}
\usepackage[percent]{overpic}
\newtheorem{theorem}{Theorem}[section]
\newtheorem{lemma}[theorem]{Lemma}

\newtheorem{proposition}[theorem]{Proposition}

\newtheorem{Assumption}{Assumption}[section]

\newtheorem{RHP}{RH problem}[section]

\usepackage {caption}

\newcommand*{\dif}{\mathop{}\!\mathrm{d}}
\definecolor{Stealth}{RGB}{50, 100, 150}

\usepackage{tikz} % Optional PGF libraries
\usetikzlibrary{shapes,decorations.pathmorphing}
\usepackage{pgf}
\usepackage{pgflibraryarrows}
\usepackage{pgflibrarysnakes}
\usetikzlibrary{decorations.text}
\usepgfmodule{shapes}
\usetikzlibrary{decorations.pathmorphing}
\usetikzlibrary{decorations.markings}
\usetikzlibrary{patterns}
\usetikzlibrary{automata}
\usetikzlibrary{positioning}
\usetikzlibrary {arrows.meta}
\usepackage{pgfplots}
\usetikzlibrary{backgrounds}
\tikzset{partial ellipse/.style args={#1:#2:#3}{insert path={+ (#1:#3) arc (#1:#2:#3)} }}
\tikzset{-Stealth-/.style={decoration={ markings, mark=at position #1 with {\arrow{Stealth}}},postaction={decorate}}}
\tikzset{->-/.style={decoration={ markings, mark=at position #1 with {\arrow{>}}},postaction={decorate}}}

\definecolor{Dgreen}{RGB}{0,153,0}

\hypersetup{hypertex=true,
            colorlinks=true,
            linkcolor=blue,
            anchorcolor=blue,
            citecolor=red}

\DeclareMathOperator*{\im}{Im}
\DeclareMathOperator*{\re}{Re}
\newcommand{\R}{\mathbb{R}}

\newcommand{\bigo}[1]{\mathcal{O} \left( #1 \right) }

\newcommand{\dd}{\mathrm{d}}

\newcommand{\E}{\mathrm{e}}

\begin{document}

\title[ Painlev\'e \uppercase\expandafter{\romannumeral34\relax} asymptotics for the defocusing mKdV equation in transition regions]{ Painlev\'e \uppercase\expandafter{\romannumeral34\relax} asymptotics for the defocusing mKdV equation with step-like initial data in transition regions}
\author{Engui Fan$^1$, Zhaoyu Wang$^2$,   Yidan Zhang$^3$ }
\footnotetext[1]{ \  School of Mathematical Sciences and Key Laboratory for Nonlinear Science, Fudan   University, Shanghai 200433,  China.   E-mail: \texttt{faneg@fudan.edu.cn.}}
\footnotetext[2]{ \
	Department of Mathematics and Newtouch Center for Mathematics of Shanghai University, Shanghai University, Shanghai 200444,  China. E-mail: \texttt{zhaoyuwang@shu.edu.cn.}}
\footnotetext[3]{ \  School of Mathematical Sciences, Fudan   University, Shanghai 200433,  China. E-mail: \texttt{23110840018@m.fudan.edu.cn.}}

\maketitle
\begin{abstract}
In this paper, we study the Cauchy problem for the defocusing modified Korteweg-de Vries (mKdV) equation with a step-like initial data.
Based on the Riemann-Hilbert problem associated with the mKdV equation, we
derive the long-time asymptotic expansion of the solution to the defocusing mKdV equation in two transition regions
 using the nonlinear steepest descent method. It comes out that the leading term in the expansion is shown to match the corresponding background constants. The subleading term, however, decays at the order $\mathcal{O}(t^{-2/3})$, and its coefficient is derived from the associated Painlev\'e \uppercase\expandafter{\romannumeral34\relax} model.
%In a recent paper, we study the long time asymptotics for Cauchy problem of the defocusing modified Korteweg-de Vries (mKdV) equation with step-like initial data approaching nonzero constants $c_l$ and $c_r$ as $x \to -\infty$ and $x\to+\infty$ for $c_l>c_r>0$, whose solution exhibits a rarefaction wave structure. The present paper deals with the asymptotic analysis in two transition regions. After characterizing the global solution of the Cauchy problem in terms of a Riemann-Hilbert (RH) problem, we apply the nonlinear steepest descent method to it. With the leading term being the background constants, the sub-leading term comes from the Painlev\'e \uppercase\expandafter{\romannumeral34\relax} parametrix and admits the decay of $\mathcal{O}(t^{-2/3})$ as $t$ large.
\\[4pt]
{\bf Key words:} Defocusing mKdV equation, Riemann-Hilbert problem, nonlinear steepest descent method, long-time asymptotics, Painlev\'e \uppercase\expandafter{\romannumeral34\relax}.\\[4pt]
{\bf MSC 2020:}  35Q51; 35Q15; 35C20; 35P25; 34M55.

\end{abstract}

\tableofcontents
%\quad

\section{Introduction}
\noindent

In this paper, we investigate the  long-time
 asymptotics in transition regions  for the    defocusing modified Korteweg-de Vries (mKdV) equation with step-like
initial data
\begin{alignat}{2}
&q_t(x,t)-6q^2(x,t)q_{x}(x,t)+q_{xxx}(x,t)=0, &\qquad&x\in\R,\quad t> 0, \label{equ:mkdv} \\
&q(x,0)=q_0(x)\to
\begin{cases}
c_{l}, &x\to -\infty,\\
c_{r}, &x\to +\infty,
\end{cases} \label{Initial data}
\end{alignat}
where $c_{l}>c_{r}>0$, and under the assumption that the solution $q(x,t)$ approaches to some nonzero real constants as $x \to \pm \infty$ respectively,
i.e.,
\begin{equation} \label{boundaryconditions}
q(x,t)\to
\begin{cases}
c_{l}, & x\to -\infty, \\
c_{r}, & x\to +\infty,
\end{cases}
\end{equation}

% ensure compatibility of the boundary conditions \eqref{boundaryconditions} with the evolution
%\eqref{equ:mkdv}--\eqref{Initial data}, the initial data $q_0(x)$ must satisfy
%\begin{equation}\label{condition:q0limits}
%q_0(x)\to
%\begin{cases}
%c_{l}, &x\to -\infty,\\
%c_{r}, &x\to +\infty.
%\end{cases}
%\end{equation}
%The precise rates of convergence in \eqref{boundaryconditions} and \eqref{condition:q0limits} will be specified in the main theorems.

The defocusing mKdV equation \eqref{equ:mkdv} is a canonical model in mathematical physics,
describing various nonlinear phenomena such as acoustic waves and phonons in certain
anharmonic lattices \cite{Zab1967,Ono1992}, and Alfv\'en waves in cold, collisionless
plasmas \cite{kaku1969,Kha1998}. Early works on the long-time asymptotics of the mKdV equation \eqref{equ:mkdv} concentrated on initial data vanishing as $x\to\pm\infty$.
%see the works \cite{AblowitzSegur, Zakharov-Manakov-1976} based on the inverse scattering method.
%A major advance was made by Deift and Zhou \cite{DZAnn}, who introduced the nonlinear steepest descent method for an oscillatory RH problem and used it to obtain detailed long-time asymptotics of the defocusing mKdV equation with Schwartz class initial data. Such a technique has been widely applied to a variety of integrable nonlinear partial differential equations (PDEs) \cite{DIZSurvey}.
We also refer the readers to \cite{CandLdmkdv,lenellsmkdv} for long-time asymptotics of the defocusing mKdV equation with vanishing initial data in lower regularity spaces.

The study of long-time asymptotics for Cauchy problems of nonlinear integrable systems with step-like initial data has a long history
and can be traced back to the pioneering work \cite{GuPitae} of Gurevich and Pitaevskii on the Korteweg-de Vries (KdV) equation.
Working within the framework of Whitham modulation theory \cite{Whitham1974}, they predicted the emergence of
highly oscillatory structures, now referred to as dispersive shock waves in the long-time dynamics.
The rigorous mathematical justification of this phenomenon was subsequently done by Khruslov \cite{Khruslov1976}
via the inverse scattering transform in terms of the Marchenko integral equations and the so-called asymptotic soliton.
Rigorous asymptotic analysis of step-like Cauchy problems for integrable systems began with the papers \cite{MonItsKot2009, BuckVena2007}.
Since then, the long-time asymptotics of integrable systems with step-like initial data have been extensively
studied: for the KdV equation \cite{EgorovaKDVstep-a, EgorovaKDVstep-b};
for the focusing nonlinear Schr\"odinger equation (NLS) \cite{MonKotSheIMRN,MonLenSheCMP,MonLenSheCMP3};
for the focusing mKdV equation \cite{KotMinakovJMP,MinakovJPAMaTheor,KotMinakovJMPAG1,KotMinakovJMPAG2,GraMinakovSIAM};
as well as for some nonlocal integrable systems \cite{RDJDE2021,RDCMP,RDJMPAG,xnonlocalmkdv}.

Since the scaling transformation
$$q\to-q, \ \ x\to -x,\ \ t\to-t$$
 keep the form of the defocusing mKdV equation \eqref{equ:mkdv} unchanged,
it is sufficient to consider two  cases  %$c_l \ge |c_r|$ ( or $c_r \ge |c_l|$).
 $c_l \ge c_r \ge 0$ and $c_l >0>c_r>-c_l$, which  lead to quite different asymptotic analysis that a kink soliton region appears in the later case \cite{WW2026}.
  For the case  $c_r>c_l>0$, the long-time asymptotics of dispersive shock wave solution has been obtained in \cite{WW2026}.  For the case $c_l>c_r>0$,
 we   obtained   long-time asymptotics for the Cauchy problem \eqref{equ:mkdv}--\eqref{Initial data}  in   three main regions
 in our recent paper \cite{xuzyd}.  More precisely, we summarize    as follows (see Figure \ref{fig:cone}):
\begin{itemize}
    \item Left region $\mathcal{R}_{\textup{\uppercase\expandafter{\romannumeral1}}}$:= $\left\{(x,t): \xi<-\frac{c_l^2}{2}, \ \xi:=\frac{x}{12t} \right\}$,
    in which the leading term is given by the constant $c_l$ and the sub-leading term is expressed in terms of the parabolic cylinder function;
    \item Middle region $\mathcal{R}_{\textup{\uppercase\expandafter{\romannumeral1}}}$:= $\left\{(x,t): -\frac{c_l^2}{2}<\xi<-\frac{c_r^2}{2} \right\}$, in which  the leading term is given by a slowly varying factor, while the sub-leading term is derived from the Airy function;
    \item Right region $\mathcal{R}_{\textup{\uppercase\expandafter{\romannumeral3}}}$:=  $\left\{(x,t):-\frac{c_r^2}{2} < \xi<-\frac{c_r^2}{6} \right\}\cup \left\{(x,t): \xi>-\frac{c_r^2}{6} \right\}$, in which the leading term is given by the constant $c_r$, and the sub-leading term arises either from the Bessel function or from the non-analytic part of the reflection coefficient on the real line.
\end{itemize}
 The remaining problem is to characterize the long-time asymptotic behavior in the regions adjacent to these three regions, namely the transition regions, 
    \begin{itemize}
        \item The first transition region $\mathcal{T}_{\textup{\uppercase\expandafter{\romannumeral1}}}:=\left\{(x,t): \left|\xi+\frac{c_{l}^2}{2}\right|t^{\frac{2}{3}}<C, \ \ C>0\right\}$,
        \item The second transition region $\mathcal{T}_{\textup{\uppercase\expandafter{\romannumeral2}}}:=\left\{(x,t): \left|\xi+\frac{c_{r}^2}{2}\right|t^{\frac{2}{3}}<C\right\}$,
    \end{itemize}
as illustrated in Figure~\ref{fig:cone}.

\subsection{Main results}\label{sec: main results}

Besides the boundary condition \eqref{Initial data}, we assume that the initial
data $q_0(x)$ satisfies the following conditions.
\begin{Assumption}\label{assumption on q_0}
    \hfill
   \begin{itemize}
       \item [\rm(a)] $x^m(q_0-c_l)|_{\mathbb{R}^-}\in L^1(\mathbb{R}^-)$, $x^m\left(q_0-c_r\right)|_{\mathbb{R}^+}\in L^1(\mathbb{R}^+)$, $m=0,\cdots,8$.
       \item [\rm(b)] $q_0\in\mathcal{C}^4(\mathbb{R})$, $\partial_x^n q_0\in L^\infty(\mathbb{R})$, $n=0,\cdots,4$.
   \end{itemize}
\end{Assumption}

\begin{figure}[H]
    \begin{center}
    \begin{tikzpicture}[node distance=2cm]
    %\filldraw[yellow!20,line width=2] (0,0)--(6,1)--(7,0);
    %\filldraw[purple!20,line width=2] (0,0)--(-6,1)--(-7,0);
    %\draw[red, dashed](0,0)--(6,1)node[above,black]{$x=6t$};
    %\draw[yellow!20, fill=yellow!20] (0,0)--(-5,2)--(-5,0)--(0, 0);
 %   \draw[green!20, fill=blue!20] (0,0 )--(-4,4)--(-5,4)--(-5,2)--(0,0);
   % \draw[blue!20, fill=green!20] (0,0 )--(4,4)--(-4,4);
  %  \draw[green!20, fill=green!20] (0,0 )--(4,4)--(4,0)--(0,0);
    \draw[-latex](-5.5,0)--(5,0)node[right]{$x$};
    \draw[-latex](0,0)--(0,4)node[above]{$t$};

    \draw[yellow!20, fill=yellow!20] (0,0)--(-5,1.7)--(-5,2.3)--(0, 0);
     \draw[red,dashed](0,0)--(-5,2)node[above,black]{$\xi=-\frac{c_{l}^2}{2}$};
      \draw[yellow!20, fill=yellow!20] (0,0)--(-4,3.7)--(-4,4.3)--(0, 0);
    \draw[red,dashed](0,0)--(-4,4)node[above,black]{$\xi=-\frac{c_{r}^2}{2}$};
    %\draw[red, dashed](0,0)--(-1,4)node[above,black]{$\xi=-\frac{c_{r}^2}{6}$};
    \node[below]{$0$};
    \coordinate (A) at (-5.3, 1.65);
	\fill (A) node[right] {$\mathcal{T}_{\textup{\uppercase\expandafter{\romannumeral1}}}$};
    \coordinate (B) at (-4.3, 3.6);
	\fill (B) node[right] {$\mathcal{T}_{\textup{\uppercase\expandafter{\romannumeral2}}}$};

	\coordinate (C) at (-4.4, 0.65);
	\fill (C) node[right] {$\mathcal{R}_{\textup{\uppercase\expandafter{\romannumeral1}}}$};

	\coordinate (D) at (-3.5, 2);
	\fill (D) node[right] {$\mathcal{R}_{\textup{\uppercase\expandafter{\romannumeral2}}}$};

    \coordinate (E) at (-1, 2.8);
	\fill (E) node[right] {$\mathcal{R}_{\textup{\uppercase\expandafter{\romannumeral3}}}$};

    \end{tikzpicture}
    %\flushleft{\footnotesize {\bf Figure $\ref{cone}$}  }
    \caption{The different asymptotic regions of the $(x,t)$-half plane.} \label{fig:cone}
    \end{center}
\end{figure}
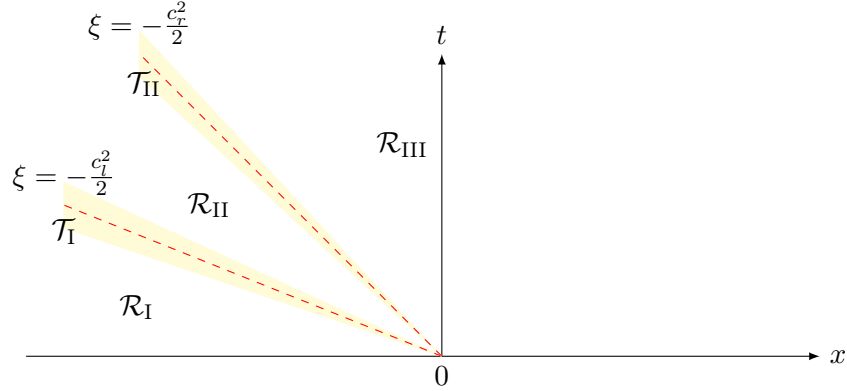

Our main results   in two regions  $\mathcal{T}_{\textup{\uppercase\expandafter{\romannumeral1}}}$ and $\mathcal{T}_{\textup{\uppercase\expandafter{\romannumeral2}}}$ 
are stated   as follows.
\begin{theorem}\label{thm:mainthm}
    Let $q(x,t)$ be the global solution of the Cauchy problem \eqref{equ:mkdv}--\eqref{Initial data}
    for the defocusing mKdV equation with a step-like  initial data   under Assumption \textup{\ref{assumption on q_0}},  and denote
    by $r(k)$ the reflection coefficient. As $t\rightarrow+\infty$, we have the following asymptotic expansions  for  $q(x,t)$
    in the transition regions.
    \begin{itemize}
    \item[{\rm (I)}] For $\xi\in\mathcal{T}_{\textup{\uppercase\expandafter{\romannumeral1}}}$, we have
    \begin{align}\label{asy formula:TI}
    q(x,t)=c_l+t^{-2/3}f_{\textup{\uppercase\expandafter{\romannumeral1}}}(\xi,s)+\mathcal{O}(t^{1/3-2\epsilon}),
    \end{align}
    where  $\epsilon$ is any real number with $\frac{1}{2}<\epsilon<\frac{2}{3}$ and
    \begin{align*}
        &s=-t^{\frac{2}{3}}\frac{g_{\textup{\uppercase\expandafter{\romannumeral1}}}^{(1)}(c_l)}{\left(\frac{3}{2}\right)^{\frac{1}{3}}\left|g_{\textup{\uppercase\expandafter{\romannumeral1}}}^{(2)}(c_l)\right|^{\frac{1}{3}}},\; f_\textup{\uppercase\expandafter{\romannumeral1}}(\xi,s)=-\frac{1}{2}\left(\frac{3}{2}\right)^{-\frac{2}{3}}\left|g_{\textup{\uppercase\expandafter{\romannumeral1}}}^{(2)}(c_l)\right|^{-\frac{2}{3}}\left(s^2\frac{\textup{Ai} '(s)}{\textup{Ai} (s)}-2s\right).
    \end{align*}
    Here, $\textup{Ai} (s)$ is the Airy function and
\begin{align*}
    & g_{\textup{\uppercase\expandafter{\romannumeral1}}}^{(1)}(c_l)=6(2c_l)^{\frac{1}{2}}(c_l^2+2\xi),\quad g_{\textup{\uppercase\expandafter{\romannumeral1}}}^{(2)}(c_l)=3(2c_l)^{-\frac{1}{2}}(c_l^2+2\xi)+4(2c_l)^{\frac{3}{2}}.
\end{align*}

    \item[{\rm (II)}] For $\xi\in\mathcal{T}_{\textup{\uppercase\expandafter{\romannumeral2}}}$, we have
    \begin{align}\label{asy formula:RII}
        q(x,t)=c_r+t^{-2/3}f_{\textup{\uppercase\expandafter{\romannumeral2}}}(\xi,s)+\mathcal{O}(t^{1/3-2\epsilon}),
    \end{align}
    where $\epsilon$ is any real number with $\frac{1}{2}<\epsilon<\frac{2}{3}$ and
    \begin{align*}
   &s=-t^{\frac{2}{3}}\frac{g_{\textup{\uppercase\expandafter{\romannumeral2}}}^{(1)}(c_r)}{\left(\frac{3}{2}\right)^{\frac{1}{3}}\left|g_{\textup{\uppercase\expandafter{\romannumeral2}}}^{(2)}(c_r)\right|^{\frac{1}{3}}},\\
   &f_{\textup{\uppercase\expandafter{\romannumeral2}}}(\xi,s)=2\left(\frac{3}{2}\right)^{-\frac{2}{3}}\left|g_{\textup{\uppercase\expandafter{\romannumeral2}}}^{(2)}(c_r)\right|^{-\frac{2}{3}}\left(2^{\frac{1}{3}}a(s)q(-2^{\frac{1}{3}}s)+u(s)+\frac{s}{2}\right),
    \end{align*}
    with
    \begin{align*}
         &g_{\textup{\uppercase\expandafter{\romannumeral2}}}^{(1)}(c_r)=6(2c_r)^{\frac{1}{2}}(c_r^2+2\xi),\quad g_{\textup{\uppercase\expandafter{\romannumeral2}}}^{(2)}(c_r)=3(2c_r)^{-\frac{1}{2}}(c_r^2+2\xi)+4(2c_r)^{\frac{3}{2}}.
    \end{align*}
Moreover, $q(s)$ satisfies the Painlev\'e \textup{\uppercase\expandafter{\romannumeral2}} equation
\begin{equation*}
     q''(s)=sq(s)+2q^3(s)+\frac{\arg r(c_r)}{\pi}-\frac{1}{2},
\end{equation*}
with asymptotics
\begin{equation*}
    q(s)=\begin{cases}
        -\left(\frac{\arg r(c_r)}{\pi}-\frac{1}{2}\right) s^{-1}    +\bigo{s^{-4}}, &\qquad s\to +\infty,
\\
\sqrt{-\frac{s}{2}}+\bigo{s^{-1}}, &\qquad s\to -\infty,
    \end{cases}
\end{equation*}
and $u(s)$ satisfies the Painlev\'e \textup{\uppercase\expandafter{\romannumeral34}} equation
\begin{equation*}
     u''(s)=4u(s)^2+2su(s)+\frac{u'(s)^2-\frac{\arg r(c_r)^2}{\pi^2}}{2u(s)},
\end{equation*}
with asymptotics
\begin{equation*}
u(s)=\left\{ \begin{array}{ll}
-\frac{\arg r(c_r)}{2\pi\sqrt{s}}+\bigo{s^{-2}}, &\qquad s\to +\infty,
\\
-\frac{s}{2}+\bigo{s^{-2}}, &\qquad s\to -\infty.
\end{array}\right .
\end{equation*}
\end{itemize}
\end{theorem}

%We clarify that no soliton structures emerge under the considered step-like initial data satisfying Assumption \ref{assumption on q_0}; see part (b) of Proposition \ref{prop:a,b,r} below.

Establishing transition asymptotics for integrable systems is usually a challenging task, which typically involve the nonlinear special functions -- Painlev\'{e} transcendents and exhibit some universal features. For instance, one encounters the Painlev\'e II transcendents and its  higher order analogues in \cite{Charlier2020,CG10,DZAnn,HM1981,hz,mis,wang2023defocusing,xu2024transient}, the Painlev\'e I transcendents and its  higher order analogues in \cite{BT,CG2009}, and a model RH problem associated with the Painlev\'{e} IV equation in \cite{MLS2025}.
In the present work, we shows that the analysis in both transition regions involves RH problems relevant to the Painlev\'e XXXIV transcendents, which is different from the classical local parametrices used in the analysis in the other regions \cite{xuzyd}.
 To be precise, the Painlev\'e \textup{\uppercase\expandafter{\romannumeral34}} equation, which reads as
\begin{equation*}
     u''(s)=4u(s)^2+2su(s)+\frac{u'(s)^2-(2b)^2}{2u(s)},
\end{equation*}
depends on a parameter $b$. This equation can be obtained from the well-known Painlev\'e \textup{\uppercase\expandafter{\romannumeral2}} equation
\begin{equation*}
     q''(s)=sq(s)+2q^3-\nu,\quad \nu=2b+\frac{1}{2},
\end{equation*}
and its Hamiltonian \cite{ikj2008, FN1980}.
 It is worthwhile to point out that the transition asymptotics of the focusing mKdV equation for step-like initial data
 $$q_0(x) \to \begin{cases}
     0, \quad \text{as}\; x\to +\infty,\\
     c, \quad \text{as}\; x \to -\infty, \;c>0,
 \end{cases} $$
 is related to the RH problem built from Laguerre polynomials \cite{BM2018}, and the Painlev\'e \uppercase\expandafter{\romannumeral34\relax} transcendents play an important role in asymptotic studies of critical behaviors arsing from integrable differential equations \cite{FLYZ2006, algas}, random unitary ensembles \cite{ikj2008}  and orthogonal polynomials \cite{XZ2011}.
However, the results in the two transition regions appear to be different; see Theorem \ref{thm:mainthm}. In fact, both results are derived from the same   Painlev\'e \uppercase\expandafter{\romannumeral34\relax} model, which is used (as shown in Appendix \ref{p34}) to construct the local RH problem within the nonlinear steepest descent analysis of the original mKdV RH problem. The apparent difference arises from the choice of coefficients; see Appendix \ref{p34} for further details.

\subsection{Arrangement and notations}\label{sec: main results}

The rest of this paper is organized as follows.
In Section \ref{sec:IST}, we recall the spectral analysis and inverse scattering transform, of which the details can be found in \cite{xuzyd}.
Sections \ref{sec:asymptotic analysis in RI} and \ref{sec:asymptotic analysis in RII} are devoted to the asymptotic analysis of the obtained
RH problem in different transition regions, from which the main results stated in Theorem \ref{thm:mainthm} above are proved.

Throughout this paper, the following notations will be used.
\begin{itemize}
    \item As usual, the classical Pauli matrices $\{\sigma_j\}_{j=1,2,3}$ are defined by
	\begin{equation*}
		\sigma_1:=\begin{pmatrix}0 & 1 \\ 1 & 0\end{pmatrix}, \quad
		\sigma_2:=\begin{pmatrix}0 & - i \\ i & 0\end{pmatrix}, \quad
		\sigma_3:=\begin{pmatrix}1 & 0 \\ 0 & -1\end{pmatrix}.
	\end{equation*}
    For a $2\times 2$ matrix $A$, we also define
	$$e^{\hat{\sigma}_j}A:=e^{\sigma_j}Ae^{-\sigma_j}, \quad j=1,2,3.$$
	\item For a complex-valued function $f(z)$, we use $f^{*}(z):=\overline{f(\bar{z})}$ for $z\in\mathbb{C}$ to denote its Schwartz conjugation.
    \item For a region $U\subseteq \mathbb{C}$, we use $U^*$ to denote the conjugated region of $U$. We also set
	\begin{equation*}
\mathbb{C}^{\pm}:=\left\{z
		\in\mathbb{C}: \pm \text{Im}\ z>0 \right\},\qquad \mathbb{R}^{\pm}:=\left\{z
		\in\mathbb{R}: \pm  z>0 \right\}.
	\end{equation*}
    \item For $1 \leqslant p < \infty$, the $L^p$-space is defined as:
                \begin{equation*}
                L^p(\mathbb{R}) = \left\{ f: \mathbb{R} \to \mathbb{C} \ \middle| \ f \text{ is measurable and } \|f\|_p < \infty \right\},
                \end{equation*}
                where the $L^p$-norm is given by $\|f\|_p = \left( \int_{\mathbb{R}} |f(x)|^p  dx \right)^{1/p}$.
                For $p = \infty$, the $L^\infty$-space is defined by
                \begin{equation*}
                L^\infty(\mathbb{R}) = \left\{ f: \mathbb{R} \to \mathbb{C} \ \middle| \ f \text{ is measurable and } \|f\|_\infty < \infty \right\},
                \end{equation*}
                where the essential supremum norm is given by $\|f\|_\infty = \inf \left\{ M \geqslant 0 \ | \ |f(x)| \leqslant M \text{ a.e.} \right\}.$
                The $\mathcal{C}^{N}(\Omega)$, $N=1,2,\cdots,\infty$ is defined as the space of $N$-times continuously differentiable functions on $\Omega$.
    \item
For any smooth oriented curve $\Sigma$, the Cauchy operator $\mathcal{C}$ on $\Sigma$ is defined  by
	\begin{align*}
		\mathcal{C}f(z)=\frac{1}{2\pi i}\int_{\Sigma}\frac{f(y)}{y-z}\dif y, \qquad  z\in\mathbb{C}\setminus \Sigma.
	\end{align*}
	Given a function $f \in L^p(\Sigma)$, $1\leqslant p<\infty$,
	\begin{align*}
		\mathcal{C}_\pm f(z):=\lim_{\substack{z'\to z\in\Sigma\\ z'\in\pm\text{ side of } \Sigma}}\frac{1}{2\pi i}\int_{\Sigma}\frac{f(y)}{y-z'}\dif y
	\end{align*}
   stands for the positive/negative (according to the orientation of $\Sigma$) non-tangential boundary value of $\mathcal{C}f$.

 %   For all smooth oriented curve $\Sigma$, the Cauchy operator $C$ on $\Sigma$ is defined  by
%	\begin{align*}
%		Cf(z)=\frac{1}{2\pi i}\int_{\Sigma}\frac{f(\zeta)}{\zeta-z}\dif\zeta, \qquad  z\in\mathbb{C}\setminus \Sigma.
%	\end{align*}
%	Given a function $f \in L^p(\Sigma)$, $1\leqslant p<\infty$,
%	\begin{align*}
%		C_\pm f(z):=\lim_{z'\to z\in\Sigma}\frac{1}{2\pi i}\int_{\Sigma}\frac{f(\zeta)}{\zeta-z'}\dif\zeta
%	\end{align*}
%   stands for the positive / negative (according to the orientation of $\Sigma$) non-tangential boundary value of $Cf$.
%   It is also used to adopting $f_{\pm}(z)$ to represent the non-tangential limits from the positive / negative side respectively,
%   i.e., $f_{\pm}(z)=\lim_{{\rm positive/negative \ side}\ni \zeta\to z}f(\zeta)$.

   \item If $A$ is a matrix, then $(A)_{ij}$ stands for its $(i,j)$-th entry, and $[A]_j$ represents the $j$-th column.
   %We use $A^{\rm H}$ to denote its conjugate transpose, which means $A^{\rm H}=\bar A^{\rm T}$.
%   \item We use $\mathbb{C}^+:=\{k| \im k>0\}$ and $\mathbb{C}^-:=\{k| \im k<0\}$ to represent the  upper and lower half plane respectively. Similarly, we introduce $\mathbb{R}^+:=(0,+\infty)$ and $\mathbb{R}^-:=(-\infty,0)$.
   \item
   We use the notation $a\lesssim b$ (resp.\ $a \gtrsim b$) to indicate that $a\leqslant Cb$ (resp.\ $a\geqslant Cb$) for some generic positive constant $C$.
   \item   In Sections \ref{sec:asymptotic analysis in RI} and \ref{sec:asymptotic analysis in RII}, we adopt the same notations (such as  $M^{(\infty)}$, $M^{(err)}$, $P^{(r)}$, \dots  ) during the analysis,  which should be understood in different contexts. We believe this will not lead to any confusion.

\end{itemize}

\section{Spectral analysis and inverse scattering transform}\label{sec:IST}

In this section, we state some main results on the inverse scattering transform associated with the Cauchy problem \eqref{equ:mkdv}--\eqref{Initial data}. The details can be found in \cite{xuzyd}.
\subsection{Direct scattering transform}
As the member of mKdV hierarchy \cite{AKNS}, the Lax pair of
the defocusing mKdV equation \eqref{equ:mkdv} is given by
\begin{align}\label{equ:lax pair}
    \left\{
        \begin{aligned}
        &\Phi_x+ik\sigma_3\Phi=Q\Phi, \\
        &\Phi_t+4ik^3\sigma_3\Phi=V\Phi,
        \end{aligned}
    \right.
\end{align}
where $\Phi=\Phi(k;x,t)$ is a $2\times 2$ matrix-valued function with
the spectral parameter $k\in\mathbb{C}$. Here, $Q$ and $V$
are some matrices associated with the potential function $q(x,t)$, defined by
\begin{align*}
    &Q=Q(x,t)=\begin{pmatrix}0 & q(x,t) \\ q(x,t) & 0\end{pmatrix}, \label{equ:Q(x,t)}\\
    &V=V(x,t)=4k^2Q(x,t)+2ik\sigma_3\left(Q_x\left(x,t\right)-Q^2\left(x,t\right)\right)+2Q^3(x,t)-Q_{xx}(x,t).
\end{align*}
%The compatible condition $\Phi_{xt} = \Phi_{tx}$ of the Lax pair \eqref{equ:lax pair} is equivalent to the mKdV equation \eqref{equ:mkdv}.

For $j\in\{l,r\}$, substituting $q(x,t)=c_{j}$ into \eqref{equ:lax pair}, we obtain
the corresponding ``background'' Lax pair
\begin{equation}\label{laxpair:background}
     \left\{
        \begin{aligned}   &\Phi^b_{j,x}+ik\sigma_3\Phi^b_{j}=Q^b_j\Phi^b_{j}, \\
        &\Phi^b_{j,t}+4ik^3\sigma_3\Phi^b_{j}=V^b_j\Phi^b_{j},
        \end{aligned}
    \right.
\end{equation}
    where
    \begin{equation}\label{equ:Q_j^b,V_j^b}
        Q^b_j=\begin{pmatrix}0 & c_{j} \\ c_j & 0\end{pmatrix},\quad V^b_j=\begin{pmatrix}
            -2ic_j^2k& 4c_j k^2+2c_j^3\\4c_j k^2+2c_j^3&2ic_j^2k
        \end{pmatrix}.
    \end{equation}

The explicit solutions to the ``background'' Lax pair \eqref{laxpair:background} are given by
\begin{equation}\label{equ:Phi_j^p}
    \Phi_{j}^{b}(k;x,t)=\Delta_{j}(k)e^{-i(X_{j}(k)x+\Omega_{j}(k)t)\sigma_3},\quad j\in\{l,r\},
\end{equation}
where
\begin{align}\label{def:X_j(k)}
    X_{j}(k)=\sqrt{k^2-c_{j}^2}, \quad \Omega_{j}(k)=2(2k^2+c_{j}^2)X_{j}(k),\quad j\in\{l,r\},
\end{align}
and
\begin{equation}\label{equ:Delta_j}
    \Delta_{j}(k)=\frac{1}{2}
    \left(
        \begin{array}{cc}
        \chi_{j}(k)+\chi_{j}^{-1}(k) & i\left(\chi_{j}(k)-\chi_{j}^{-1}(k)\right)\\
        -i\left(\chi_{j}(k)-\chi_{j}^{-1}(k)\right) & \chi_{j}(k)+\chi_{j}^{-1}(k)
        \end{array}
    \right), \quad j\in\{l,r\}.
\end{equation}
Here, the function $\chi_j$ is defined by
\begin{equation*}
    \chi_{j}: \mathbb{C}\backslash[-c_{j},c_{j}]\rightarrow\mathbb{C}, \quad \chi_{j}(k)=\left(\frac{k-c_{j}}{k+c_{j}}\right)^{\frac{1}{4}}, \quad j\in\{l,r\},
\end{equation*}
with the branch cut being chosen such that $\chi_j(k)=1+\mathcal{O}(k^{-1})$ as $k\rightarrow\infty$.

For $j\in\{l,r\}$, we denote $\Phi^{b}_j(k;x):=\Phi^{b}_{j}(k;x,0)$ and  $Q_0(x):=Q(x,0)$. To proceed, we consider the Lax pair \eqref{equ:lax pair}
for $t=0$ and define the Jost functions $\Phi_j(k;x):=\Phi_j(k;x,0)$, which satisfy the $x$-part of \eqref{equ:lax pair} and admit the following asymptotic conditions
\begin{align*}
    &\Phi_{l}(k;x)=\Phi_{l}^{b}(k;x)\left(I+o(1)\right), \quad k\in\mathbb{R}\setminus\{-c_l,c_l\},\ x\rightarrow-\infty, \\
    &\Phi_{r}(k;x)=\Phi_{r}^{b}(k;x)\left(I+o(1)\right), \quad  k\in\mathbb{R}\setminus\{-c_r,c_r\},\ x\rightarrow+\infty.
\end{align*}
By Lax pairs \eqref{equ:lax pair} and \eqref{laxpair:background}, for $j\in\{l,r\}$, the Jost functions $\Phi_j(k;x)$ are defined via the Volterra integral equations
\begin{align}\label{equ:VolterraforPhi}
    \Phi_{j}(k;x)=\Phi^{b}_{j}(k;x)+\int_{\infty_j}^{x}\Phi^b_j(k;x)(\Phi^b_j)^{-1}(k;y)\left(Q_0(y)-Q^b_j\right)\Phi_{j}(k;y)\dif y,
\end{align}
where $\infty_r:=+\infty$, $\infty_l:=-\infty$ and $Q^b_j$ is defined in \eqref{equ:Q_j^b,V_j^b}.

Some basic properties of the Jost functions $\Phi_{l}$ and $\Phi_{r}$ are summarized in the following proposition,
which is listed here for latter use \cite{xuzyd}.
\begin{proposition}\label{prop: properties of Phi(r,j)}
    Under the Assumption \textup{\ref{assumption on q_0}} on the initial data, the Jost functions $\Phi_{l}$ and $\Phi_{r}$ defined by \eqref{equ:VolterraforPhi} have the following properties for $j\in\{l,r\}$:
    \begin{itemize}
        \item[\rm (a)] For each $x\in\mathbb{R}$, we have\\
        $[\Phi_r (k;x)]_1$ is holomorphic for $k\in\mathbb{C}^{-}$ and has continuous extension to $\overline{\mathbb{C}^{-}}\backslash\{-c_{r}, c_{r}\}$,\\
        $[\Phi_r(k;x)]_2$ is holomorphic for $k\in\mathbb{C}^{+}$ and has continuous extension to $\overline{\mathbb{C}^{+}}\backslash\{-c_{r}, c_{r}\}$,\\
        $[\Phi_l(k;x)]_1$ is holomorphic for $k\in\mathbb{C}^{+}$ and has continuous extension to $\overline{\mathbb{C}^{+}}\backslash\{-c_{l}, c_{l}\}$,\\
        $[\Phi_l(k;x)]_2$ is holomorphic for $k\in\mathbb{C}^{-}$ and has continuous extension to $\overline{\mathbb{C}^{-}}\backslash\{-c_{l}, c_{l}\}$.
        \item[\rm(b)] $\det\Phi_j(k;x)=1$, $k\in\mathbb{R}\setminus[-c_j,c_j]$.
        \item[\rm(c)] As $k\rightarrow\infty$, we have
        \begin{align*}
                \left([\Phi_l(k;x)]_1, [\Phi_r(k;x)]_2\right)e^{ikx\sigma_3}=I+\mathcal{O}(k^{-1}), \quad k\in \mathbb{C}^{+},\\
                \left([\Phi_r(k;x)]_1, [\Phi_l(k;x)]_2\right)e^{ikx\sigma_3}=I+\mathcal{O}(k^{-1}), \quad k\in \mathbb{C}^{-}.
        \end{align*}
        \item[\rm(d)] $\Phi_j(k;x)$ admits the following symmetries
        \begin{equation*}
            \sigma_1\overline{\Phi_{j}(\overline{k};x)}\sigma_1=\Phi_{j}(k;x), \quad
            [\Phi_j(k;x)]_1=\sigma_1[\Phi_j(-k;x)]_2,\quad
            \overline{\Phi_j(-\bar{k};x)}=\Phi_j(k;x).
        \end{equation*}
        \item[\rm(e)]For $k\in (-c_j, c_j)$, we have
        \begin{equation*}
           [\Phi_j(k;x)]_1 = [\Phi_j(k;x)]_2.
        %\quad[\Phi_r]_1 (k;x)=- [\Phi_r]_2 (k;x),
        \end{equation*}
        %\item[(f)] As $k\rightarrow\pm c_j$, we have $\Phi_j(k;x)=\mathcal{O}((k\mp c_j)^{-\frac{1}{4}})$.
    \end{itemize}
\end{proposition}
%\begin{remark}
%    When $k$ lies on the branch cut $(-c_l, c_l)$, $[\Phi_l]_1(k;x)$ denotes the boundary value of $[\Phi_l]_1(x; \tilde{k})$ as $\tilde{k}$ approaches $k$ from $\mathbb{C}^+$, while $[\Phi_l]_2(k;x)$ denotes the boundary value as $\tilde{k}$ approaches $k$ from $\mathbb{C}^-$. The definition for $\Phi_r(k;x)$, $k\in(-c_r,c_r)$
 %   is analogous.
%\end{remark}

Since the matrices $\Phi_{l}(k;x)$ and $\Phi_{r}(k;x)$ are both solutions to the $x$-part of the Lax pair \eqref{equ:lax pair} for $k\in\mathbb{R}\setminus\{\pm c_l,\pm c_r\}$ and $(x,t)\in\mathbb{R}\times[0,+\infty)$, they must be linearly dependent.
Consequently, there exists a scattering matrix $S(k)$, independent of $x$, such that
\begin{equation*}
    S(k)=\Phi_{r}^{-1}(k;x)\Phi_{l}(k;x),\quad x\in\mathbb{R},\;k\in\mathbb{R}\setminus\{\pm c_l,\pm c_r\}.\end{equation*}
Due to the symmetries of the $x$-part in \eqref{equ:lax pair}, the scattering matrix $S(k)$
admits the following structure
\begin{equation*}
    S(k)=\begin{pmatrix}a(k) & b^*(k) \\ b(k) & a^*(k) \end{pmatrix},\quad k\in\mathbb{R}\setminus\{\pm c_l,\pm c_r\},
\end{equation*}
where the scattering coefficients $a(k)$ and $b(k)$ are given by
\begin{align*}
    &a(k)=\det\left([\Phi_{l}(k;x)]_1, [\Phi_{r}(k;x)]_2\right), \\
    &b(k)=\det\left([\Phi_{r}(k;x)]_1, [\Phi_{l}(k;x)]_1\right).
\end{align*}
To proceed, we define the reflection coefficient
\begin{equation*}
    r(k):=\frac{b(k)}{a(k)},\quad k\in\mathbb{R}\setminus\{\pm c_l,\pm c_r\}.
\end{equation*}

It can be seen that the spectral functions $a(k)$, $b(k)$ and the reflection coefficient $r(k)$ admit the following properties \cite{xuzyd}.
\begin{proposition}\label{prop:a,b,r}
    Spectral functions $a(k)$, $b(k)$ and reflection coefficient $r(k)$ satisfy the following properties:
\begin{itemize}
    \item[\rm (a)] Spectral function $a(k)$ is holomorphic for $k\in\mathbb{C}^{+}$, and it could be continuously
    extended up to the boundary $\mathbb{R}\setminus\{\pm c_l, \pm c_r\}$. As $k\to \pm c_j,$ for $j\in\{l,r\}$,
    we have $a(k)=\mathcal{O}((k\mp c_j)^{-1/4})$. For $k\in\mathbb{C}^+$, as $k\to\infty$,
    we have $a(k)=1+\mathcal{O}(k^{-1})$. The spectral function $b(k)$ is defined continuously for $k\in\mathbb{R}\setminus\{\pm c_l, \pm c_r\}$.
    In particular, $a(k)$ and $b(k)$ belong to $\mathcal{C}^{8}(\mathbb{R}\setminus\{\pm c_l, \pm c_r\})$.
    \item[\rm (b)] $a(k)$ has no zeros on the complex plane $\mathbb{C}$.
    \item[\rm (c)] In their domains of definition, we have
        \begin{align}\label{symmetry a,b,r}
            a(k)=\overline{a(-\bar{k})}, \quad b(k)=\overline{b(-\bar{k})}, \quad  r(k)=\overline{r(-\bar{k})}.
        \end{align}
    \item[\rm (d)] The functions $a(k)$, $b(k)$, and $r(k)$ obey the following jump relations:
    \begin{align*}
        &a_{+}(k)=a^*_{-}(k), \; b_{+}(k)=-b^*_{-}(k),\; r_{+}(k)=-r^*_{-}(k), && k\in (-c_{r}, c_{r}), \\
        &a_{+}(k)=b_{-}^{*}(k), \; b_{+}(k)=a^*_{-}(k), \; r_{+}(k)=r_{-}^{*}(k)^{-1},  && k\in (-c_l, -c_{r})\cup(c_{r}, c_l), \\
        &a_{+}(k)=a_{-}(k), \; b_{+}(k)=b_{-}(k), \; r_{+}(k)=r_{-}(k),&& k\in (-\infty,-c_l)\cup(c_l,+\infty).
    \end{align*}
    It can be readily seen that $\vert r(k) \vert=1$ for $k\in(-c_l, -c_{r})\cup(c_{r}, c_l)$
    and $|r(k)|<1$ for $k$ belongs to the other intervals on $\mathbb{R}$.
    \item[\rm (e)] The reflection coefficient $r(k)\in\mathcal{C}^{8}(\mathbb{R}\setminus\{\pm c_l,\pm c_r\})$ has the following asymptotics near the branch cut points:
    \begin{equation}\label{expansionrk}
        r(k)=\begin{cases}
            \sum_{m=0}^{7} b_{l, m}\left(k-c_l\right)^{m / 2}+o(\left(k-c_l\right)^{\frac{7}{2}}), &  k \searrow c_l, \\
            \sum_{m=0}^{7} i^m b_{l,m}\left(c_l-k\right)^{m / 2}+o(\left(c_l-k\right)^{\frac{7}{2}}), & k \nearrow c_l, \\
            \sum_{m=0}^{7} i^m b_{r,m}\left(k-c_r\right)^{m / 2}+o(\left(k-c_r\right)^{\frac{7}{2}}), &  k \searrow c_r, \\
            \sum_{m=0}^{7}(-1)^m b_{r,m}\left(c_r-k\right)^{m / 2}+o(\left(c_r-k\right)^{\frac{7}{2}}), &  k \nearrow c_r,\\
            \sum_{m=0}^{7} (-i)^m\overline{b_{ l,m}}\left(k+c_l\right)^{m / 2}+o(\left(k+c_l\right)^{\frac{7}{2}}), & k \searrow -c_l, \\
            \sum_{m=0}^{7}  \overline{b_{l,m}}\left(-c_l-k\right)^{m / 2}+o(\left(-c_l-k\right)^{\frac{7}{2}}) ,& k \nearrow -c_l, \\
            \sum_{m=0}^{7} (-1)^m \overline{b_{r,m}}\left(k+c_r\right)^{m / 2}+o(\left(k+c_l\right)^{\frac{7}{2}}), & k \searrow -c_r, \\
            \sum_{m=0}^{7} (-i)^m \overline{b_{r,m}}\left(-c_r-k\right)^{m / 2}+o(\left(-c_r-k\right)^{\frac{7}{2}}), &  k \nearrow -c_r,
            \end{cases}
    \end{equation}
    where the coefficients $b_{j,m} \in \mathbb{C}$ satisfies
\begin{equation*}
    \left|b_{j, 0}\right|=1, \quad b_{j,1} \neq 0, \quad \sum_{m=0}^n i^{n-m}(-i)^m b_{j,n-m}\overline{b_{j,m}}=0, \quad j\in\{l,r\}, \;n=1, \cdots, 7.
\end{equation*}
Moreover, $r(k)$ has the following asymptotics as $k\to\pm\infty$:
\begin{equation*}
    \partial_k^mr(k)=\mathcal{O}(k^{-5}),\quad m=0,\cdots,8.
\end{equation*}
\end{itemize}
\end{proposition}

\subsection{RH characterization of the mKdV equation}
In order to construct a basic RH problem for the long-time asymptotic analysis, it is required to operate
the time evolution of the scattering data. Assuming that the solution $q(x,t)$ of Cauchy problem
\eqref{equ:mkdv}--\eqref{Initial data} exists for $t\geqslant 0$, it is followed that
\begin{align}
    &\Phi_{l}(k;x,t)=\Phi_{l}^{b}(k;x,t)\left(I+o(1)\right), \quad k\in\mathbb{R},\ x\rightarrow-\infty, \label{equ:Jostsolwithtime-1}\\
    &\Phi_{r}(k;x,t)=\Phi_{r}^{b}(k;x,t)\left(I+o(1)\right), \quad   k\in\mathbb{R},\ x\rightarrow+\infty, \label{equ:Jostsolwithtime-2}
\end{align}
where $\Phi_{j}^{b}(k;x,t)$ for $j\in\{l,r\}$  is defined in \eqref{equ:Phi_j^p}.

Since $\Phi_l$ and $\Phi_r$ are defined as simultaneous solutions of the Lax pair \eqref{equ:lax pair}, they must be linearly dependent. Consequently, from \eqref{equ:Jostsolwithtime-1} and \eqref{equ:Jostsolwithtime-2}, we obtain the following relation:
\begin{equation*}
    \Phi_{l}(k;x,t)=\Phi_{r}(k;x,t)S(k;t), \quad k\in\mathbb{R}\setminus\{\pm c_{l}, \pm c_r\}.
\end{equation*}
On account of the Lax pair \eqref{equ:lax pair}, it is obtained that
\begin{align*}
 \frac{\partial S(k;t)}{\partial t}\equiv 0,
\end{align*}
which implies that $\partial_t a(k;t)=0$ and $\partial_t b(k;t)=0$. This claim shows that the scattering data $a$ and $b$ are constant with respect to time.

%Let us define the following matrix-valued function $M^{(0)}(k):=M(k;x,0)$ by
%\begin{equation}\label{def:M0(k)}
%   M^{(0)}(k):=M(k;x,0)=\left\{
%        \begin{aligned}
%        &\left(\frac{[\Phi_l(k;x,0)]_1}{a(k)}, [\Phi_r(k;x,0)]_2\right)e^{ikx\sigma_3}, \quad k\in \mathbb{C}^{+}, \\
%        &\left([\Phi_r(k;x,0)]_1, \frac{[\Phi_l(k;x,0)]_2}{a^{*}(k)}\right)e^{ikx\sigma_3}, \quad k\in \mathbb{C}^{-},
%        \end{aligned}
%        \right.
%\end{equation}
%The matrix-valued function $M^{(0)}$ satisfies the RH problem below.
%\begin{RHP}\label{RHP:basic 0RHP}
%    \hfill
%    \begin{itemize}
%        \item $M^{(0)}(k)$ is holomorphic for $k\in\mathbb{C}\backslash\mathbb{R}$.
%        \item $M^{(0)}(k)$ has continuous boundary values $M^{(0)}_{\pm}(k)$ on $\mathbb{R}$ with the jump condition
 %       \begin{align*}
%            M^{(0)}_{+}(k)=M^{(0)}_{-}(k)V_0(k),
%        \end{align*}
%        where
%        \begin{equation}\label{equ:jump V^{(0)}(k)}
%            V^{(0)}(k)=
 %               \begin{cases}
 %               \begin{pmatrix}1-rr^* & - r^*e^{-2ikx}\\ re^{2ikx} & 1\end{pmatrix}, &k\in(-\infty,-c_l)\cup(c_l,+\infty),  \\
 %               \begin{pmatrix}0 & - r^*_{-}e^{-2ikx}\\ r_{+}e^{2ikx} & 1\end{pmatrix}, &k\in (-c_l,-c_{r})\cup(c_{r},c_l),\\
 %               \begin{pmatrix}0 & -e^{-2ikx}\\ e^{2ikx} & 0\end{pmatrix},  &k\in (-c_{r}, c_{r}).
 %               \end{cases}
 %       \end{equation}

 %       \item As $k\rightarrow \infty$ in $\mathbb{C}\setminus\mathbb{R}$, we have $M(k)=I+\mathcal{O}(k^{-1})$.
%    \end{itemize}
%\end{RHP}

Due to the time dependence, we are motivated to construct a piecewise analytic matrix-valued function as follows:
\begin{equation}
  M(k)=  M(k;x,t):=\left\{
        \begin{aligned}
        &\left(\frac{[\Phi_l(k;x,t)]_1}{a(k)}, [\Phi_r(k;x,t)]_2\right)e^{it\theta(k)\sigma_3}, \quad k\in \mathbb{C}^{+}, \\
        &\left([\Phi_r(k;x,t)]_1, \frac{[\Phi_l(k;x,t)]_2}{a^{*}(k)}\right)e^{it\theta(k)\sigma_3}, \quad k\in \mathbb{C}^{-},
        \end{aligned}
        \right.
\end{equation}
where $\theta(k)=\theta(k;\xi):=4k^3+12k\xi$ and $\xi=x/(12t)$. The matrix-valued function $M(\cdot):\mathbb{C}\backslash\mathbb{R}\to GL(2,\mathbb{C})$ satisfies the RH problem below.
\begin{RHP}\label{RHP:basic RHP}
    \hfill
    \begin{itemize}
        \item $M(k)$ is holomorphic for $k\in\mathbb{C}\backslash\mathbb{R}$.
        \item $M(k)$ has continuous boundary values $M_{\pm}(k)$ on $\mathbb{R}$ with the jump condition
        \begin{align*}
            M_{+}(k)=M_{-}(k)V(k),
        \end{align*}
        where
        \begin{equation}\label{equ:jump V(k)}
            V(k)=
                \begin{cases}
                \begin{pmatrix}1-r(k)r^*(k) & - r^*(k)e^{-2it\theta(k)}\\ r(k)e^{2it\theta(k)} & 1\end{pmatrix}, &k\in(-\infty,-c_l)\cup(c_l,+\infty),  \\
                \begin{pmatrix}0 & - r^*_{-}(k)e^{-2it\theta(k)}\\ r_{+}(k)e^{2it\theta(k)} & 1\end{pmatrix}, &k\in (-c_l,-c_{r})\cup(c_{r},c_l),\\
                \begin{pmatrix}0 & -e^{-2it\theta(k)}\\ e^{2it\theta(k)} & 0\end{pmatrix},  &k\in (-c_{r}, c_{r}).
                \end{cases}
        \end{equation}

        \item As $k\rightarrow \infty$ in $\mathbb{C}\setminus\mathbb{R}$, we have $M(k)=I+\mathcal{O}(k^{-1})$.
        \item As $k\rightarrow\pm c_{l}$, we have $M(k)=\mathcal{O}(1)$.

        \item As $k\rightarrow\pm c_{r}$, we have
        \begin{align*}
            &M(k)=\mathcal{O}\begin{pmatrix}(k\mp c_{r})^{\frac{1}{4}} & (k\mp c_{r})^{-\frac{1}{4}} \\ (k\mp c_{r})^{\frac{1}{4}} & (k\mp c_{r})^{-\frac{1}{4}}\end{pmatrix}, \quad k\in\mathbb{C}^{+}, \\
            &M(k)=\mathcal{O}\begin{pmatrix}(k\mp c_{r})^{-\frac{1}{4}} & (k\mp c_{r})^{\frac{1}{4}} \\ (k\mp c_{r})^{-\frac{1}{4}} & (k\mp c_{r})^{\frac{1}{4}}\end{pmatrix}, \quad k\in\mathbb{C}^{-}.
        \end{align*}
    \end{itemize}
\end{RHP}                                                                      It follows from item (d) of Proposition \ref{prop: properties of Phi(r,j)} and
item (c) of Proposition \ref{prop:a,b,r} that the solution $M(k)$ to the RH problem \ref{RHP:basic RHP}
automatically satisfies the following symmetries:
\begin{align*}
M(k)=\sigma_1M^{*}(k)\sigma_1=\overline{M(-\bar{k})}=\sigma_1M(-k)\sigma_1.
\end{align*}
Furthermore, under Assumption \ref{assumption on q_0} for the initial data $q_0$,
it follows from the proof of \cite[Theorem 7]{lenellsmkdvfourier} that
a classical solution $q(x,t)$ of the mKdV equation \eqref{equ:mkdv} exists and
can be reconstructed via the following limit for all $(x,t)\in\mathbb{R}\times[0,+\infty)$:
\begin{equation}\label{equ:recovering formula}
    q(x,t)=2i\lim_{k\rightarrow\infty} \left(kM(k)\right)_{12}.
\end{equation}

%\subsubsection{Factorizations of jump matrix}

A crucial step of performing the nonlinear steepest descent analysis is to open lenses, which is aided by two well-known
factorizations of the jump matrix $V(k)$ defined in \eqref{equ:jump V(k)}. Denote $r_2(k)={r^*(k)}/{(1-r(k)r^*(k))}$ for
$k\in\mathbb{R}\setminus[-c_l,c_l]$.
By Proposition \ref{prop:a,b,r},  the factorizations
utilized throughout the context can be listed as follows.

For $k\in(-\infty,-c_l)\cup(c_l,+\infty)$,
\begin{align*}
    \begin{pmatrix} 1-rr^* & -r^*e^{-2it\theta} \\ re^{2it\theta} & 1 \end{pmatrix}&=\begin{pmatrix} 1 & -r^*e^{-2it\theta} \\ 0 & 1\end{pmatrix}\begin{pmatrix} 1 & 0 \\ re^{2it\theta} & 1\end{pmatrix}\label{factor11}\\
        &=\begin{pmatrix} 1 & 0 \\r_2^*e^{2it\theta} & 1 \end{pmatrix}\left(1-rr^*\right)^{\sigma_3}\begin{pmatrix} 1 & -r_2e^{-2it\theta}\\ 0 & 1\end{pmatrix}.
\end{align*}

For $k\in (-c_l,-c_{r})\cup(c_{r},c_l)$, %where $r_+=1/r_-^*$,
\begin{align*}
    \begin{pmatrix} 0 & -r_-^*e^{-2it\theta} \\ r_+e^{2it\theta} & 1 \end{pmatrix}&=\begin{pmatrix} 1 & -r_-^*e^{-2it\theta} \\ 0 & 1\end{pmatrix}\begin{pmatrix} 1 & 0 \\ r_+e^{2it\theta} & 1\end{pmatrix}\\
        &=\begin{pmatrix} 1 & 0 \\ r_{2,-}^*e^{2it\theta}& 1 \end{pmatrix}\begin{pmatrix} 0 & -r_-^*e^{-2it\theta} \\ r_{+}e^{2it\theta} & 0\end{pmatrix}\begin{pmatrix} 1 & -r_{2,+}e^{-2it\theta}\\ 0 & 1\end{pmatrix}.
\end{align*}

For $k\in (-c_{r}, c_{r})$, %where $r_{+}=-r_-^*$,
\begin{align*}
    \begin{pmatrix}0 & -e^{-2it\theta}\\ e^{2it\theta} & 0\end{pmatrix}&=\begin{pmatrix} 1 & -r_-^*e^{-2it\theta} \\ 0 & 1\end{pmatrix}\begin{pmatrix}0 & -e^{-2it\theta}\\ e^{2it\theta} & 0\end{pmatrix}\begin{pmatrix} 1 & 0 \\ re^{2it\theta} & 1\end{pmatrix}\\
    &=\begin{pmatrix} 1 & 0 \\ r_{2,-}^*e^{2it\theta} & 1 \end{pmatrix}\begin{pmatrix} 0 & -e^{-2it\theta} \\ e^{2it\theta} & 0\end{pmatrix}\begin{pmatrix} 1 & -r_{2,+}e^{-2it\theta}\\ 0 & 1\end{pmatrix}.
\end{align*}

\section{Asymptotic analysis of the RH problem for $M$ in $\mathcal{T}_{\textup{\uppercase\expandafter{\romannumeral1}}}$}\label{sec:asymptotic analysis in RI}
This section is devoted to the long-time asymptotic analysis
of the RH problem \ref{RHP:basic RHP} in the region $\mathcal{T}_{\textup{\uppercase\expandafter{\romannumeral1}}}$. It is assumed that $-C<t^{2/3}\left(\xi+\frac{c_l^2}{2}\right)<0$ throughout this section since the analysis on the other half region of $\mathcal{T}_{\textup{\uppercase\expandafter{\romannumeral1}}}$ is similar.

\subsection{First transformation: $M\rightarrow M^{(1)}$}
We begin with the introduction of an auxiliary function, so-called $g$-function \cite{gfunction, G2001,GT2002},  to control the exponentially growing off-diagonal factors in the jump matrix \eqref{equ:jump V(k)}.
% Let us start with the construction of the $g$-function, which is crucial to deform
%the RH problem \ref{RHP:basic RHP} into a form suitable for asymptotic analysis.
%\subsubsection{The $g$-function}
For $\xi\in\mathcal{T}_{\textup{\uppercase\expandafter{\romannumeral1}}}$, we introduce
\begin{equation}\label{equ:g function RI}
g_{\textup{\uppercase\expandafter{\romannumeral1}}}(k)=g_{\textup{\uppercase\expandafter{\romannumeral1}}}(k;\xi):=(4k^2+12\xi+2c_{l}^2)X_{l}(k)
\end{equation}
with $X_l(k)=\sqrt{k^2-c_l^2}$ given in \eqref{def:X_j(k)}.
The branch of the square root is chosen such that $X_l(k)=k+\mathcal{O}(k^{-1})$ as $k\rightarrow\infty$.
%It's readily verified that $g_{\textup{\uppercase\expandafter{\romannumeral1}}}$ defined in \eqref{equ:g function RI} satisfies the following properties:
\begin{proposition}The  function $g_{\textup{\uppercase\expandafter{\romannumeral1}}}$ defined in \eqref{equ:g function RI} satisfies the following properties:
    \begin{itemize}
        \item $g_{\textup{\uppercase\expandafter{\romannumeral1}}}(k)$ is holomorphic for $k\in\mathbb{C}\setminus[-c_l,c_l]$.
        \item As $k\rightarrow\infty$ in $\mathbb{C}\setminus[-c_l,c_l]$,
        we have $g_{\textup{\uppercase\expandafter{\romannumeral1}}}(k)=\theta(k)+\mathcal{O}(k^{-1})$.
        \item For $k\in(-c_l,c_l)$, $g_{\textup{\uppercase\expandafter{\romannumeral1}},+}(k)+g_{\textup{\uppercase\expandafter{\romannumeral1}}, -}(k)=0$.
        \item As $k\to\pm c_l$, we have
        \begin{equation}\label{g_I asy cl}
             g_{\textup{\uppercase\expandafter{\romannumeral1}}}(k)=6(\pm2c_l)^{\frac{1}{2}}(c_l^2+2\xi)(k\mp c_l)^{\frac{1}{2}}+\left[3(\pm2c_l)^{-\frac{1}{2}}(c_l^2+2\xi)+4(\pm2c_l)^{\frac{3}{2}}\right](k\mp c_l)^{\frac{3}{2}}+\mathcal{O}((k\mp c_l)^{\frac{5}{2}}).
        \end{equation}

    \end{itemize}
\end{proposition}
It is readily seen that the $k$-derivative of $g_{\textup{\uppercase\expandafter{\romannumeral1}}}$ is given by
\begin{equation}\label{g_1'}   g_{\textup{\uppercase\expandafter{\romannumeral1}}}'(k)=\frac{12k\left(k-\eta_l(\xi)\right)\left(k+\eta_l(\xi)\right)}{X_l(k)},
\end{equation}
where
\begin{align*}
   \eta_l(\xi)=\sqrt{-\xi+\frac{c_l^2}{2}}.
\end{align*}
Moreover, direct calculations based on the above definition show that for $\xi\in\mathcal{T}_{\textup{\uppercase\expandafter{\romannumeral1}}}$, two saddle points $\pm\eta_l$ tend to $c_l$ at a rate of at least $\mathcal{O}(t^{-\frac{2}{3}})$ as $t\to+\infty$.
The signature table for $\im g_{\textup{\uppercase\expandafter{\romannumeral1}}}$ is illustrated in Figure \ref{fig:signs img RI}.
\begin{figure}[H]
\begin{center}
    \tikzset{every picture/.style={line width=0.75pt}} %set default line width to 0.75pt
    \begin{tikzpicture}[x=0.75pt,y=0.75pt,yscale=-1,xscale=1]
    %Curve Lines [id:da11389389910090686]
    \draw    (413,61) .. controls (387,61) and (386,242) .. (419,242) ;
    %Straight Lines [id:da8850126286527125]
    \draw    (159,152) -- (475,152) ;
    %Curve Lines [id:da9567273527008515]
    \draw    (203,61) .. controls (228,62) and (228,242) .. (200,243) ;
    % Text Node
    \draw (200,155.4) node [anchor=north west][inner sep=0.75pt]  [font=\tiny]  {\footnotesize$-\eta_l$};
       \fill (221,152) circle (1.2pt);
    % Text Node
    \draw (395,156.4) node [anchor=north west][inner sep=0.75pt]  [font=\tiny]  {\footnotesize$\eta_l$};
        \fill (394,152) circle (1.2pt);
    % Text Node
    \draw (232,155.4) node [anchor=north west][inner sep=0.75pt]  [font=\tiny]  {\footnotesize$-c_{l}$};
    % Text Node
    \draw (366,156.4) node [anchor=north west][inner sep=0.75pt]  [font=\tiny]  {\footnotesize$c_{l}$};
    % Text Node
    \draw (255.18,155.4) node [anchor=north west][inner sep=0.75pt]  [font=\tiny,rotate=-1.56]  {\footnotesize$-c_{r}$};
    % Text Node
    \draw (344,156.4) node [anchor=north west][inner sep=0.75pt]  [font=\tiny]  {\footnotesize$c_{r}$};
    % Text Node
    \draw (441,102.4) node [anchor=north west][inner sep=0.75pt]    {$+$};
    % Text Node
    \draw (447,183.4) node [anchor=north west][inner sep=0.75pt]    {$-$};
    % Text Node
    \draw (294,90.4) node [anchor=north west][inner sep=0.75pt]    {$-$};
    % Text Node
    \draw (299,202.4) node [anchor=north west][inner sep=0.75pt]    {$+$};
    % Text Node
    \draw (170,185.4) node [anchor=north west][inner sep=0.75pt]    {$-$};
    % Text Node
    \draw (173,103.4) node [anchor=north west][inner sep=0.75pt]    {$+$};
    % Text Node

    \fill (237,152) circle (1.2pt);
    % Text Node
   % \draw (261,147.4) node [anchor=north west][inner sep=0.75pt]  [font=\tiny]  {$($};
   \fill (261,152) circle (1.2pt);
    % Text Node
    %\draw (369,146.4) node [anchor=north west][inner sep=0.75pt]  [font=\tiny]  {$)$};
    \fill (369,152) circle (1.2pt);
    % Text Node
    %\draw (347,147.4) node [anchor=north west][inner sep=0.75pt]  [font=\tiny]  {$)$};
    \fill (347,152) circle (1.2pt);
    \end{tikzpicture}
\caption{The signature table of $\im g_{\textup{\uppercase\expandafter{\romannumeral1}}}$ for $\xi\in\mathcal{T}_{\textup{\uppercase\expandafter{\romannumeral1}}}$,
where ``$+$'' and ``$-$'' denote $\im g_{\textup{\uppercase\expandafter{\romannumeral1}}} > 0$ and $\im g_{\textup{\uppercase\expandafter{\romannumeral1}}} < 0$ in the corresponding regions, respectively.}\label{fig:signs img RI}
\end{center}
\end{figure}
\subsubsection{RH problem for $M^{(1)}$}
By the function $g_{\textup{\uppercase\expandafter{\romannumeral1}}}$ defined in \eqref{equ:g function RI}, we introduce a new matrix-valued function $M^{(1)}$ by
\begin{equation}\label{def:M1 RI}
   M^{(1)}(k)=  M^{(1)}(k;x,t):=M(k)e^{it\left(g_{\textup{\uppercase\expandafter{\romannumeral1}}}(k)-\theta(k)\right)\sigma_3}.
\end{equation}
Then the RH problem for $M^{(1)}$ reads as follows:
\begin{RHP}\label{RHP:M1 RI}
\hfill
\begin{itemize}
    \item $M^{(1)}(k)$ is holomorphic for $k\in\mathbb{C}\setminus\mathbb{R}$.
    \item $M^{(1)}(k)$ has continuous boundary values $M_{\pm}^{(1)}(k)$ on $\mathbb{R}$ with the jump condition
    \begin{equation*}
        M^{(1)}_{+}(k)=M^{(1)}_{-}(k)V^{(1)}(k), \quad k\in\mathbb{R},
    \end{equation*}
    where
    \begin{equation}\label{equ:jump V1 RI}
        V^{(1)}(k)=
            \begin{cases}
            \begin{pmatrix}1-r(k)r^*(k) & -r^*(k)e^{-2itg_{\textup{\uppercase\expandafter{\romannumeral1}}}(k)}\\ r(k)e^{2itg_{\textup{\uppercase\expandafter{\romannumeral1}}}(k)} & 1\end{pmatrix}, &k\in(-\infty,-c_l)\cup(c_l,+\infty),  \\
            \begin{pmatrix}0 & -r_{-}^*(k)\\ r_{+}(k) & e^{-2itg_{\textup{\uppercase\expandafter{\romannumeral1}},+}(k)} \end{pmatrix}, &k\in (-c_l,-c_{r})\cup(c_{r},c_l),\\
            \begin{pmatrix}0 & -1 \\ 1 & 0\end{pmatrix},  &k\in (-c_{r}, c_{r}).
            \end{cases}
    \end{equation}
    \item As $k\rightarrow\infty$ in $\mathbb{C}\setminus\mathbb{R}$, we have $M^{(1)}(k)=I+\mathcal{O}(k^{-1})$.
    \item $M^{(1)}(k)$ admits the same singular behavior as $M(k)$ at branch points $\pm c_{l}$ and $\pm c_{r}$; see RH problem \textup{\ref{RHP:basic RHP}}.
\end{itemize}
\end{RHP}
\subsection{Second transformation: $M^{(1)}\rightarrow M^{(2)}$}
In this section, we introduce an auxiliary function
$D_{\textup{\uppercase\expandafter{\romannumeral1}}}$ to
pave the way for the subsequent contour deformation
along the rays $(-\infty, -\eta_{l})$ and $(\eta_{l}, +\infty)$. Moreover, we need to keep the segments $[-\eta_l,-c_l]\cup[c_l,\eta_l]$ on the line.
\subsubsection{The $D_{\textup{\uppercase\expandafter{\romannumeral1}}}$ function}
Define $D_{\textup{\uppercase\expandafter{\romannumeral1}}}: \mathbb{C}\setminus([-c_l,-c_r]\cup[c_r,c_l])   \times \mathcal{T}_{\textup{\uppercase\expandafter{\romannumeral1}}}\to\mathbb{C}$ by
\begin{equation}\label{equ:def D function RI}
\begin{aligned}
D_{\textup{\uppercase\expandafter{\romannumeral1}}}(k)=D_{\mathrm{I}}(k;\xi):=\exp\bigg\{\frac{X_{l}(k)}{2\pi i}
 \left(\int_{-c_l}^{-c_{r}}+\int_{c_{r}}^{c_l}\right)
\frac{\log r_{+}(z)}{X_{l+}(z)(z-k)}\dif z\bigg\}.
\end{aligned}
\end{equation}

The necessary properties of the function
$D_{\textup{\uppercase\expandafter{\romannumeral1}}}$ are given as follows.
\begin{proposition}\label{prop: D function RI}
The function $D_{\textup{\uppercase\expandafter{\romannumeral1}}}$ defined in \eqref{equ:def D function RI} satisfies the following properties for $\xi\in \mathcal{T}_{\textup{\uppercase\expandafter{\romannumeral1}}}$:
\begin{itemize}
    \item[\rm (a)]$D_{\textup{\uppercase\expandafter{\romannumeral1}}}(k)$ is holomorphic for $k\in \mathbb{C}\backslash([-c_l,-c_r]\cup[c_r,c_l])$.
    \item[\rm (b)]$D_{\textup{\uppercase\expandafter{\romannumeral1}}}(k)$ satisfies the following jump relations:
    \begin{equation*}
          \begin{aligned}
       & D_{\textup{\uppercase\expandafter{\romannumeral1}}, +}(k)D_{\textup{\uppercase\expandafter{\romannumeral1}}, -}(k)=r_{+}(k), &&k\in (-c_l,-c_{r})\cup(c_{r},c_l),\\
       & D_{\textup{\uppercase\expandafter{\romannumeral1}}, +}(k)D_{\textup{\uppercase\expandafter{\romannumeral1}}, -}(k)=1, &&k\in (-c_{r}, c_{r}).
        \end{aligned}
    \end{equation*}
    \item[\rm (c)]$D_{\textup{\uppercase\expandafter{\romannumeral1}}}(k)$ admits the symmetry: $D_{\textup{\uppercase\expandafter{\romannumeral1}}}(-k)=D_{\textup{\uppercase\expandafter{\romannumeral1}}}(k)^{-1}$ for $k\in\mathbb{C}\backslash([-c_l,-c_r]\cup[c_r,c_l])$.
    \item[\rm (d)]As $k\rightarrow \infty$, $D_{\textup{\uppercase\expandafter{\romannumeral1}}}(k)=1+\mathcal{O}(k^{-1})$ for all $\xi\in\mathcal{T}_{\textup{\uppercase\expandafter{\romannumeral1}}}$.
    \item[\rm (e)]At the branch points $\pm c_l$ and $\pm c_r$, we have:
    \begin{equation}\label{equ:singular behavior of D}
         \begin{aligned}
       & D_{\textup{\uppercase\expandafter{\romannumeral1}}}(k)=e^{\frac{i}{2}\arg b_{l,0}}\left(1+\mathcal{O}\left((c_l\mp k)^{1/2}\right)\right), &&k\rightarrow\pm c_{l},  \\
       & D_{\textup{\uppercase\expandafter{\romannumeral1}}}(k)=(\pm k- c_r)^{-\frac{\arg b_{r,0}}{2\pi}}e^{ d_{\textup{\uppercase\expandafter{\romannumeral1}},r}}\left(1+\mathcal{O}\left((\pm k- c_r)^{1/2}\right)\right), &&k\rightarrow\pm c_{r},
    \end{aligned}
    \end{equation}

where
\begin{align*}
  d_{\textup{\uppercase\expandafter{\romannumeral1}},r}&=\frac{\sqrt{c_l^2-c_r^2}}{2\pi}\int_{- c_{l}}^{- c_r}\frac{\log r_{+}(z)}{X_{l+}(z)(z- c_r)}\dif z+\frac{\arg b_{r,0}}{2\pi }\log (c_r-c_l)
\end{align*}
  with the principal branch being chosen for the complex power functions as well as the logarithms.
   \item[\rm (f)] $D_{\textup{\uppercase\expandafter{\romannumeral1}}}$ is analytic at $\pm\eta_l$, and
$D_{\textup{\uppercase\expandafter{\romannumeral1}}}\left(\eta_l\right)=e^{id_{\textup{\uppercase\expandafter{\romannumeral1}},\eta}},D_{\textup{\uppercase\expandafter{\romannumeral1}}}\left(-\eta_l\right)=e^{-id_{\textup{\uppercase\expandafter{\romannumeral1}},\eta}}$ with
    \begin{align*}
        d_{\textup{\uppercase\expandafter{\romannumeral1}},\eta}&=-\frac{\sqrt{\eta_l^2-c_l^2}}{2\pi }\left(\int_{-c_l}^{-c_{r}}+\int_{c_{r}}^{c_l}\right)\frac{\log r_{+}(z)}{X_{l+}(z)(z-\eta_l)}\dif z,
    \end{align*}
   and the principal branch being used for the logarithms.
\end{itemize}
\end{proposition}
\begin{proof}
\hfill
\begin{itemize}
   \item[(a)] It follows from the definition of $D_{\textup{\uppercase\expandafter{\romannumeral1}}}$ by considering basic ideas
   of the Cauchy transformation.
   \item[(b)] The jump condition is an immediate consequence by
   using the well-known Sokhotski-Plemelj formula.
   \item[(c)] It can be straightforward verified.
   \item[(d)] With the aid of the fact $|r(k)|=1$ for $k\in (-c_l,-c_{r})\cup(c_{r},c_l)$,
   the asymptotics of $D_{\textup{\uppercase\expandafter{\romannumeral1}}}(k)$ at
   infinity denoted by $D_{\textup{\uppercase\expandafter{\romannumeral1}}, \infty}(\xi)$
   is given by
    \begin{equation*}
        D_{\textup{\uppercase\expandafter{\romannumeral1}}, \infty}(\xi)=\exp\left\{-\frac{1}{2\pi i}\left(\int_{-c_l}^{-c_{r}}+\int_{c_{r}}^{c_l}\right)\frac{i\arg r_{+}(z)}{X_{l+}(z)}\,\dif z\right\}.
        \end{equation*}
        By the symmetry relation \eqref{symmetry a,b,r} for $r(k)$ and the fact that $X_{l}(-k)=-X_{l}(k)$ for $k\in(c_l,\infty)$,
        it can be directly verified that $D_{\textup{\uppercase\expandafter{\romannumeral1}}, \infty}(\xi)=1$
        for all $\xi\in\mathcal{T}_{\textup{\uppercase\expandafter{\romannumeral1}}}$.

        \item[(e)] From item (e) of Proposition \ref{prop:a,b,r}, we know that
        \begin{equation*}
            \log r(z)=\log b_{l,0}+i\frac{b_{l,1}}{b_{l,0}}\sqrt{c_l-z}+\mathcal{O}(c_l-z), \quad z\to c_l,
        \end{equation*}
        which  implies that
        \begin{align*}
            \int_{c_{r}}^{c_l}\frac{\log r_{+}(z)}{X_{l+}(z)(z-k)}\dif z=\frac{\pi i \log b_{l,0}}{\sqrt{2c_l}}(k-c_l)^{-\frac{1}{2}}+\mathcal{O}(1),\quad z\to c_l.
        \end{align*}
        Combined with the asymptotics of $X_l(k)$, the first asymptotics of \eqref{equ:singular behavior of D}  can be obtained immediately. The asymptotic expansion of $D_{\textup{\uppercase\expandafter{\romannumeral1}}}$ near $\pm c_l$ follows in a similar manner.
        \item[(f)] It follows from the definition of $D_{\textup{\uppercase\expandafter{\romannumeral1}}}$.
\end{itemize}
\end{proof}
\subsubsection{RH problem for $M^{(2)}$}
With the help of $D_{\textup{\uppercase\expandafter{\romannumeral1}}}$ function,
we define
\begin{equation}\label{def:M2 RI}
   M^{(2)}(k)= M^{(2)}(k;x,t):=M^{(1)}(k)D_{\textup{\uppercase\expandafter{\romannumeral1}}}^{-\sigma_3}(k).
\end{equation}
Then RH problem for $M^{(2)}$ reads as follows:
\begin{RHP}
\hfill
\begin{itemize}
    \item $M^{(2)}(k)$ is holomorphic for $k\in\mathbb{C}\setminus\mathbb{R}$.
    \item $M^{(2)}(k)$ has continuous boundary values $M_{\pm}^{(2)}(k)$ on $\mathbb{R}$ with the jump condition
    \begin{equation*}
        M^{(2)}_{+}(k)=M^{(2)}_{-}(k)V^{(2)}(k),
    \end{equation*}
    where
    \begin{equation*}
        V^{(2)}(k)=
            \begin{cases}
            \begin{pmatrix}1-r(k)r^*(k) & -D_{\textup{\uppercase\expandafter{\romannumeral1}}}^2(k)r^*(k)e^{-2itg_\textup{\uppercase\expandafter{\romannumeral1}}(k)}\\ D_{\textup{\uppercase\expandafter{\romannumeral1}}}^{-2}(k)r(k)e^{2itg_\textup{\uppercase\expandafter{\romannumeral1}}(k)} & 1\end{pmatrix}, &k\in(-\infty,-c_l)\cup(c_l,+\infty),  \\

            \begin{pmatrix}0 & -1\\ 1 & D_{\textup{ \uppercase\expandafter{\romannumeral1}}, +}(k)D_{\textup{\uppercase\expandafter{\romannumeral1}}, -}^{-1}(k)e^{-2itg_{\textup{\uppercase\expandafter{\romannumeral1}},+}(k)} \end{pmatrix}, &k\in (-c_l,-c_{r})\cup(c_{r},c_l),\\
            \begin{pmatrix}0 & -1 \\ 1 & 0\end{pmatrix},  &k\in (-c_{r}, c_{r}).
            \end{cases}
    \end{equation*}
    \item As $k\rightarrow\infty$ in $\mathbb{C}\setminus\mathbb{R}$, we have $M^{(2)}(k)=I+\mathcal{O}(k^{-1})$.
\end{itemize}
\end{RHP}
\subsection{Third transformation: $M^{(2)}\rightarrow M^{(3)}$}
The aim of the third transformation is to open lenses in
regions $U^{(3)}_j,U^{(3)*}_j,\;j=1,2$, which are
illustrated in Figure \ref{fig:U_j domain RI}.
To this end, we need to construct the analytic approximation
for the spectral functions $r(k)$.

\subsubsection{Analytic approximation}
The following two proposition establish the
analytic approximation for $r$ in different regions
as illustrated in Figure \ref{fig:U_j domain RI}. Their proofs
are analogous to those of Lemma 5.2 in \cite{LenellsDNLS}, respectively,
and thus are omitted here.
\begin{proposition}[Analytic approximation of $r$]\label{analytic extension of r RI}There exist continuous functions
\begin{equation*}
    r_a: \left(\overline{U_1^{(3)}}\cup\overline{U_2^{(3)}}\right) \times \mathcal{T}_\textup{\uppercase\expandafter{\romannumeral1}} \rightarrow \mathbb{C} \; \text { and } \; r_r: \left((-\infty,-\eta_l)\cup(\eta_l,+\infty)\right) \times \mathcal{T}_\textup{\uppercase\expandafter{\romannumeral1}}\rightarrow \mathbb{C},
\end{equation*}
which satisfy the following properties:
\begin{itemize}
    \item [\rm (a)] $r(k)=r_a(k)+r_r(k)$, where $r_a(k)=r_a(k;\xi)$ and $r_r(k)=r_r(k;\xi)$ for all $(k;\xi) \in \{(-\infty,-\eta_l)\cup(\eta_l,+\infty)\} \times \mathcal{T}_\textup{\uppercase\expandafter{\romannumeral1}}$.
    \item [\rm (b)] $r_a(-k)=r_a^*(k)$ for $k\in\overline{U_1^{(3)}}\cup\overline{U_2^{(3)}}.$
    \item [\rm (c)] For all $\xi \in \mathcal{T}_\textup{\uppercase\expandafter{\romannumeral1}}$, the function $r_a: U_1^{(3)}\cup U_3^{(3)} \rightarrow \mathbb{C}$ is holomorphic.
    Moreover, for $(k;\xi)\in \left(\overline{U_1^{(3)}}\cup\overline{U_2^{(3)}}\right) \times \mathcal{T}_\textup{\uppercase\expandafter{\romannumeral1}}$, we have
\begin{equation*}
    \left|r_a(k)-r\left(\pm \eta_l\right)\right| \lesssim \left|k\mp \eta_l\right| e^{\frac{t}{4}|\operatorname{Im} g_\textup{\uppercase\expandafter{\romannumeral1}}(k)|},
\end{equation*}
% for all $\xi\in\mathcal{R}_\textup{\uppercase\expandafter{\romannumeral1}}$ uniformly,
and
\begin{equation*}
    \left|r_a(k)\right| \lesssim \frac{e^{\frac{t}{4}|\operatorname{Im} g_\textup{\uppercase\expandafter{\romannumeral1}}(k)|}}{1+|k|^2}.
\end{equation*}
\item [\rm (d)] For all $\xi \in \mathcal{T}_\textup{\uppercase\expandafter{\romannumeral1}}$, the function $r_r \in L^p\left((-\infty,-\eta_l)\cup(\eta_l,+\infty)\right),\; p\in[1,+\infty]$, and as $t \rightarrow +\infty$,
\begin{equation*}
    \left\|r_r(k)\right\|_{L^p\left((-\infty,-\eta_l)\cup(\eta_l,+\infty)\right)} = \mathcal{O}\left(t^{-1}\right).
\end{equation*}
\end{itemize}
\end{proposition}

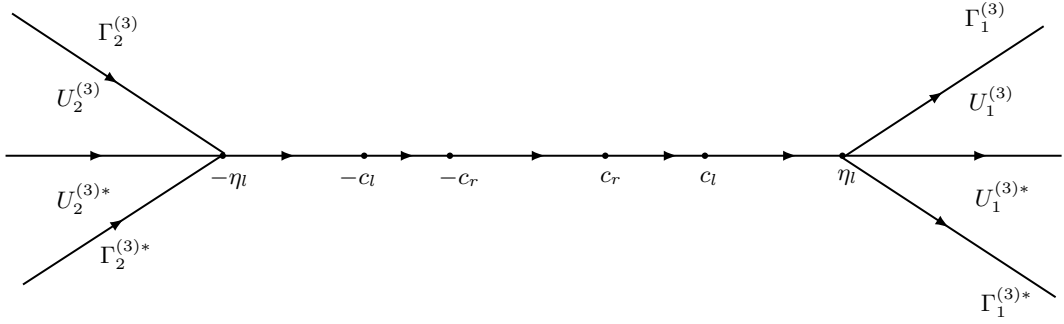
\begin{figure}
\begin{center}
\tikzset{every picture/.style={line width=0.75pt}}
\begin{tikzpicture}[x=0.75pt,y=0.75pt,yscale=-1,xscale=1]
%uncomment if require: \path (0,300);

% 实轴：三段合并为一条，在7个位置放置 latex 箭头
% 箭头位置由原始贝塞尔偏移量反推（总长度 61→591，共530单位）
\draw[
  postaction={decorate, decoration={markings,
    mark=at position 0.093 with {\arrow{latex}},   % x≈110，最左段中点
    mark=at position 0.274 with {\arrow{latex}},   % x≈206，-η_l 左侧
    mark=at position 0.387 with {\arrow{latex}},   % x≈266，-c_l 与 -c_r 之间
    mark=at position 0.511 with {\arrow{latex}},   % x≈332，-c_r 与 c_r 之间
    mark=at position 0.636 with {\arrow{latex}},   % x≈398，c_r 与 c_l 之间
    mark=at position 0.749 with {\arrow{latex}},   % x≈458，c_l 与 η_l 之间
    mark=at position 0.930 with {\arrow{latex}}    % x≈554，最右段中点
  }}
] (61,151) -- (591,151);

% 右下对角线：箭头方向与路径方向相同，从 (481,152) → (588,222)
\draw[
  postaction={decorate, decoration={markings,
    mark=at position 0.5 with {\arrow{latex}}
  }}
] (481,152) -- (588,222);

% 右上对角线：原路径 (582,86)--(481,152)，箭头指向 (582,86)
% 故将路径反转，以使 \arrow{latex} 方向正确
\draw[
  postaction={decorate, decoration={markings,
    mark=at position 0.5 with {\arrow{latex}}
  }}
] (481,152) -- (582,86);

% 左上对角线：原路径 (171,150)--(64.26,79.61)，箭头指向 (171,150)
% 故将路径反转，以使 \arrow{latex} 方向正确
\draw[
  postaction={decorate, decoration={markings,
    mark=at position 0.5 with {\arrow{latex}}
  }}
] (64.26,79.61) -- (171,150);

% 左下对角线：箭头方向与路径方向相同，从 (69.76,215.63) → (171,150)
\draw[
  postaction={decorate, decoration={markings,
    mark=at position 0.5 with {\arrow{latex}}
  }}
] (69.76,215.63) -- (171,150);

% 保留原始注释掉的线段
%\draw    (323,51) -- (481,152) ;
%\draw [shift={(407.06,104.73)}, rotate = 212.59] ...
%\draw    (481,152) -- (332,249) ;
%\draw    (171,150) -- (323,51) ;

% Text Node
\draw (477,156.4) node [anchor=north west][inner sep=0.75pt]  [font=\tiny]  {\footnotesize$\eta_l$};
\fill (481,151) circle (1.2pt);
% Text Node
\draw (162,156.4) node [anchor=north west][inner sep=0.75pt]  [font=\tiny]  {\footnotesize$-\eta_l$};
\fill (170,151) circle (1.2pt);
% Text Node
\draw (227,156.4) node [anchor=north west][inner sep=0.75pt]  [font=\tiny]  {\footnotesize$-c_{l}$};
\fill (241,151) circle (1.2pt);
% Text Node
\draw (408,157.4) node [anchor=north west][inner sep=0.75pt]  [font=\tiny]  {\footnotesize$c_{l}$};
\fill (412,151) circle (1.2pt);
\draw (358,157.4) node [anchor=north west][inner sep=0.75pt]  [font=\tiny]  {\footnotesize$c_{r}$};
\fill (362,151) circle (1.2pt);
\draw (277,156.4) node [anchor=north west][inner sep=0.75pt]  [font=\tiny]  {\footnotesize$-c_{r}$};
\fill (284,151) circle (1.2pt);
% Text Node
\draw (541,72.4) node [anchor=north west][inner sep=0.75pt]  [font=\tiny]  {\footnotesize$\Gamma _{1}^{( 3)}$};
% Text Node
\draw (549,215.4) node [anchor=north west][inner sep=0.75pt]  [font=\tiny]  {\footnotesize$\Gamma _{1}^{( 3) *}$};
% Text Node
\draw (106,79.4) node [anchor=north west][inner sep=0.75pt]  [font=\tiny]  {\footnotesize$\Gamma _{2}^{( 3)}$};
% Text Node
\draw (107,191.4) node [anchor=north west][inner sep=0.75pt]  [font=\tiny]  {\footnotesize$\Gamma _{2}^{( 3) *}$};
% Text Node
\draw (543,114.4) node [anchor=north west][inner sep=0.75pt]  [font=\tiny]  {\footnotesize$U_{1}^{( 3)}$};
% Text Node
\draw (546,162.4) node [anchor=north west][inner sep=0.75pt]  [font=\tiny]  {\footnotesize$U_{1}^{( 3) *}$};
% Text Node
%\draw (314,98.4) node [anchor=north west][inner sep=0.75pt]  [font=\tiny]  {$U_{2}^{( 3)}$};
% Text Node
\draw (85,112.4) node [anchor=north west][inner sep=0.75pt]  [font=\tiny]  {\footnotesize$U_{2}^{( 3)}$};
% Text Node
\draw (85,163.4) node [anchor=north west][inner sep=0.75pt]  [font=\tiny]  {\footnotesize$U_{2}^{( 3) *}$};
\end{tikzpicture}
\caption{The jump contours $ \Gamma^{(3)}$ and regions $U^{(3)}_j,U^{(3)*}_j,\;j=1,2$ of RH problem for $ M^{(3)}$ when $\xi\in\mathcal{T}_\textup{\uppercase\expandafter{\romannumeral1}}$.}\label{fig:U_j domain RI}
\end{center}
\end{figure}

\subsubsection{RH problem for $ M^{(3)}$}
Now we are ready to introduce a transformation
\begin{equation}\label{def:M3 RI}
     M^{(3)}(k)=M^{(3)}(k;x,t):=M^{(2)}(k)D_{\textup{\uppercase\expandafter{\romannumeral1}}}^{\sigma_3}(k)G(k)D_{\textup{\uppercase\expandafter{\romannumeral1}}}^{-\sigma_3}(k),
\end{equation}
where
\begin{equation*}
    G(k):=
        \begin{cases}
        \begin{pmatrix}1 & 0 \\ -r_a(k)e^{2itg_{\textup{\uppercase\expandafter{\romannumeral1}}}(k)} & 1\end{pmatrix}, &k\in U^{(3)}_1\cup U^{(3)}_2, \\
        \begin{pmatrix}1 & -r_a^{*}(k)e^{-2itg_\textup{\uppercase\expandafter{\romannumeral1}}(k)} \\ 0 & 1\end{pmatrix}, &k\in U_1^{(3)*}\cup U_2^{(3)*}, \\
        I,  &\textnormal{elsewhere}.
        \end{cases}
\end{equation*}
Here, the domains $U^{(3)}_j$ for $j=1,2$ are illustrated in Figure \ref{fig:U_j domain RI}.

Then RH problem for $ M^{(3)}$ reads as follows:
\begin{RHP}
\hfill
\begin{itemize}
    \item $M^{(3)}(k)$ is holomorphic for $k\in\mathbb{C}\setminus\Gamma^{(3)}$, where $ \Gamma^{(3)}:=\cup_{j=1}^{2}(\Gamma^{(3)}_j\cup\Gamma_j^{(3)*})\cup\mathbb{R}$; see Figure \textup{\ref{fig:U_j domain RI}} for an illustration.
    \item $M^{(3)}(k)$ has continuous boundary values $M_{\pm}^{(3)}(k)$ on $k\in \Gamma^{(3)}$ with the jump condition
    \begin{equation*}
        M^{(3)}_{+}(k)= M^{(3)}_{-}(k)V^{(3)}(k),
    \end{equation*}
    where
    \begin{equation}\label{equ:jump V3 RI}
        V^{(3)}(k)=
            \begin{cases}
            \begin{pmatrix} 1 & 0\\ D_{\textup{\uppercase\expandafter{\romannumeral1}}}^{-2}(k)r_a(k)e^{2itg_\textup{\uppercase\expandafter{\romannumeral1}}(k)} & 1\end{pmatrix}, &k\in\Gamma^{(3)}_1\cup\Gamma^{(3)}_2,  \\
            \begin{pmatrix} 1 & -D_{\textup{\uppercase\expandafter{\romannumeral1}}}^{2}(k)r_a^*(k)e^{-2itg_\textup{\uppercase\expandafter{\romannumeral1}}(k)}\\ 0 & 1\end{pmatrix}, & k\in\Gamma^{(3)*}_1\cup\Gamma^{(3)*}_2,  \\

            \begin{pmatrix}
                1-r_r(k)r^*_r(k)&-D_\textup{\uppercase\expandafter{\romannumeral1}}^2(k)r_r^*(k)e^{-2itg_\textup{\uppercase\expandafter{\romannumeral1}}(k)}\\D_\textup{\uppercase\expandafter{\romannumeral1}}^{-2}(k)r_r(k)e^{2itg_\textup{\uppercase\expandafter{\romannumeral1}}(k)}&1
            \end{pmatrix},&k\in(-\infty,-\eta_l)\cup(\eta_l,+\infty),\\
            \begin{pmatrix}1-r(k)r^*(k) & -D_{\textup{\uppercase\expandafter{\romannumeral1}}}^2(k)r^*(k)e^{-2itg_\textup{\uppercase\expandafter{\romannumeral1}}(k)}\\ D_{\textup{\uppercase\expandafter{\romannumeral1}}}^{-2}(k)r(k)e^{2itg_\textup{\uppercase\expandafter{\romannumeral1}}(k)} & 1\end{pmatrix},&k\in(-\eta_l,-c_l)\cup(c_l,\eta_l),\\
            \begin{pmatrix}
0&-1\\1&D_{\textup{\uppercase\expandafter{\romannumeral1}},+}(k)D_{\textup{\uppercase\expandafter{\romannumeral1}},-}^{-1}(k)e^{-2itg_{\textup{\uppercase\expandafter{\romannumeral1}},+}(k)}\\
            \end{pmatrix},&k\in(-c_l,-c_r)\cup(c_r,c_l),\\
            \begin{pmatrix}0 & -1 \\ 1 & 0\end{pmatrix},  & k\in (-c_r,c_r).
            \end{cases}
    \end{equation}
    \item As $k\rightarrow\infty$ in $\mathbb{C}\setminus\Gamma^{(3)}$, we have $ M^{(3)}(k)=I+\mathcal{O}(k^{-1})$.
   % \item As $k\rightarrow\pm c_l$, we have $ M^{(3)}(k)=\mathcal{O}\left((k\mp c_l)^{-1/4}\right)$.
\end{itemize}
\end{RHP}
\subsection{Analysis of RH problem for $M^{(3)}$}
It is readily seen that $V^{(3)} \to I$ exponentially
on the contours $\Gamma^{(3)}_j \cup \Gamma_j^{(3)*}$ for $j=1,2$ as $t \to +\infty$ except around the points $\pm c_l$.
Item (d) of Proposition \ref{analytic extension of r RI} ensures that $V^{(3)}$ in \eqref{equ:jump V3 RI} decays to the identity matrix
on the intervals $(-\infty, -\eta_l)$ and $(\eta_l, +\infty)$ as $t\to+\infty$.
Both of these two facts lead us to consider the following global parametrix.
\subsubsection{Global parametrix}
%As $t$ large enough, the jump matrix $V^{(3)}$ approaches
%\begin{equation}\label{equ:jump Vinfty}
%    V^{(\infty)}=\begin{pmatrix} 0 & -1 \\ 1 & 0 \end{pmatrix}, \quad k\in(-c_l,c_l).
%\end{equation}
%For $k\in\mathbb{C}\backslash[-c_l,c_l]$, $V^{(3)}\rightarrow I$ as $t\rightarrow \infty$.
%Then the following parametrix for $M^{(\infty)}$ is naturally established.
For $\xi \in \mathcal{T}_\textup{\uppercase\expandafter{\romannumeral1}}$,  the global parametrix $M^{(\infty)}(k)$  satisfies the following  RH problem.

\begin{RHP}\label{RHP:Minfty RI}
\hfill
\begin{itemize}
    \item $M^{(\infty)}(k)$ is holomorphic for $k\in\mathbb{C}\setminus[-c_l,c_l]$.
    \item $M^{(\infty)}(k)$ has continuous boundary values $M_{\pm}^{(\infty)}(k)$ on $(-c_l,c_l)$ satisfying the following jump condition:
    \begin{align*}
        M_{+}^{(\infty)}(k)=M_{-}^{(\infty)}(k)V^{(\infty)}(k),
    \end{align*}
    where
    \begin{equation*}
        V^{(\infty)}(k)=\begin{pmatrix}
            0 & -1\\
            1 & 0
        \end{pmatrix}.
    \end{equation*}
    \item As $k\rightarrow\infty$ in $\mathbb{C}\setminus[-c_l, c_l]$, we have $M^{(\infty)}(k)=I+\mathcal{O}(k^{-1})$.
    \item As $k\rightarrow \pm c_l$, $M^{(\infty)}(k)=\mathcal{O}((k\mp c_l)^{-1/4})$.
\end{itemize}
\end{RHP}
It follows from the direct calculation that the unique solution to $M^{(\infty)}(k)$ is given by
\begin{align}\label{equ:sol of Minfty}
    M^{(\infty)}(k)=\Delta_l(k),
\end{align}
where $\Delta_{l}(k)$ is defined by \eqref{equ:Delta_j} for $j=l$.

\subsubsection{Local parametrices near $\pm c_l$}\label{subsubsec: local para pm cl}
Let
\begin{align}\label{def local}
    D_\varrho(c_l)=\left\{k: |k-c_l|<\varrho \right\}, \quad D_\varrho(-c_l)=\left\{k: |k+c_l|<\varrho \right\}
\end{align}
be two small disks around $c_l$ and $-c_l$ respectively, where
\begin{equation*}
\varrho<{\rm min}\left\{2 |\eta_l-c_l|t^{\epsilon}, \frac{1}{3}|\eta_l+c_l|,  \frac{1}{3} |\eta_l|  \right\}, \quad \frac{1}{2}<\epsilon<\frac{2}{3}.
\end{equation*}
For $j\in\{r,l\}$, we intend to solve the following local RH problem for $M^{(j)}$.
\begin{RHP}
\hfill
\begin{itemize}
    \item $M^{(j)}(k)$ is holomorphic for $k\in\overline{D_{\varrho}(\pm c_l)}\setminus\Gamma^{(j)}$ \textup{(}``$+$'' for $j=r$, and ``$-$'' for $j=l$\textup{)},
    where
    \begin{align*}
        \Gamma^{(r)}:=D_\varrho(c_l)\cap \Gamma^{(3)},\quad \Gamma^{(l)}:=D_\varrho(-c_l)\cap \Gamma^{(3)}.
    \end{align*}
    %see Figure \ref{fig:local jumps RI} for an illustration.
    \item  $M^{(j)}(k)$ has continuous boundary values $M_{\pm}^{(j)}(k)$ on $k\in \Gamma^{(j)}$ with the jump condition
    \begin{align*}
        M^{(j)}_{+}(k)=M^{(j)}_{-}(k)V^{(3)}(k)\big|_{\Gamma^{(j)}},
    \end{align*}
    where $V^{(3)}(k)$ is given by \eqref{equ:jump V3 RI}.
    \item As $k\rightarrow\infty$ in $\mathbb{C}\setminus\Gamma^{(j)}$, we have $M^{(j)}(k)=I+\mathcal{O}(k^{-1})$.
\end{itemize}
\end{RHP}
The solution $M^{(j)}$ for this local RH problem can be
constructed by the Painlev\'e XXXIV parametrix shown in Appendix
\ref{p34} in a standard manner. In the rest part of this section, we focus on the construction of $M^{(r)}$ near $k=c_l$, and the construction of $M^{(l)}$ near $k=-c_l$ can be obtained in a similar way.

First, we introduce the following change of variables. Recall the asymptotics of $g_{\textup{\uppercase\expandafter{\romannumeral1}}}$ in \eqref{g_I asy cl} that
\begin{equation*}
     g_{\textup{\uppercase\expandafter{\romannumeral1}}}(k)=g_{\textup{\uppercase\expandafter{\romannumeral1}}}^{(1)}(c_l)(k- c_l)^{\frac{1}{2}}+g_{\textup{\uppercase\expandafter{\romannumeral1}}}^{(2)}(c_l)(k- c_l)^{\frac{3}{2}}+\mathcal{O}((k- c_l)^{\frac{5}{2}}),\quad k\to c_l,
\end{equation*}
where
\begin{equation}\label{equ:coefficients of gI}
    g_{\textup{\uppercase\expandafter{\romannumeral1}}}^{(1)}(c_l)=6(2c_l)^{\frac{1}{2}}(c_l^2+2\xi),\quad g_{\textup{\uppercase\expandafter{\romannumeral1}}}^{(2)}(c_l)=3(2c_l)^{-\frac{1}{2}}(c_l^2+2\xi)+4(2c_l)^{\frac{3}{2}}.
\end{equation}

For $(k;\xi)\in \left( D_{\varrho}(c_l)\setminus(-\infty,c_l] \right)\times \mathcal{T}_\textup{\uppercase\expandafter{\romannumeral1}}$, it can be readily seen that $g_{\textup{\uppercase\expandafter{\romannumeral1}}}^{(1)}(c_l)=\mathcal{O}(t^{-2/3})$, thus we have
\begin{equation}
    \mathrm{Im}  g_{\textup{\uppercase\expandafter{\romannumeral1}}}(k)\lesssim \pm t^{-2/3}\left|k- c_l\right|^{\frac{1}{2}},\quad k\in \left(D_{\varrho}(c_l)\setminus(-\infty,c_l]\right)\cap\mathbb{C}_{\pm}.
\end{equation}

For $k \in D_{\varrho}(c_l)\setminus(-\infty,c_l]$, we set
\begin{equation}\label{vira cha 1}
    \zeta_r(k) :=-\left(\frac{3}{2}\right)^{\frac{2}{3}}t^{\frac{2}{3}}\left[g_{\textup{\uppercase\expandafter{\romannumeral1}}}(k)-g_{\textup{\uppercase\expandafter{\romannumeral1}}}^{(1)}(c_l)(k-c_l)^{\frac{1}{2}}\right]^{\frac{2}{3}},
\end{equation}
where the cut $(\cdot)^{\frac{2}{3}}$ runs along $\mathbb{R}^{-}$. This is a one-to-one conformal mapping form $k$-plane to $\zeta_r$-plane and  $\zeta_r^\prime(c_l)=-\left(\frac{3}{2}\right)^{\frac{2}{3}}t^{\frac{2}{3}}\left|g_{\textup{\uppercase\expandafter{\romannumeral1}}}^{(2)}(c_l)\right|^{\frac{2}{3}}<0$.
Moreover, we have
\begin{equation}\label{asy zeta1/2}
    \zeta_r(k)^{\frac{1}{2}}=\begin{cases}
    -it^{\frac{1}{3}}(\frac{3}{2})^{\frac{1}{3}}\left|g_{\textup{\uppercase\expandafter{\romannumeral1}}}^{(2)}(c_l)\right|^{\frac{1}{3}}(k-c_l)^{\frac{1}{2}}\left(1+\mathcal{O}(k-c_l)\right),&k \to c_l,\ k \in \mathbb{C}^+,\\
    it^{\frac{1}{3}}(\frac{3}{2})^{\frac{1}{3}}\left|g_{\textup{\uppercase\expandafter{\romannumeral1}}}^{(2)}(c_l)\right|^{\frac{1}{3}}(k-c_l)^{\frac{1}{2}}\left(1+\mathcal{O}(k-c_l)\right),&k \to c_l,\ k \in \mathbb{C}^-,
\end{cases}
\end{equation}
as well as
\begin{equation}\label{asy zeta 3/2}
    \zeta_r(k)^{\frac{3}{2}}=\begin{cases}
    \frac{3i}{2}tg_{\textup{\uppercase\expandafter{\romannumeral1}}}^{(2)}(c_l)(k-c_l)^{\frac{3}{2}}\left(1+\mathcal{O}(k-c_l)\right),&k \to c_l,\ k \in \mathbb{C}^+,\\
    -\frac{3i}{2}tg_{\textup{\uppercase\expandafter{\romannumeral1}}}^{(2)}(c_l)(k-c_l)^{\frac{3}{2}}\left(1+\mathcal{O}(k-c_l)\right),&k \to c_l,\ k \in \mathbb{C}^-.
\end{cases}
\end{equation}
Then we define
\begin{equation*}
   S(k)= S(k;\xi):=\begin{cases}
        it\frac{g_{\textup{\uppercase\expandafter{\romannumeral1}}}^{(1)}(c_l)(k-c_l)^{\frac{1}{2}}}{\zeta_r(k)^{\frac{1}{2}}}, &k\in D_{\varrho}(c_l)\cap\mathbb{C}^+,\\
       - it\frac{g_{\textup{\uppercase\expandafter{\romannumeral1}}}^{(1)}(c_l)(k-c_l)^{\frac{1}{2}}}{\zeta_r(k)^{\frac{1}{2}}}, &k\in D_{\varrho}(c_l)\cap\mathbb{C}^-.
    \end{cases}
\end{equation*}
From \eqref{asy zeta1/2} we know that $S(k)$ is analytic in $D_{\varrho}(c_l)$ and define
\begin{equation}\label{defS}
    s:=S(c_l)=-t^{\frac{2}{3}}\frac{g_{\textup{\uppercase\expandafter{\romannumeral1}}}^{(1)}(c_l)}{\left(\frac{3}{2}\right)^{\frac{1}{3}}\left|g_{\textup{\uppercase\expandafter{\romannumeral1}}}^{(2)}(c_l)\right|^{\frac{1}{3}}}.
\end{equation}
It can be concluded from \eqref{asy zeta1/2}, \eqref{asy zeta 3/2} and \eqref{defS} that
\begin{equation*}
    \frac{4}{3}\zeta_r(k)^{\frac{3}{2}}+2S(k)\zeta_r(k)^{\frac{1}{2}}=\begin{cases}
        2itg_{\textup{\uppercase\expandafter{\romannumeral1}}}(k), &k\in D_{\varrho}(c_l)\cap\mathbb{C}^+,\\
        -2itg_{\textup{\uppercase\expandafter{\romannumeral1}}}(k), &k\in D_{\varrho}(c_l)\cap\mathbb{C}^-.
    \end{cases}
\end{equation*}
For $k \in D_{\varrho}(-c_l)\setminus[-c_l,+\infty)$, we set
\begin{equation}\label{vira cha 2}
    \zeta_l(k)=-\left(\frac{3}{2}\right)^{\frac{2}{3}}t^{\frac{2}{3}}\left[g_{\textup{\uppercase\expandafter{\romannumeral1}}}(k)-g_{\textup{\uppercase\expandafter{\romannumeral1}}}^{(1)}(c_l)(-k-c_l)^{\frac{1}{2}}\right]^{\frac{2}{3}},
\end{equation}
which is a one-to-one conformal mapping and $\zeta_l^\prime(-c_l)=\left(\frac{3}{2}\right)^{\frac{2}{3}}t^{\frac{2}{3}}\left|g_{\textup{\uppercase\expandafter{\romannumeral1}}}^{(2)}(c_l)\right|^{\frac{2}{3}}>0$.

To match $M^{(r)}$ near $c_l$, the change of variable from $k$-plane to $\zeta_r$-plane in \eqref{vira cha 1} inspires us to consider the problem in $\zeta_r$-plane which maps $c_l$ to the origin.
Let
    $\widehat D_{\varrho}^{(r)}:=\zeta_r(D_{\varrho}(c_l)),\
    \widehat\Gamma^{(r)}_1:=\zeta_r(\Gamma^{(3)*}_1)$, and  $\widehat \Gamma^{(r)}:=\zeta_r(\Gamma^{(r)})$.

Therefore, we can define by $\widehat M^{(r)}(\zeta_r):=M^{(r)}(k(\zeta_r))$ with the jump condition $\widehat M^{(r)}_+(\zeta_r)=\widehat M^{(r)}_-(\zeta_r)\widehat V^{(3)}(\zeta_r)$ on $\widehat\Gamma^{(r)}$, where
 \begin{equation}\label{equ:jump zetar}
      \widehat V^{(3)}(\zeta_r)=
            \begin{cases}
             \begin{pmatrix} 1 & D_{\textup{\uppercase\expandafter{\romannumeral1}}}^{2}(k(\zeta_r))r_a^*(k(\zeta_r))e^{-2itg_\textup{\uppercase\expandafter{\romannumeral1}}(k(\zeta_r))}\\ 0 & 1\end{pmatrix}, & \zeta_r\in\widehat\Gamma^{(r)}_1,  \\
            \begin{pmatrix} 1 & 0\\ -D_{\textup{\uppercase\expandafter{\romannumeral1}}}^{-2}(k(\zeta_r))r_a(k(\zeta_r))e^{2itg_\textup{\uppercase\expandafter{\romannumeral1}}(k(\zeta_r))} & 1\end{pmatrix}, &\zeta_r\in\widehat\Gamma^{(r)*}_1,  \\

            \begin{pmatrix}
               1&D_\textup{\uppercase\expandafter{\romannumeral1}}^2(k(\zeta_r))r_r^*(k(\zeta_r))e^{-2itg_\textup{\uppercase\expandafter{\romannumeral1}}(k(\zeta_r))}\\-D_\textup{\uppercase\expandafter{\romannumeral1}}^{-2}(k(\zeta_r))r_r(k(\zeta_r))e^{2itg_\textup{\uppercase\expandafter{\romannumeral1}}(k(\zeta_r))}& 1-r_r(k(\zeta_r))r^*_r(k(\zeta_r))
            \end{pmatrix},&\zeta_r\in(-\varrho,\zeta_r(\eta_l)),\\
            \begin{pmatrix}1 & D_{\textup{\uppercase\expandafter{\romannumeral1}}}^2(k(\zeta_r))r^*(k(\zeta_r))e^{-2itg_\textup{\uppercase\expandafter{\romannumeral1}}(k(\zeta_r))}\\ -D_{\textup{\uppercase\expandafter{\romannumeral1}}}^{-2}(k(\zeta_r))r(k(\zeta_r))e^{2itg_\textup{\uppercase\expandafter{\romannumeral1}}(k(\zeta_r))} & 1-r(k(\zeta_r))r^*(k(\zeta_r))\end{pmatrix},&\zeta_r\in(\zeta_r(\eta_l),0),\\
            \begin{pmatrix}
D_{\textup{\uppercase\expandafter{\romannumeral1}},+}(k(\zeta_r))D_{\textup{\uppercase\expandafter{\romannumeral1}},-}^{-1}(k(\zeta_r))e^{-2itg_{\textup{\uppercase\expandafter{\romannumeral1}},+}(k(\zeta_r))}&1\\-1&0
            \end{pmatrix},&\zeta_r\in(0,\zeta_r(c_r)),\\
            \begin{pmatrix}0 & 1 \\ -1 & 0\end{pmatrix},  & \zeta_r\in (\zeta_r(c_r),\varrho);
            \end{cases}
    \end{equation}
see Figure \ref{Fig14}.
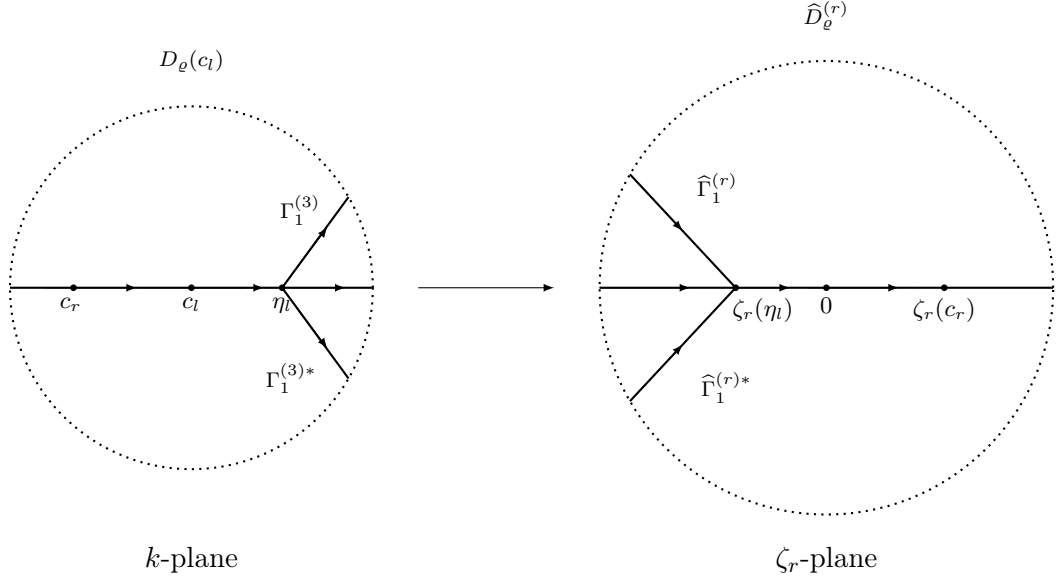
\begin{figure}
 \begin{center}
   \begin{tikzpicture}[scale=1.2]
    \draw[dotted,thick](-3,0) circle (2);

        \draw[](-3.6,0)--(-2.4,0);
        \draw[-latex](-4.6,0)--(-3.6,0);
      \draw[-latex](-3,0)--(-2.2,0);
        \draw[-latex](-2.2,0)--(-1.3,0);
        \node[shape=circle,fill=black, scale=0.13]  at (-3,0){0};
        \node[below] at (-3,0) {\footnotesize $c_l$};
        \node[shape=circle,fill=black, scale=0.13]  at (-4.3,0){0};
        \node[below] at (-4.33,0) {\footnotesize $c_r$};
        \node[shape=circle,fill=black, scale=0.13]  at (-2,0){0};
        \node[below] at (-2,0) {\footnotesize $\eta_l$};
        \draw[thick](-2,0)--(-1.268,1);
       %  \draw[thick](-2.4,0)--(-1,0.8);
        \draw [-latex] (-2,0)--(-1.5,0.67);
        \draw[thick](-2,0)--(-1.268,-1);
        \draw [-latex] (-2,0)--(-1.5,-0.67);
        \draw [thick] (-5, 0)--(-3.6, 0);
        \draw [thick] (-3.6, 0)--(-1, 0);

        \node  at (-1.8,0.86) {\scriptsize $\Gamma^{(3)}_{1}$};
\node  at (-1.9,-1) {\scriptsize $\Gamma^{(3)*}_{1}$};

        \node[]  at (-3,2.5) {\scriptsize $D_{\varrho}(c_l)$};
      %  \draw [dotted,-latex ]  (-1, 0)  to  [out=90,  in=0] (-3, 2);
    \draw[dotted,thick](4,0) circle (2.5);
        \draw[thick](1.8349,1.25)--(3,0);
        \draw [-latex] (1.8349,1.25)--(2.41745,0.625);
        \draw[thick](1.8349,-1.25)--(3,0);
        \draw [-latex] (1.8349,-1.25)--(2.41745,-0.625);
        \draw[thick ](3,0)--(5,0);
        \draw[-latex](3,0)--(3.6,0);
        \draw[-latex](1.5,0)--(2.5,0);
         \draw[-latex](3.8,0)--(4.8,0);
        \node[shape=circle,fill=black, scale=0.13]  at (3,0){0};
        \node[below] at (3.3,0) {\footnotesize $\zeta_r(\eta_l)$};
        \node[shape=circle,fill=black, scale=0.13]  at (4,0){0};
        \node[below] at (4,0) {\footnotesize $0$};
        \node[shape=circle,fill=black, scale=0.13]  at (5.3,0){0};
        \node[below] at (5.3,0) {\footnotesize $\zeta_r(c_r)$};

        \draw [thick] (1.5, 0)--(3, 0);
        \draw [thick] (5, 0)--(6.5, 0);

        \node  at (2.8,1.1) {\scriptsize $\widehat\Gamma^{(r)}_1$};
        \node  at (2.9,-1.1) {\scriptsize $\widehat\Gamma^{(r)*}_1$};

        \node[]  at (4,3) {\scriptsize $\widehat D_{\varrho}^{(r)}$};
        %\draw [dotted,-latex ]  (6.5, 0)  to  [out=90,  in=0] (4, 2.5);
    \node    at (-3, -3 )  {  $ k$-plane};
    \node    at (4, -3 )  {  $ \zeta_r$-plane};
 \draw [-latex] (-0.5,0)--(1,0);
    \end{tikzpicture}
    \caption{ {The   map  between  $D_{\varrho}(c_l)$  and $\widehat D_{\varrho}^{(r)}$.}  }
     \label{Fig14}
  \end{center}
\end{figure}

It follows from the definition of $D_{\varrho}(c_l)$ in \eqref{def local} that as $t\to\infty$, we have
\begin{equation*}
    k-c_l=\mathcal{O}(t^{-2/3+\varepsilon}),\quad k\in D_{\varrho}(c_l).
\end{equation*}
With respect to item (e) of Proposition \ref{prop: D function RI}, item (c) of Proposition \ref{analytic extension of r RI} and the asymptotics of $r(k)$ near $c_l$ in Proposition \ref{prop:a,b,r}, we have, for large $t$,
\begin{align*}
    D_{\textup{\uppercase\expandafter{\romannumeral1}}}(k)&=e^{\frac{i}{2}\arg b_{l,0}}+\mathcal{O}(t^{-1/3+\epsilon/2}),\\
    |r_a(k)-r(c_l)|&\lesssim|k-c_l|^{1/2}+|k-c_l|e^{\frac{t}{4}|\textup{Im} g_\textup{\uppercase\expandafter{\romannumeral1}}|}\\
    &\lesssim t^{-1/3+\epsilon/2}+t^{-2/3+\epsilon}e^{\frac{t}{4}|\textup{Im} g_\textup{\uppercase\expandafter{\romannumeral1}}|},\quad k\in D_{\varrho}(c_l)\cap\mathbb{C}^+,
\end{align*}
which implies that for $k\in D_{\varrho}(c_l)\cap\mathbb{C}^+$,
\begin{align}\label{asy1}
       | D_{\textup{\uppercase\expandafter{\romannumeral1}}}^{-2}(k)r_a(k)-D_{\textup{\uppercase\expandafter{\romannumeral1}}}^{-2}(c_l)r_a(c_l)|=|D_{\textup{\uppercase\expandafter{\romannumeral1}}}^{-2}(k)r_a(k)-1|&\lesssim
   t^{-1/3+\epsilon/2}+t^{-2/3+\epsilon}e^{\frac{t}{4}|\textup{Im} g_\textup{\uppercase\expandafter{\romannumeral1}}|}.
\end{align}

Therefore, RH problem $\widehat M^{(r)}(\zeta_r)$ can be approximated by the following RH problem $\widehat N^{(r)}(\zeta_r)$ which is associated with the Painlev\'e \textup{\uppercase\expandafter{\romannumeral34}} parametrix.
\begin{RHP}
\hfill
\begin{itemize}
    \item $\widehat N^{(r)}(\zeta_r)$ is holomorphic for $\zeta_r\in\mathbb{C}\setminus \left(\widehat\Gamma^{(r)}\setminus(-\infty,\zeta_r(\eta_l))\right)$.
    \item  $\widehat N^{(r)}(\zeta_r)$ has continuous boundary values $\widehat N^{(r)}_\pm(\zeta_r)$ on $\zeta_r\in \widehat\Gamma^{(r)}\setminus(-\infty,\zeta_r(\eta_l))$ with the jump condition
    \begin{align*}
        \widehat N^{(r)}_+(\zeta_r)=\widehat N^{(r)}_-(\zeta_r)\widehat V_N^{(r)}(\zeta_r),
    \end{align*}
    where
    \begin{equation}\label{equ:jumpVN}
      \widehat V_N^{(r)}(\zeta_r)=
            \begin{cases}
             \begin{pmatrix} 1 & e^{2\hat\theta(\zeta_r)}\\ 0 & 1\end{pmatrix}, & \zeta_r\in\widehat\Gamma^{(r)}_1,  \\
            \begin{pmatrix} 1 & 0\\ -e^{2\hat\theta(\zeta_r)} & 1\end{pmatrix}, &\zeta_r\in\widehat\Gamma^{(r)*}_1,  \\

           \begin{pmatrix} 1 & 0\\ -e^{2\hat\theta(\zeta_r)} & 1\end{pmatrix} \begin{pmatrix} 1 & e^{2\hat\theta(\zeta_r)}\\ 0 & 1\end{pmatrix} ,&\zeta_r\in(\zeta_r(\eta_l),0),\\

            \begin{pmatrix}0 & 1 \\ -1 & 0\end{pmatrix},  & \zeta_r\in (0,\varrho)
            \end{cases}
    \end{equation}
    with
    \begin{equation}\label{hattheta}
        \hat \theta(\zeta_r)=\frac{2}{3}\zeta_r^{\frac{3}{2}}+s\zeta_r^{\frac{1}{2}}.
    \end{equation}
    %\item As $\zeta_r\rightarrow\infty$ in $\mathbb{C}\setminus \left(\widehat\Gamma^{(3,r)}\setminus(-\infty,\zeta_r(\eta_l))\right)$, we have $\widehat N^{(r)}(\zeta_r)=I+\mathcal{O}(\zeta_r^{-1})$.
\end{itemize}
\end{RHP}
Next the following transformation is used to remove the jump contours $\widehat\Gamma^{(r)}_1$ and $\widehat\Gamma^{(r)*}_1$ from the starting point $\zeta_r(\eta_l)$ to $0$.
Define a matrix function $G^{(r)}(\zeta_r)$ on complex plane $\zeta_r$ by
\begin{equation*}
    G^{(r)}(\zeta_r)=\begin{cases}
        \begin{pmatrix} 1 & -e^{2\hat\theta(\zeta_r)}\\ 0 & 1\end{pmatrix}, & \zeta_r\in\Omega_1^{(r)},\\
         \begin{pmatrix} 1 & 0\\ -e^{2\hat\theta(\zeta_r)} & 1\end{pmatrix}, &\zeta_r\in\Omega_1^{(r)*},\\
         I,&\text{elsewhere}.
    \end{cases}
\end{equation*}
Therefore, we can construct a new transformation
\begin{equation}\label{trans N to S}
    \widehat S^{(r)}(\zeta_r)=\widehat N^{(r)}(\zeta_r) G^{(r)}(\zeta_r),
\end{equation}
which implies that $\widehat S^{(r)}(\zeta_r)$ satisfies the following RH problem.
\begin{RHP}
\hfill
\begin{itemize}
    \item $\widehat S^{(r)}(\zeta_r)$ is holomorphic for $\zeta_r\in\mathbb{C}\setminus \cup_{j=1,2,4}\Sigma_j$, where $\Sigma_j$,  $j=1,2,4$ are shown in Figures \textup{\ref{figp34}} and \textup{\ref{Fig17}}.
    \item  $\widehat S^{(r)}(\zeta_r)$ has continuous boundary values $\widehat S^{(r)}_\pm(\zeta_r)$ on $\zeta_r\in \cup_{j=1,2,4}\Sigma_j$ with the jump condition
    \begin{align*}
        \widehat S^{(r)}_+(\zeta_r)=\widehat S^{(r)}_-(\zeta_r)\widehat V_S^{(r)}(\zeta_r),
    \end{align*}
    where
    \begin{equation*}
      \widehat V_S^{(r)}(\zeta_r)=
            \begin{cases}
             \begin{pmatrix}0 & 1 \\ -1 & 0\end{pmatrix},  & \zeta_r\in \Sigma_1\\
             \begin{pmatrix} 1 & e^{2\hat\theta(\zeta_r)}\\ 0 & 1\end{pmatrix}, & \zeta_r\in\Sigma_2,  \\
          %  \begin{pmatrix} 1 & 0\\ 0 & 1\end{pmatrix}, &\zeta_r\in\Sigma_3,  \\

           \begin{pmatrix} 1 & 0\\ -e^{2\hat\theta(\zeta_r)} & 1\end{pmatrix},&\zeta_r\in\Sigma_4.
            \end{cases}
    \end{equation*}
\end{itemize}
\end{RHP}

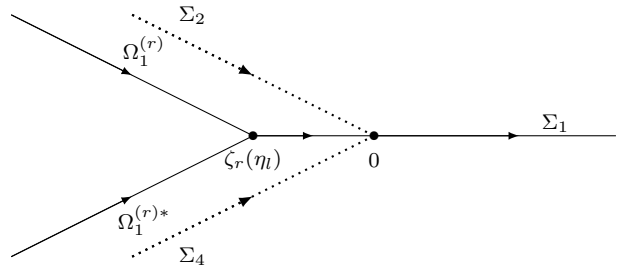
\begin{figure}[H]
\begin{center}
\begin{tikzpicture}[scale=1.6]

	%	\draw [dotted,-latex ](-6.8,0)--(-0.3,0);
			\draw [black ](-4.5,0)--(-1.5,0);
			\draw [black, -latex](-4.5,0)--(-4,0);
			\draw [black, -latex](-3.5,0)--(-2.3,0);
			\draw [ black, ](-4.5,0)--(-6.5,1);
			\draw [black,-latex ] (-6.5,1)--(-11/2,1/2);
			\draw [ black, ](-4.5,0)--(-6.5,-1);
			\draw [black,-latex ] (-6.5,-1)--(-11/2,-1/2);
		%	\draw [black,  ](-2.5,0)--(-0.5, 1 );
		%	\draw [black,-latex ] (-2.5,0)--(-1.5, 1/2);
		%	\draw[ black, ](-2.5,0)--(-0.5,-1);
		%	\draw [black,-latex ] (-2.5,0)--(-1.5,-1/2);

			\node[black]    at (-4.5,0)  {\scriptsize $\bullet$};
		
			\node    at (-5,1)  {\scriptsize $\Sigma_2$};
            \node    at (-5,-1)  {\scriptsize $\Sigma_4$};
             \node    at (-2,0.1)  {\scriptsize $\Sigma_1$};
			\node    at (-4.5,-0.2)  {\scriptsize $\zeta_r(\eta_l)$};

 \draw[dotted,thick,black](-5.5,-1)--(-3.5, 0 );
 \draw[dotted,thick, -latex,black  ](-5.5,-1)--(-4.5, -0.5 );
 %\draw[dotted,thick, -latex,black  ](-3.5,0)--(-2.5, 0.5 );
\draw[dotted,thick,black](-5.5,1)--(-3.5,0 );
  \draw[dotted,thick, -latex ,black ](-5.5, 1)--(-4.5, 0.5 );
   %\draw[dotted,thick, -latex,black   ](-3.5,0)--(-2.5, -0.5  );

 \node    at (-5.4, 0.7)  {\scriptsize $\Omega_1^{(r)}$};

  \node    at (-5.4,-0.7)  {\scriptsize $\Omega_1^{(r)*}$};

\node    at (-3.5,-0.2)  {\scriptsize $0$};
\node    at (-3.5,0)  {\scriptsize $\bullet$};
\end{tikzpicture}
\end{center}
\caption { Two auxiliary lines are added to the jump contour of $\widehat N^{(r)}(\zeta_r)$, which can be deformed into the Painlev\'e XXXIV model.}
 \label{Fig17}
\end{figure}
With the help of Painlev\'e XXXIV parametrix defined in Appendix \ref{p34}, the solution to the RH problem $\widehat S^{(r)}(\zeta_r)$ can be constructed by
\begin{equation}\label{equ:hat S trans}
    \widehat S^{(r)}(\zeta_r)=\begin{pmatrix}
		1 & 0\\
		i a(s) & 1
		\end{pmatrix}M^{P_{34}}(\zeta_r; 0, 0, s)e^{\hat\theta(\zeta_r)\sigma_3}Q^{(r)}(\zeta_r),
\end{equation}
where
\begin{equation}\label{def:Qr}
    Q^{(r)}(\zeta_r)=\begin{cases}
        \sigma_1,&\zeta_r\in\mathbb{C}^+,\\
        \sigma_3,&\zeta_r\in\mathbb{C}^-.
    \end{cases}
\end{equation}
Deduced from the asymptotics of $M^{P_{34}}(\zeta_r;  0,0,s)$ in \eqref{eq:Psi-infinity}, $\widehat S^{(r)}(\zeta_r)$ has the following asymptotics as $\zeta_r\to\infty$:
\begin{equation}\label{asy S}
    \widehat S^{(r)}(\zeta_r)=\left(I+\mathcal{O} \left( \zeta_r^{-1} \right) \right)
		\frac{\zeta_r^{-\frac{1}{4}\sigma_3}}{\sqrt{2}}
		\begin{pmatrix}
		1 & i
		\\
		i & 1
		\end{pmatrix}Q^{(r)}.
\end{equation}

Define $E^{(r)}(\zeta_r)=\widehat M^{(r)}(\zeta_r)\widehat N^{(r)}(\zeta_r)^{-1}$, then we can obtain the following RH problem for $E^{(r)}(\zeta_r)$.

\begin{RHP}
\hfill
\begin{itemize}
    \item $E^{(r)}(\zeta_r)$ is holomorphic for $k\in\mathbb{C}\setminus \widehat\Gamma^{(r)}$.
    \item  $E^{(r)}(\zeta_r)$ has continuous boundary values $E_\pm(\zeta_r)$ on $\widehat\Gamma^{(r)}$ with the jump condition
    \begin{align*}
        E^{(r)}_+(\zeta_r)=E^{(r)}_-(\zeta_r) V_E^{(r)}(\zeta_r),
    \end{align*}
    where
    \begin{equation*}
     V_E^{(r)}(\zeta_r)=\widehat N^{(r)}_-(\zeta_r)\widehat V^{(3)}(\zeta_r)\widehat V_N(\zeta_r)^{-1}\widehat N^{(r)}_-(\zeta_r)^{-1}.
    \end{equation*}

\end{itemize}
\end{RHP}
\begin{proposition}\label{pro:estimate E}
    For $\xi\in\mathcal{T}_{\textup{\uppercase\expandafter{\romannumeral1}}}$, as $t\to+\infty$, $E^{(r)}(\zeta_r)$ exists uniquely and satisfies
    \begin{equation}\label{esti for E}
\begin{aligned}
     \left\| E^{(r)}(\zeta_r)-I\right\|_{L^\infty\left(\widehat D^{(r)}_\varrho\right)}&=\mathcal{O}(t^{-\frac{1}{3}-\frac{\epsilon}{2}}),\\
        \left\| E^{(r)}(\zeta_r)-I\right\|_{L^1\left(\widehat D^{(r)}_\varrho\right)}&=\mathcal{O}(t^{-1-\frac{\epsilon}{2}}),\\
         \left\| E^{(r)}(\zeta_r)-I\right\|_{L^2\left(\widehat D^{(r)}_\varrho\right)}&=\mathcal{O}(t^{-\frac{2}{3}-\frac{\epsilon}{2}}).
\end{aligned}
    \end{equation}
\end{proposition}
\begin{proof}
From formulas of  $\widehat V^{(3)}(\zeta_r)$ in \eqref{equ:jump zetar} and $\widehat V_N^{(r)}(\zeta_r)$ in \eqref{equ:jumpVN}, it can be calculated that
    \begin{equation}
        \widehat V^{(3)}(\zeta_r)\widehat V_N^{(r)}(\zeta_r)^{-1}= \begin{cases}
             \begin{pmatrix} 1 & D_{\textup{\uppercase\expandafter{\romannumeral1}}}^{2}r_a^*e^{-2itg_\textup{\uppercase\expandafter{\romannumeral1}}}-e^{2\hat\theta}\\ 0 & 1\end{pmatrix}, & \zeta_r\in\widehat\Gamma^{(r)}_1,  \\
            \begin{pmatrix} 1 & 0\\ e^{2\hat\theta}-D_{\textup{\uppercase\expandafter{\romannumeral1}}}^{-2}r_ae^{2itg_\textup{\uppercase\expandafter{\romannumeral1}}} & 1\end{pmatrix}, &\zeta_r\in\widehat\Gamma^{(r)*}_1,  \\
          \begin{pmatrix}
               1&D_\textup{\uppercase\expandafter{\romannumeral1}}^2r_r^*e^{-2itg_\textup{\uppercase\expandafter{\romannumeral1}}}\\-D_\textup{\uppercase\expandafter{\romannumeral1}}^{-2}r_re^{2itg_\textup{\uppercase\expandafter{\romannumeral1}}}& 1-r_rr^*_r
            \end{pmatrix},&\zeta_r\in(-\varrho,\zeta_r(\eta_l)),\\

            \begin{pmatrix}1 & D_{\textup{\uppercase\expandafter{\romannumeral1}}}^2r^*e^{-2itg_\textup{\uppercase\expandafter{\romannumeral1}}}\\- D_{\textup{\uppercase\expandafter{\romannumeral1}}}^{-2}re^{2itg_\textup{\uppercase\expandafter{\romannumeral1}}} & 1-rr^*\end{pmatrix}\begin{pmatrix} 1 & -e^{2\hat\theta}\\ 0 & 1\end{pmatrix} \begin{pmatrix} 1 & 0\\ e^{2\hat\theta} & 1\end{pmatrix} ,&\zeta_r\in(\zeta_r(\eta_l),0),\\
            \begin{pmatrix}
1&D_{\textup{\uppercase\expandafter{\romannumeral1}},+}D_{\textup{\uppercase\expandafter{\romannumeral1}},-}^{-1}e^{-2itg_{\textup{\uppercase\expandafter{\romannumeral1}},+}}\\0&1
            \end{pmatrix},&\zeta_r\in(0,\zeta_r(c_r)),\\
          I,  & \text{elsewhere}.
            \end{cases}
    \end{equation}

For $\zeta_r\in\widehat\Gamma^{(r)}_1$, we have that
\begin{align*}
|D_{\textup{\uppercase\expandafter{\romannumeral1}}}^{2}r_a^*e^{-2itg_\textup{\uppercase\expandafter{\romannumeral1}}}-e^{2\hat\theta}|\leqslant\left|\left(D_{\textup{\uppercase\expandafter{\romannumeral1}}}^{2}r_a^*-1\right)e^{2\hat\theta}\right|+\left|D_{\textup{\uppercase\expandafter{\romannumeral1}}}^{2}r_a^*e^{2\hat\theta}\left(e^{2(S(k)-s)\zeta_r^{1/2}}-1\right)\right|.
\end{align*}
On one hand, from \eqref{asy1}, we have that
\begin{align*}
\left|\left(D_{\textup{\uppercase\expandafter{\romannumeral1}}}^{2}r_a^*-1\right)e^{2\hat\theta}\right|\lesssim|k-c_l|^{1/2}e^{-t^{1/3}|k-c_l|^{1/2}},
\end{align*}
uniformly for all $\zeta_r\in\widehat\Gamma^{(r)}_1$. Thus we can infer that for all $\xi\in\mathcal{T}_{\textup{\uppercase\expandafter{\romannumeral1}}}$,
\begin{equation}\label{equ:asy for1}
    \begin{aligned}
\left\|\left(D_{\textup{\uppercase\expandafter{\romannumeral1}}}^{2}r_a^*-1\right)e^{2\hat\theta}\right\|_{L^1\left(\widehat\Gamma^{(r)}_1\right)}&\lesssim\int_0^{\varrho}u^{1/2}e^{-t^{1/3}u^{1/2}}\dif u=\mathcal{O}(t^{-1}),\\
\left\|\left(D_{\textup{\uppercase\expandafter{\romannumeral1}}}^{2}r_a^*-1\right)e^{2\hat\theta}\right\|^2_{L^2\left(\widehat\Gamma^{(r)}_1\right)}&\lesssim\int_0^{\varrho}\left(u^{1/2}e^{-t^{1/3}u^{1/2}}\right)^2\dif u=\mathcal{O}(t^{-4/3}),\\
\left\|\left(D_{\textup{\uppercase\expandafter{\romannumeral1}}}^{2}r_a^*-1\right)e^{2\hat\theta}\right\|_{L^\infty\left(\widehat\Gamma^{(r)}_1\right)}&\lesssim\sup_{0\leqslant u\leqslant\varrho} u^{1/2}e^{-t^{1/3}u^{1/2}}=\mathcal{O}(t^{-1/3}).
\end{aligned}
\end{equation}
On the other hand, from the analyticity of $S(k)$ near $c_l$, we have
\begin{equation*}
    \left|e^{2(S(k)-s)\zeta_r^{1/2}}-1 \right|\lesssim|k-c_l|\zeta_r^{3/2},
\end{equation*}
which yields
\begin{equation}\label{equ:asy for2}
\begin{aligned}
     \left\|D_{\textup{\uppercase\expandafter{\romannumeral1}}}^{2}r_a^*e^{2\hat\theta}\left(e^{2(S(k)-s)\zeta_r^{1/2}}-1\right) \right\|_{L^\infty\left(\widehat\Gamma^{(r)}_1\right)}&=\mathcal{O}(t^{-2/3}),\\
      \left\|D_{\textup{\uppercase\expandafter{\romannumeral1}}}^{2}r_a^*e^{2\hat\theta}\left(e^{2(S(k)-s)\zeta_r^{1/2}}-1\right) \right\|_{L^1\left(\widehat\Gamma^{(r)}_1\right)}&=\mathcal{O}(t^{-4/3}),\\
       \left\|D_{\textup{\uppercase\expandafter{\romannumeral1}}}^{2}r_a^*e^{2\hat\theta}\left(e^{2(S(k)-s)\zeta_r^{1/2}}-1\right) \right\|^2_{L^2\left(\widehat\Gamma^{(r)}_1\right)}&=\mathcal{O}(t^{-2}).\\
\end{aligned}
\end{equation}
Recalling the asymptotics of $\widehat S^{(r)}$ from \eqref{asy S}, we have
\begin{align*}
    \widehat N^{(r)}(\zeta_r)=\mathcal{O}(\zeta_r^{-1/4}),\;\widehat N^{(r)}(\zeta_r)^{-1}=\mathcal{O}(\zeta_r^{-1/4}),\quad \zeta_r\to 0.
\end{align*}
Combining the results in \eqref{equ:asy for1} and \eqref{equ:asy for2} yields,  for $p=1,2,\infty$,
\begin{equation*}
    \begin{aligned}
       \left \|\widehat N^{(r)}_-(\zeta_r)\widehat V^{(3)}(\zeta_r)\widehat V_N^{(r)}(\zeta_r)^{-1} \widehat N^{(r)}_-(\zeta_r)^{-1}-I\right\|_{L^p\left(\widehat\Gamma^{(r)}_1\right)}&\lesssim\left\||\zeta_r|^{-1/2}\right\|_{L^p\left(\widehat\Gamma^{(r)}_1\right)}\left\|\widehat V^{(3)}(\zeta_r)\widehat V_N^{(r)}(\zeta_r)^{-1}- I\right\|_{L^p\left(\widehat\Gamma^{(r)}_1\right)}\\&\lesssim t^{-\epsilon/2}\left\|\widehat V^{(3)}(\zeta_r)\widehat V_N^{(r)}(\zeta_r)^{-1}- I\right\|_{L^p\left(\widehat\Gamma^{(r)}_1\right)}.
    \end{aligned}
\end{equation*}
A similar analysis applies to the estimates for $\zeta_r\in\widehat\Gamma^{(r)*}_1$.

For $\zeta_r\in(-\infty,\zeta_r(\eta_l))$, it can be obtained from item (d) of Proposition \ref{analytic extension of r RI} that
\begin{equation*}
    \begin{aligned}
       \left \|\widehat N^{(r)}_-(\zeta_r)\widehat V^{(3)}(\zeta_r)\widehat V_N^{(r)}(\zeta_r)^{-1}\widehat N^{(r)}_-(\zeta_r)^{-1}-I\right\|_{L^\infty\left(-\infty,\zeta_r(\eta_l)\right)}=\mathcal{O}(t^{-1-\epsilon/2}).
    \end{aligned}
\end{equation*}

For $\zeta_r\in(\zeta_r(\eta_l),0)$, the jump matrix $\widehat V^{(3)}(\zeta_r)\widehat V_N^{(r)}(\zeta_r)^{-1}$ can be decomposed as
\begin{equation*}
    \begin{aligned}
        \widehat V^{(3)}(\zeta_r)\widehat V_N^{(r)}(\zeta_r)^{-1}&= \begin{pmatrix} 1 & 0\\ -D_{\textup{\uppercase\expandafter{\romannumeral1}}}^{-2}re^{2itg_\textup{\uppercase\expandafter{\romannumeral1}}} & 1\end{pmatrix} \begin{pmatrix} 1 & D_{\textup{\uppercase\expandafter{\romannumeral1}}}^{2}r^*e^{-2itg_\textup{\uppercase\expandafter{\romannumeral1}}}\\ 0 & 1\end{pmatrix}\begin{pmatrix} 1 & -e^{2\hat\theta}\\ 0 & 1\end{pmatrix} \begin{pmatrix} 1 & 0\\ e^{2\hat\theta} & 1\end{pmatrix}\\
        &=\begin{pmatrix} 1 & 0\\ -D_{\textup{\uppercase\expandafter{\romannumeral1}}}^{-2}re^{2itg_\textup{\uppercase\expandafter{\romannumeral1}}} & 1\end{pmatrix} \begin{pmatrix} 1 & D_{\textup{\uppercase\expandafter{\romannumeral1}}}^{2}r^*e^{-2itg_\textup{\uppercase\expandafter{\romannumeral1}}}-e^{2\hat\theta}\\ 0 & 1\end{pmatrix} \begin{pmatrix} 1 & 0\\ e^{2\hat\theta} & 1\end{pmatrix}\\
        &=\begin{pmatrix} 1 & 0\\ e^{2\hat\theta}-D_{\textup{\uppercase\expandafter{\romannumeral1}}}^{-2}re^{2itg_\textup{\uppercase\expandafter{\romannumeral1}}} & 1\end{pmatrix}+\left(D_{\textup{\uppercase\expandafter{\romannumeral1}}}^{2}r^*e^{-2itg_\textup{\uppercase\expandafter{\romannumeral1}}}-e^{2\hat\theta}\right)\begin{pmatrix}e^{2\hat\theta}&1\\-D_{\textup{\uppercase\expandafter{\romannumeral1}}}^{-2}re^{2itg_\textup{\uppercase\expandafter{\romannumeral1}}+2\hat\theta}&-D_{\textup{\uppercase\expandafter{\romannumeral1}}}^{-2}re^{2itg_\textup{\uppercase\expandafter{\romannumeral1}}}
        \end{pmatrix}.
    \end{aligned}
\end{equation*}
Since $|D_{\textup{\uppercase\expandafter{\romannumeral1}}}^{-2}r-1|=\mathcal{O}(|k-c_l|^{1/2})=\mathcal{O}(t^{-1/3+\epsilon/2})$ and $|e^{2itg_\textup{\uppercase\expandafter{\romannumeral1}}}|=|e^{2\hat\theta}|=1$ for $\zeta_r\in(\zeta_r(\eta_l),0)$, a direct calculation yields
\begin{equation*}
    \begin{aligned}
       \left \|\widehat N^{(r)}_-(\zeta_r)\widehat V^{(3)}(\zeta_r)\widehat V_N^{(r)}(\zeta_r)^{-1}\widehat N^{(r)}_-(\zeta_r)^{-1}-I\right\|_{L^\infty\left((\zeta_r(\eta_l),0)\right)}=\mathcal{O}(t^{-1/3-\epsilon/2}).
    \end{aligned}
\end{equation*}

For $\zeta_r\in(0,\zeta_r(c_r))$, we have
\begin{equation*}    |D_{\textup{\uppercase\expandafter{\romannumeral1}},+}D_{\textup{\uppercase\expandafter{\romannumeral1}},-}^{-1}e^{-2itg_{\textup{\uppercase\expandafter{\romannumeral1}},+}}|\lesssim|k-c_l|^{1/2}e^{-t^{1/3}|k-c_l|^{1/2}},
\end{equation*}
and, by a calculation similar to that in \eqref{equ:asy for1}, it follows that
\begin{equation*}
    \begin{aligned}
       \left \|\widehat N^{(r)}_-(\zeta_r)\widehat V^{(3)}(\zeta_r)\widehat V_N^{(r)}(\zeta_r)^{-1}\widehat N^{(r)}_-(\zeta_r)^{-1}-I\right\|_{L^\infty\left((0,\zeta_r(c_r))\right)}=\mathcal{O}(t^{-1/3-\epsilon/2}).
    \end{aligned}
\end{equation*}
Therefore, it can be inferred from the small RH problem arguments that the solution $E^{(r)}(\zeta_r)$ exists uniquely and satisfies the estimate  \eqref{esti for E}.
\end{proof}

Next, we construct the local model $\widetilde M^{(r)}$ near $c_l$  with the help of RH problem $\widehat N^{(r)}$. Define
\begin{equation}\label{def:widetildeMr}
   \widetilde M^{(r)}(k) = \widetilde M^{(r)}(k;x,t):=P^{(r)}(k)\widehat N^{(r)}(\zeta_r(k)),\quad k\in D_{\varrho}(c_l),
\end{equation}
where
\begin{equation*}
   P^{(r)}(k)= P^{(r)}(k;x,t):=M^{(\infty)}(k)Q^{(r)}(\zeta_r(k))^{-1}\frac{1}{\sqrt{2}}(I-i\sigma_1)\zeta_r^{\frac{\sigma_3}{4}}(k).
\end{equation*}
Here, $M^{(\infty)}(k)$, $Q^{(r)}(\zeta_r(k))$ and $\zeta_r$ are defined in \eqref{equ:sol of Minfty}, \eqref{def:Qr} and \eqref{vira cha 1} respectively. Indeed, taking transformations \eqref{trans N to S} and \eqref{equ:hat S trans} into \eqref{def:widetildeMr} implies that
\begin{equation*}
    \widetilde M^{(r)}(k)=P^{(r)}(k)\begin{pmatrix}
		1 & 0\\
		i a(s) & 1
		\end{pmatrix}M^{P_{34}}(\zeta_r(k); 0, 0, s)e^{\hat\theta(\zeta_r(k))\sigma_3}Q^{(r)}(\zeta_r(k))G^{(r)}(\zeta_r(k))^{-1}.
\end{equation*}
We claim that   $P^{(r)}(k)$ is analytic in $D_{\varrho}(c_l)$. Indeed, for $k\in(c_l,c_l+\varrho)$, the jump of $\zeta_r^{1/4}$ implies that
\begin{equation*}
    \begin{aligned}
         P^{(r)}_+(k)&=M^{(\infty)}_+(k)Q^{(r)}(\zeta_r(k))_-^{-1}\frac{1}{\sqrt{2}}(I-i\sigma_1)\left(\zeta_r(k)\right)_-^{\frac{\sigma_3}{4}}\\
         &=M^{(\infty)}_-(k)\sigma_3\frac{1}{\sqrt{2}}(I-i\sigma_1)i^{-\sigma_3}\left(\zeta_r(k)\right)_+^{\frac{\sigma_3}{4}}\\
         &=M^{(\infty)}_-(k)Q^{(r)}(\zeta_r(k))_+^{-1}\frac{1}{\sqrt{2}}(I-i\sigma_1)\left(\zeta_r(k)\right)_+^{\frac{\sigma_3}{4}}\\
         &= P^{(r)}_-(k).
    \end{aligned}
\end{equation*}
For $k\in(c_l-\varrho,c_l)$, we have
\begin{equation*}
    \begin{aligned}
         P^{(r)}_+(k)&=M^{(\infty)}_+(k)Q^{(r)}(\zeta_r(k))_+^{-1}\frac{1}{\sqrt{2}}(I-i\sigma_1)\left(\zeta_r(k)\right)^{\frac{\sigma_3}{4}}\\
         &=M^{(\infty)}_-(k)\begin{pmatrix}
             0&1\\-1&0
         \end{pmatrix}\sigma_1\frac{1}{\sqrt{2}}(I-i\sigma_1)\left(\zeta_r(k)\right)^{\frac{\sigma_3}{4}}\\
         &=M^{(\infty)}_-(k)Q^{(r)}(\zeta_r(k))_-^{-1}\frac{1}{\sqrt{2}}(I-i\sigma_1)\left(\zeta_r(k)\right)^{\frac{\sigma_3}{4}}\\
         &= P^{(r)}_-(k).
    \end{aligned}
\end{equation*}
Moreover, as $k\to c_l$, $P^{(r)}(k)$ has the following asymptotics:
\begin{equation}\label{asy for Pr}
    P^{(r)}(k)=  P^{(r)}(c_l) +\mathcal{O}(|k-c_l|),
\end{equation}
where
\begin{equation}\label{prcl}
    P^{(r)}(c_l)=  \begin{pmatrix}
      \frac{1-i}{2}(2c_l)^{\frac{1}{4}}\left(\frac{3}{2}\right)^{\frac{1}{6}}t^{\frac{1}{6}}\left|g_{\textup{\uppercase\expandafter{\romannumeral1}}}^{(2)}(c_l)\right|^{\frac{1}{6}}&\frac{1-i}{2}(2c_l)^{-\frac{1}{4}}\left(\frac{3}{2}\right)^{-\frac{1}{6}}t^{-\frac{1}{6}}\left|g_{\textup{\uppercase\expandafter{\romannumeral1}}}^{(2)}(c_l)\right|^{-\frac{1}{6}}\\\frac{1+i}{2}(2c_l)^{\frac{1}{4}}\left(\frac{3}{2}\right)^{\frac{1}{6}}t^{\frac{1}{6}}\left|g_{\textup{\uppercase\expandafter{\romannumeral1}}}^{(2)}(c_l)\right|^{\frac{1}{6}}&\frac{-1-i}{2}(2c_l)^{-\frac{1}{4}}\left(\frac{3}{2}\right)^{-\frac{1}{6}}t^{-\frac{1}{6}}\left|g_{\textup{\uppercase\expandafter{\romannumeral1}}}^{(2)}(c_l)\right|^{-\frac{1}{6}}
    \end{pmatrix}.
\end{equation}

We can construct the local RH problem for $\widetilde M^{(l)}(k)$ near $-c_l$ by the symmetry $\widetilde M^{(l)}(k)=\sigma_1\widetilde M^{(r)}(-k)\sigma_1$, i.e.,
\begin{equation*}
     \widetilde M^{(l)}(k)=P^{(l)}(k)\sigma_1\begin{pmatrix}
		1 & 0\\
		i a(s) & 1
		\end{pmatrix}M^{P_{34}}(\zeta_l(k); 0, 0, s)e^{\hat\theta(\zeta_l(k))\sigma_3}\sigma_1Q^{(l)}(\zeta_l(k))G^{(l)}(\zeta_l(k))^{-1},
\end{equation*}
where
\begin{align*}
    &P^{(l)}(k)=\sigma_1P^{(r)}(-k)\sigma_1,\quad k \in D_{\varrho}(-c_l),\\
    &Q^{(l)}(k)=\sigma_1Q^{(r)}(-k)\sigma_1,\quad k \in D_{\varrho}(-c_r)\setminus\mathbb{R},\\
    &G^{(l)}(k)=\sigma_1G^{(r)}(-k)\sigma_1,\quad k\in\mathbb{C}.
\end{align*}
It can be deduced from the symmetry of $P^{(l)}(k)$ and $P^{(r)}(k)$ that the former one is also analytic in $D_{\varrho}(-c_l)$ with the following asymptotics as $k\to -c_l$:
\begin{equation}\label{asymptotic for P^l II}
     P^{(l)}(k)= \sigma_1P^{(r)}(c_l)\sigma_1+\mathcal{O}(|k+c_r|),\quad k\to -c_l,
\end{equation}
where $P^{(r)}(c_l)$ is  given in \eqref{prcl}.
Now we introduce the following lemma, which describes the relation between the local parametrices $\widetilde M^{(j)}$, $j\in\{r,l\}$ and the global parametrix $M^{(\infty)}$, and the approximation of the jump matrices.
\begin{lemma}\label{Lemma v3-vr-II}
    For $\xi\in\mathcal{T}_\textup{\uppercase\expandafter{\romannumeral1}}$ and $t>0$, the function $\widetilde M^{(j)}$ has the following properties for $j\in\{l,r\}$:
    \begin{itemize}
        \item [\rm (a)] As $t\to+\infty$,
        \begin{equation*}
            \|\widetilde M^{(j)}(k)M^{(\infty)}(k)^{-1}-I\|_{L^\infty(\partial D_\varrho(\pm c_l))}=\mathcal{O}(t^{1/3-\epsilon}).
        \end{equation*}
        Here, the notation ``$\pm$'' stands for ``$+$'' when $j=r$ and ``$-$'' when $j=l$.
 \item [\rm (b)] Across $\Gamma^{(j)}$, the jump matrix $\widetilde V^{(j)}$ of $\widetilde M^{(j)}(k)$ satisfies
 \begin{align*}
    & \|V^{(3)}-\widetilde V^{(j)}\|_{L^1(\Gamma^{(j)})}=\mathcal{O}(t^{-1}),\\
    & \|V^{(3)}-\widetilde V^{(j)}\|_{L^2(\Gamma^{(j)})}=\mathcal{O}(t^{-\frac{2}{3}}),\\
     & \|V^{(3)}-\widetilde V^{(j)}\|_{L^\infty(\Gamma^{(j)})}=\mathcal{O}(t^{-\frac{1}{3}}).
 \end{align*}
    \end{itemize}
\end{lemma}
\begin{proof}
\hfill
\begin{itemize}
    \item [(a)] We take $j=r$ for example. It follows directly from the asymptotic behavior of $P^{(r)}(k)$ in \eqref{asy for Pr} and the large-$\zeta_r$ expansion of $\widehat S^{(r)}(\zeta_r)$ in \eqref{asy S} that
    \begin{align*}
        \widetilde M^{(r)}(k)M^{(\infty)}(k)^{-1}-I&=P^{(r)}(k)\widehat N^{(r)}(\zeta_r(k))M^{(\infty)}(k)^{-1}-I\\
        &=P^{(r)}(k)\left(I+\frac{M_1^{P_{34}}(s)}{\zeta_r}+\mathcal{O}(\zeta_r^{-2}) \right) P^{(r)}(k)^{-1}-I\\
        &=P^{(r)}(k)\left(\frac{M_1^{P_{34}}(s)}{\zeta_r}+\mathcal{O}(\zeta_r^{-2}) \right) P^{(r)}(k)^{-1}=\mathcal{O}(t^{1/3-\epsilon}).
    \end{align*}
     \item [(b)] Using the definitions of $\widetilde M^{(r)}(k)$ in \eqref{def:widetildeMr} and the transformation $\widehat M^{(r)}(\zeta_r):=M^{(r)}(k(\zeta_r))$, we find that $V^{(3)}(k)-\widetilde V^{(r)}(k)$ on $\Gamma^{(r)}$ coincides with $\widehat V^{(3)}(\zeta_r)-\widehat V_N^{(r)}(\zeta_r)$ on $\widehat\Gamma^{(r)}$. Consequently, item (b) follows from the analysis in Proposition \ref{pro:estimate E}.

\end{itemize}
\end{proof}

\subsection{Small norm RH problem for $M^{(err)}$}\label{subsec:small norm RH RII}
Define
\begin{align}\label{def:Merr TI}
    M^{(err)}(k)=M^{(err)}(k;x,t):=
            \begin{cases}
             M^{(3)}(k)\left(M^{(\infty)}\left(k\right)\right)^{-1}, & k\in \mathbb{C}\setminus\left(D_{\varrho}\left(c_l\right)\cup D_{\varrho}\left(-c_l\right)\right),\\
             M^{(3)}(k)\left(\widetilde M^{(r)}\left(k\right)\right)^{-1}, & k\in D_{\varrho}\left(c_l\right),\\
             M^{(3)}(k)\left(\widetilde M^{(l)}\left(k\right)\right)^{-1}, & k\in D_{\varrho}\left(-c_l\right).
            \end{cases}
    \end{align}
 It is readily seen that $M^{(err)}(k)$ satisfies the following RH problem.
    \begin{RHP}\label{RHP: Merr TI}
    \hfill
    \begin{itemize}
        \item $M^{(err)}(k)$ is holomorphic for $k\in\mathbb{C}\setminus\Gamma^{(err)}$, where
        \begin{equation*}
        \Gamma^{(err)}:=\partial D_{\varrho}\left(\eta\right)\cup\partial D_{\varrho}\left(-\eta\right)\cup \Gamma^{(3)};
        \end{equation*}
        see Figure \textup{\ref{fig:jump contour Merr RII}} for an illustration.
        \item $M^{(err)}(k)$ has continuous boundary values $M^{(err)}_\pm(k)$ on $k\in \Gamma^{(err)}$ with the jump condition
        \begin{equation*}
            M^{(err)}_{+}(k)=M^{(err)}_{-}(k)V^{(err)}(k),
        \end{equation*}
        where
         \begin{align}\label{equ:jump Verr RII}
           V^{(err)}(k)=
                \begin{cases}
                 M^{(\infty)}_-(k)V^{(3)}(k){M^{(\infty)}_+(k)}^{-1}, & k\in\Gamma^{(3)}\backslash \left(\overline{D_\varrho(c_l)}\cup \overline{D_\varrho(-c_l)}\right),\\
                 \widetilde M^{(r)}(k){M^{(\infty)}(k)}^{-1}, & k\in\partial D_\varrho(c_l),\\
               \widetilde  M^{(l)}(k){M^{(\infty)}(k)}^{-1}, & k\in\partial D_\varrho(-c_l),\\
                 \widetilde M^{(r)}_-(k)V^{(3)}(k){\widetilde M^{(r)}_+(k)}^{-1}, & k\in\Gamma^{(3)}\cap D_\varrho(c_l),\\
                \widetilde M^{(l)}_-(k)V^{(3)}(k){\widetilde M^{(l)}_+(k)}^{-1}, & k\in\Gamma^{(3)}\cap D_\varrho(-c_l).
                \end{cases}
        \end{align}
        \item As $k\rightarrow\infty$ in $k\in\mathbb{C}\setminus\Gamma^{(err)}$, we have $M^{(err)}(k)=I+\mathcal{O}(k^{-1})$.
        \item As $k\rightarrow\pm c_l$, we have $M^{(err)}(k)=\mathcal{O}(1)$.
    \end{itemize}
    \end{RHP}

\begin{figure}[H]
\begin{center}
    \tikzset{every picture/.style={line width=0.75pt}}
    \begin{tikzpicture}[x=0.75pt,y=0.75pt,yscale=-1.15,xscale=1.15]
%Shape: Circle (left) -- 放大后半径35
    \draw[color={rgb, 255:red, 0; green, 0; blue, 0}, draw opacity=1,
          decoration={markings, mark=at position 0.25 with {\arrow{latex}}},
          postaction={decorate}]
        (219,146) .. controls (219,126.68) and (234.68,111) .. (254,111)
                  .. controls (273.32,111) and (289,126.68) .. (289,146)
                  .. controls (289,165.32) and (273.32,181) .. (254,181)
                  .. controls (234.68,181) and (219,165.32) .. (219,146) -- cycle;

    %Shape: Circle (right) -- 放大后半径35
    \draw[color={rgb, 255:red, 0; green, 0; blue, 0}, draw opacity=1,
          decoration={markings, mark=at position 0.25 with {\arrow{latex}}},
          postaction={decorate}]
        (381,146) .. controls (381,126.68) and (396.68,111) .. (416,111)
                  .. controls (435.32,111) and (451,126.68) .. (451,146)
                  .. controls (451,165.32) and (435.32,181) .. (416,181)
                  .. controls (396.68,181) and (381,165.32) .. (381,146) -- cycle;
    %Straight Lines
    \draw[decoration={markings, mark=at position 0.5 with {\arrow{latex}}},
          postaction={decorate}]
        (175,92) -- (225.1,126.2);

    \draw[decoration={markings, mark=at position 0.5 with {\arrow{latex}}},
          postaction={decorate}]
        (175,197) -- (224.6,165.0);

    %Straight Lines
    \draw[decoration={markings, mark=at position 0.5 with {\arrow{latex}}},
          postaction={decorate}]
        (444.1,125.3) -- (492,90);

    \draw[decoration={markings, mark=at position 0.5 with {\arrow{latex}}},
          postaction={decorate}]
        (445.4,164.9) -- (492,195);

    %Straight Lines -- 末端箭头
    \draw[-latex] (451,146) -- (573,146);
\draw (381,146)--(451,146) ;
    %Straight Lines -- 中点箭头（两圆之间）
    \draw[decoration={markings, mark=at position 0.55 with {\arrow{latex}}},
          postaction={decorate}]
        (289,146) -- (381,146);

    %Straight Lines -- 无箭头
    \draw (93,146) -- (219,146);
  \draw  (219,146)-- (289,146);
    % 短线段（圆边界到圆心）
    \draw (225.1,126.2) -- (254,146);
    \draw (224.6,165.0) -- (254,146);
    \draw (444.1,125.3) -- (416,146);
    \draw (445.4,164.9) -- (416,146);

    % Text Node
    \draw (408,148.4) node [anchor=north west][inner sep=0.75pt][font=\tiny] {\footnotesize$c_l$};
    \draw (429,148.4) node [anchor=north west][inner sep=0.75pt][font=\tiny] {\footnotesize$\eta_l$};
    \draw (245,147.4) node [anchor=north west][inner sep=0.75pt][font=\tiny] {\footnotesize$-c_l$};
    \draw (220,147.4) node [anchor=north west][inner sep=0.75pt][font=\tiny] {\footnotesize$-\eta_l$};
    \draw (579,137.4) node [anchor=north west][inner sep=0.75pt] {$\re k$};
    \draw (300,147.4) node [anchor=north west][inner sep=0.75pt][font=\tiny] {\footnotesize$-c_{r}$};
    \draw (349,148.4) node [anchor=north west][inner sep=0.75pt][font=\tiny] {\footnotesize$c_{r}$};

    % 圆心标记
    \fill (254,146)  circle (1.2pt);
    \fill (416,146)  circle (1.2pt);
    \fill (354,146.4)  circle (1.2pt);
    \fill (308,146.4)  circle (1.2pt);
    \fill (433,146.4)  circle (1.2pt);
    \fill (236,146.4)  circle (1.2pt);

    \end{tikzpicture}
\caption{The jump contours $\Gamma^{(err)}$ of RH problem for $M^{(err)}$ for $\xi\in\mathcal{T}_\textup{\uppercase\expandafter{\romannumeral1}}$.}
\label{fig:jump contour Merr RII}
\end{center}
\end{figure}

Using items (c)-(d) in Proposition \ref{analytic extension of r RI}, together with Proposition \ref{pro:estimate E} and Lemma \ref{Lemma v3-vr-II}, a straightforward calculation yields the following proposition.
\begin{proposition}\label{pro:est VE-I}
    For $\xi\in\mathcal{T}_\textup{\uppercase\expandafter{\romannumeral1}}$, the function $V^{(err)}(\cdot)-I:\Gamma^{(err)}\to\mathbb{C}^{2\times2}$ lies in $L^p\left(\Gamma^{(err)}\right)$ for $p=1,2,\infty$, and uniformly for $\xi\in\mathcal{T}_\textup{\uppercase\expandafter{\romannumeral1}}$, it holds that
    \begin{align}
    &\Vert V^{(err)}(k)-I \Vert_{L^p\left(\Gamma^{(err)}\setminus\left(\overline{D_\varrho(c_l)}\cup \overline{D_\varrho(-c_l)}\cup\mathbb{R}\right)\right)}=
    \mathcal{O}(e^{-ct}), \label{Ve-I-1}\\
    &\Vert V^{(err)}(k)-I \Vert_{L^p\left(\partial D_\varrho(c_l)\cup \partial D_\varrho(-c_l)\right)}=\mathcal{O}(t^{\gamma_p}),\label{Ve-I-2}\\
    &\Vert V^{(err)}(k)-I \Vert_{L^p\left(\mathbb{R}\setminus\left(\overline{D_\varrho(c_l)}\cup \overline{D_\varrho(-c_l)}\right)\right)}=\mathcal{O}(t^{-1}),\label{Ve-I-3}\\
    &\Vert V^{(err)}(k)-I \Vert_{L^p\left(\Gamma^{(err)}\cap \left(D_\varrho(c_l)\cup D_\varrho(-c_l)\right)\right)}=\mathcal{O}(t^{\omega_p}).\label{Ve-I-4}
   % \begin{cases}
     %    \mathcal{O}(e^{-ct}), & k\in\Gamma^{(err)}\setminus\left(\overline{D_\varrho(c_l)}\cup \overline{D_\varrho(-c_l)}\cup\mathbb{R}\right),\\
    %     \mathcal{O}(t^{\gamma_p}), &k\in\partial D_\varrho(c_l)\cup \partial D_\varrho(-c_l),\\
    %     \mathcal{O}(t^{-1}), &k\in\mathbb{R}\setminus\left(\overline{D_\varrho(c_l)}\cup \overline{D_\varrho(-c_l)}\right),\\
   %      \mathcal{O}(t^{\omega_p}), &k\in \Gamma^{(err)}\cap \left(D_\varrho(c_l)\cup D_\varrho(-c_l)\right),
   %     \end{cases}
\end{align}
where $c>0$,
\begin{align*}
    \gamma_1=-\frac{1}{3},\;\gamma_2=-\frac{\epsilon}{2},\gamma_\infty=\frac{1}{3}-\epsilon,\quad
\omega_1=-1,\;\omega_2=-\frac{2}{3},\omega_\infty=-\frac{1}{3}.
\end{align*}
In particular, uniformly for $\xi\in\mathcal{T}_\textup{\uppercase\expandafter{\romannumeral1}}$, it holds that
\begin{equation}\label{equ:unifrom V-I}
    \Vert V^{(err)}(k)-I \Vert_{L^p\left(\Gamma^{(err)}\right)}=\mathcal{O}(t^{\gamma_p}), \quad p=1,2,\infty.
\end{equation}
\end{proposition}
\begin{proof}
    We have
    \begin{equation*}
          V^{(err)}-I =\begin{cases}
                 M^{(\infty)}\left(V^{(3)}-I\right){M^{(\infty)}}^{-1}, & k\in\Gamma^{(3)}\backslash \left(\overline{D_\varrho(c_l)}\cup \overline{D_\varrho(-c_l)}\cup[-c_l,c_l]\right),\\
                 M^{(\infty)}_-\left(V^{(3)}-V^{(\infty)}\right){M^{(\infty)}_+}^{-1}, & k\in[-c_l,c_l]\backslash \left(\overline{D_\varrho(c_l)}\cup \overline{D_\varrho(-c_l)}\right),\\
                 \widetilde M^{(r)}{M^{(\infty)}}^{-1}-I, & k\in\partial D_\varrho(c_l),\\
               \widetilde  M^{(l)}{M^{(\infty)}}^{-1}-I, & k\in\partial D_\varrho(-c_l),\\
                 \widetilde M^{(r)}_-\left(V^{(3)}-\widetilde V^{(r)}\right){\left(\widetilde M^{(r)}_{+}\right)}^{-1}, & k\in\Gamma^{(3)}\cap D_\varrho(c_l),\\
                \widetilde M^{(l)}_-\left(V^{(3)}-\widetilde V^{(l)}\right)\left({\widetilde M^{(l)}_{+}}\right)^{-1}, & k\in\Gamma^{(3)}\cap D_\varrho(-c_l).
                \end{cases}
    \end{equation*}
    Due to the signature of $g_\textup{\uppercase\expandafter{\romannumeral1}}$ outside $\overline{D_\varrho(c_l)}\cup \overline{D_\varrho(-c_l)}\cup\mathbb{R}$, we infer from item (c) of Proposition \ref{analytic extension of r RI} that $\Vert V^{(err)}(k)-I \Vert_{L^p\left(\Gamma^{(err)}\setminus\left(\overline{D_\varrho(c_l)}\cup \overline{D_\varrho(-c_l)}\cup\mathbb{R}\right)\right)}$ has exponential decay as $t\to+\infty$. Moreover, item (d) of Proposition \ref{analytic extension of r RI} implies that $\Vert V^{(err)}(k)-I \Vert_{L^p\left(\mathbb{R}\setminus\left(\overline{D_\varrho(c_l)}\cup \overline{D_\varrho(-c_l)}\right)\right)}=\mathcal{O}(t^{-1})$ uniformly for $\xi\in\mathcal{T}_\textup{\uppercase\expandafter{\romannumeral1}},\;p=1,2,\infty.$ It can be inferred from item (a) and item (b) of Lemma \ref{Lemma v3-vr-II} that \eqref{Ve-I-2} and \eqref{Ve-I-4} hold respectively.
\end{proof}
It then follows from the small norm RH problem theory \cite{DZ02} that there exists
a unique solution to RH problem \ref{RHP: Merr TI} for large positive $t$. Furthermore,
according to Beals-Coifman theory \cite{BCdecom}, the solution to $M^{(err)}$ can be given by
\begin{equation}\label{equ:BCsolforError}
M^{(err)}(k) = I + \frac{1}{2\pi i} \int_{\Gamma^{(err)}} \frac{(I + \varpi(z))(V^{(err)}(z) - I)}{z - k} \dif z,
\end{equation}
where $\varpi \in L^2(\Gamma^{(err)})$ is the unique solution of $(1 - C_{V^{(err)}})\varpi = C_{V^{(err)}} I$. And $C_{V^{(err)}}: L^2(\Gamma^{(err)}) \to L^2(\Gamma^{(err)})$ is the Cauchy operator on $\Gamma^{(err)}$, which is defined as
$C_{V^{(err)}}(f)(k) = C_{-}f(V^{(err)} - I)$ with $C_{-}$ being the Cauchy projection
operator on $\Gamma^{(err)}$.
%= \lim_{z \to k^-, k \in \Sigma^{(e)}} \int_{\Sigma^{(e)}} \frac{f(z)(V^{(e)}(z) - I)}{z - k} \,\mathrm{d}z.
The existence and uniqueness of $\varpi$ come from the boundedness of the Cauchy operator $C_{-}$, which admits
\begin{align}\label{CverrL2}
\|C_{V^{(err)}}\|_{L^2(\Gamma^{(err)})} \leqslant \|C_{-}\|_{L^2(\Gamma^{(err)}) \to L^2(\Gamma^{(err)})} \|V^{(err)} - I\|_{L^\infty(\Gamma^{(err)})} = \mathcal{O}(t^{1/3-\epsilon}).
\end{align}
In addition,
\begin{equation*}
    \|\varpi\|_{L^2(\Gamma^{(err)})} \leqslant \frac{\|C_{V^{(err)}}\|_{L^2(\Gamma^{(err)})}}{1 - \|C_{V^{(err)}}\|_{L^2(\Gamma^{(err)})}} \lesssim t^{1/3-\epsilon}.
\end{equation*}
On the other hand, $\varpi$ can be written as
$$
 \varpi= C_{V^{(err)}} I+\left(1-C_{V^{(err)}}\right)^{-1}\left(C_{V^{(err)}}^2 I\right),
$$
which implies that
\begin{equation}
\|C_{V^{(err)}}I\|_{L^2(\Gamma^{(err)})}\lesssim t^{1/3-3\epsilon/2},\quad  \|\varpi-C_{V^{(err)}}I\|_{L^2(\Gamma^{(err)})}\lesssim t^{2/3-5\epsilon/2}.\label{estiCVE}
\end{equation}

Moreover, by \eqref{equ:BCsolforError}, it follows that
%as $k\to\infty$,
%For later use, we conclude this section with behavior of $M^{(err)}(k)$ at $k=\infty$.
%By \eqref{equ:BCsolforError}, it follows that
\begin{equation*}
    M^{(err)}(k)=I+\frac{M^{(err)}_1}{k}+\mathcal{O}(k^{-2}), \quad k\rightarrow\infty,
\end{equation*}
where
\begin{equation}\label{equ:Merr_1 TI}
    M^{(err)}_1=-\frac{1}{2\pi i}\int_{\Gamma^{(err)}}\left(I + \varpi(z)\right)(V^{(err)}(z)-I)\dif z.
\end{equation}
Then we obtain the following proposition.
\begin{proposition}\label{prop: result of Merr_1 TI}
     For $\xi\in\mathcal{T}_\textup{\uppercase\expandafter{\romannumeral1}}$, as $t\rightarrow+\infty$, we have
    \begin{equation*}\label{equ:Merr_1 RII}
    \begin{aligned}
         M^{(err)}_1=&\frac{1}{4}\left(\frac{3}{2}\right)^{-\frac{2}{3}}\left|g_{\textup{\uppercase\expandafter{\romannumeral1}}}^{(2)}(c_l)\right|^{-\frac{2}{3}}\begin{pmatrix}-it^{-\frac{1}{3}}(2c_l)^{\frac{1}{2}}\left(\frac{3}{2}\right)^{\frac{1}{3}}\left|g_{\textup{\uppercase\expandafter{\romannumeral1}}}^{(2)}(c_l)\right|^{\frac{1}{3}}s^2&it^{-\frac{2}{3}}\left[s^2\frac{\textup{Ai} '(s)}{\textup{Ai} (s)}-2s\right]\\-it^{-\frac{2}{3}}\left[s^2\frac{\textup{Ai} '(s)}{\textup{Ai} (s)}-2s\right]&it^{-\frac{1}{3}}(2c_l)^{\frac{1}{2}}\left(\frac{3}{2}\right)^{\frac{1}{3}}\left|g_{\textup{\uppercase\expandafter{\romannumeral1}}}^{(2)}(c_l)\right|^{\frac{1}{3}} s^2
  \end{pmatrix}\\
  &+\mathcal{O}(t^{\frac{1}{3}-2\epsilon}),
    \end{aligned}
    \end{equation*}
    where $g_{\textup{\uppercase\expandafter{\romannumeral1}}}^{(2)}(c_l)$ and $s$ are given by \eqref{equ:coefficients of gI} and \eqref{defS} respectively.
\end{proposition}
\begin{proof}
    We first divide $M^{(err)}_1$ into four parts by
    \begin{align*}
        &I_1:=-\frac{1}{2\pi i}\int_{\Gamma^{(err)}}\varpi(z)\left(V^{(err)}(z)-I\right)\dif z, \\
        &I_2:=-\frac{1}{2\pi i}\int_{\Gamma^{(err)}\setminus\left(D_\varrho(c_l)\cup D_\varrho(-c_l)\right)}\left(V^{(err)}(z)-I\right)\dif z, \\
        &I_3:=-\frac{1}{2\pi i}\oint_{\partial D_\varrho(c_l)\cup \partial D_\varrho(-c_l)}\left(V^{(err)}(z)-I\right)\dif z,\\
        &I_4:=-\frac{1}{2\pi i}\int_{\Gamma^{(err)}\cap\left(D_\varrho(c_l)\cup D_\varrho(-c_l)\right)}\left(V^{(err)}(z)-I\right)\dif z.
       \end{align*}
The main contribution to
$M^{(err)}_1$ stems from $I_3$. For the other three parts, we have the following estimates:
$I_1=\mathcal{O}(t^{1/3-2\epsilon})$ by \eqref{estiCVE} and \eqref{equ:unifrom V-I}, $I_2=\mathcal{O}(e^{-ct})$ by \eqref{Ve-I-1}, and $I_4=\mathcal{O}(t^{-1})$ by \eqref{Ve-I-4}.

Now we turn to estimate $I_3$.
We split $I_3$ into two parts:
\begin{align*}
    I_3:=I_3^{(r)}+I_3^{(l)},
\end{align*}
where
$I_3^{(r)}$ and $I_3^{(l)}$ represent the contour integrals along $\partial D_\varrho(c_l)$ and $\partial D_\varrho(-c_l)$,
respectively.
A detailed analysis of $I_3^{(r)}$ will be provided below, and a similar approach can be applied to $I_3^{(l)}$.
From Lemma \ref{Lemma v3-vr-II} and \eqref{eq:Psi-infinity}, we have
\begin{equation}\label{equ:I_3R TI}
    \begin{aligned}
    I_3^{(r)}=&-\frac{1}{2\pi i}\oint_{\partial D_\varrho(c_l)}\left(V^{(err)}(z)-I\right)\dif z\\
    =&-\frac{1}{2\pi i}\oint_{\partial D_\varrho(c_l)}P^{(r)}(z)\left(\frac{M_1^{P_{34}}(s)}{\zeta_r(z)}+\mathcal{O}(\zeta_r^{-2}) \right) P^{(r)}(z)^{-1}\dif z\\
    =&-\frac{1}{2\pi i}\oint_{\partial D_\varrho(c_l)}P^{(r)}(z)t^{\frac{\hat\sigma_3}{6}}\left(\frac{M_1^{P_{34}}(s)}{\zeta_r(z)} \right) P^{(r)}(z)^{-1}\dif z+\mathcal{O}(t^{-1})\\
    =&-\frac{1}{2\pi i}\oint_{\partial D_\varrho(c_l)}\frac{t^{1/3}}{\zeta_r(z)}P^{(r)}(z)\begin{pmatrix}
        0&\left(M_1^{P_{34}}(s)\right)_{12}\\0&0
    \end{pmatrix} P^{(r)}(z)^{-1}\dif z\\
    &-\frac{1}{2\pi i}\oint_{\partial D_\varrho(c_l)}\frac{1}{\zeta_r(z)}P^{(r)}(z)\begin{pmatrix}
        \left(M_1^{P_{34}}(s)\right)_{11}&0\\0&\left(M_1^{P_{34}}(s)\right)_{22}
    \end{pmatrix} P^{(r)}(z)^{-1}\dif z+\mathcal{O}(t^{-1})\\
    =&-\frac{(\frac{3}{2})^{-2/3}\left|g_{\textup{\uppercase\expandafter{\romannumeral1}}}^{(2)}(c_l)\right|^{-2/3}}{t^{2/3}}P^{(r)}(c_l)\begin{pmatrix}
        \left(M_1^{P_{34}}(s)\right)_{11}&t^{1/3}\left(M_1^{P_{34}}(s)\right)_{12}\\0&\left(M_1^{P_{34}}(s)\right)_{22}
    \end{pmatrix} P^{(r)}(c_l)^{-1}+\mathcal{O}(t^{-1}),
\end{aligned}
\end{equation}
where we have used the residue theorem in the last equality. By the symmetry, we can calculate $I_3^{(l)}$ in a similar way
\begin{equation}\label{equ:I_3L TI}
     I_3^{(l)}=\frac{(\frac{3}{2})^{-2/3}\left|g_{\textup{\uppercase\expandafter{\romannumeral1}}}^{(2)}(c_l)\right|^{-2/3}}{t^{2/3}}\sigma_1P^{(r)}(c_l)\begin{pmatrix}
        \left(M_1^{P_{34}}(s)\right)_{11}&t^{1/3}\left(M_1^{P_{34}}(s)\right)_{12}\\0&\left(M_1^{P_{34}}(s)\right)_{22}
    \end{pmatrix} P^{(r)}(c_l)^{-1}\sigma_1+\mathcal{O}(t^{-1}).
\end{equation}
Combining the results from \eqref{equ:I_3R TI} and \eqref{equ:I_3L TI}, we obtain
\begin{align*}
  I_3=\frac{1}{4}\left(\frac{3}{2}\right)^{-\frac{2}{3}}\left|g_{\textup{\uppercase\expandafter{\romannumeral1}}}^{(2)}(c_l)\right|^{-\frac{2}{3}}\begin{pmatrix}-it^{-\frac{1}{3}}(2c_l)^{\frac{1}{2}}\left(\frac{3}{2}\right)^{\frac{1}{3}}\left|g_{\textup{\uppercase\expandafter{\romannumeral1}}}^{(2)}(c_l)\right|^{\frac{1}{3}}s^2&it^{-\frac{2}{3}}\left[s^2\frac{\textup{Ai} '(s)}{\textup{Ai} (s)}-2s\right]\\-it^{-\frac{2}{3}}\left[s^2\frac{\textup{Ai} '(s)}{\textup{Ai} (s)}-2s\right]&it^{-\frac{1}{3}}(2c_l)^{\frac{1}{2}}\left(\frac{3}{2}\right)^{\frac{1}{3}}\left|g_{\textup{\uppercase\expandafter{\romannumeral1}}}^{(2)}(c_l)\right|^{\frac{1}{3}} s^2
  \end{pmatrix}+\mathcal{O}(t^{-1}).
\end{align*}
Based on the analysis above, we obtain the desired result as given in \eqref{equ:Merr_1 RII}.
\end{proof}
\subsection{Proof of the part (\textup{\uppercase\expandafter{\romannumeral1}}) of Theorem \ref{thm:mainthm}}
By tracing back the transformations \eqref{def:M1 RI}, \eqref{def:M2 RI}, \eqref{def:M3 RI} and \eqref{def:Merr TI},
we conclude that for $k\in\mathbb{C}\setminus\Gamma^{(3)}$,
\begin{equation*}  M(k)=M^{(err)}(k)M^{(\infty)}(k)D_{\textup{\uppercase\expandafter{\romannumeral1}}}^{\sigma_3}(k)e^{-it(g_{\textup{\uppercase\expandafter{\romannumeral1}}}(k)-\theta(k))\sigma_3},
\end{equation*}
where $M^{(err)}$ and $M^{(\infty)}$ are defined in
\eqref{def:Merr TI} and \eqref{equ:sol of Minfty} respectively.
From the reconstruction formula stated in \eqref{equ:recovering formula}, we obtain
\begin{equation*}
q(x,t)=2i\left[(M^{(err)}_1)_{12}+\lim_{k\rightarrow\infty}k\left(\Delta_{l}(k)\right)_{12}\right].
\end{equation*}
It then follows from \eqref{equ:Delta_j} and Proposition \eqref{prop: result of Merr_1 TI} that part (\textup{\uppercase\expandafter{\romannumeral1}}) of Theorem \ref{thm:mainthm} holds.

\section{Asymptotic analysis of the RH problem for $M$ in $\mathcal{T}_{\textup{\uppercase\expandafter{\romannumeral2}}}$}\label{sec:asymptotic analysis in RII}
This section is devoted to the long-time asymptotic analysis
of the RH problem \ref{RHP:basic RHP} in the region $\mathcal{T}_{\textup{\uppercase\expandafter{\romannumeral2}}}$. It is assumed that $0<t^{2/3}\left(\xi+\frac{c_r^2}{2}\right)<C$ with $C>0$ throughout this section since the analysis on the other half region of $\mathcal{T}_{\textup{\uppercase\expandafter{\romannumeral2}}}$ is similar. As the analysis of this region is similar with that in Section \ref{sec:asymptotic analysis in RI}, we only sketch the main procedures.

\subsection{First transformation: $M\rightarrow M^{(1)}$}

We begin with the introduction of the $g$-function \cite{gfunction, G2001,GT2002},  to control the exponentially growing off-diagonal factors in the jump matrix \eqref{equ:jump V(k)}.
For $\xi\in\mathcal{T}_{\textup{\uppercase\expandafter{\romannumeral2}}}$, we introduce
\begin{equation}\label{equ:g function TII}
g_{\textup{\uppercase\expandafter{\romannumeral2}}}(k)=g_{\textup{\uppercase\expandafter{\romannumeral2}}}(k;\xi):=(4k^2+12\xi+2c_{r}^2)X_{r}(k), \quad X_{r}(k)=\sqrt{k^2-c_{r}^2},
\end{equation}
where the branch of the square root being chosen such that $X_r(k)=k+\mathcal{O}(k^{-1})$ as $k\rightarrow\infty$.
It is readily verified that $g_{\textup{\uppercase\expandafter{\romannumeral2}}}$ has the following  properties.
%properties for $g_{\textup{\uppercase\expandafter{\romannumeral2}}}$ defined in \eqref{equ:g function TII} hold true.
\begin{proposition}The function $g_{\textup{\uppercase\expandafter{\romannumeral2}}}$  defined in \eqref{equ:g function TII} satisfies the following properties:
    \begin{itemize}
        \item $g_{\textup{\uppercase\expandafter{\romannumeral2}}}(k)$ is holomorphic for $k\in\mathbb{C}\setminus[-c_{r},c_{r}]$.
        \item As $k\rightarrow\infty$ in $\mathbb{C}\setminus[-c_{r},c_{r}]$,
        we have $g_{\textup{\uppercase\expandafter{\romannumeral2}}}(k)=\theta(k)+\mathcal{O}(k^{-1})$.
        \item For $k\in(-c_{r},c_{r})$, $g_{\textup{\uppercase\expandafter{\romannumeral2}},+}(k)+g_{\textup{\uppercase\expandafter{\romannumeral2}},-}(k)=0$.
        \item As $k\to\pm c_r$, we have
        \begin{equation}\label{g_II asy cr}
             g_{\textup{\uppercase\expandafter{\romannumeral2}}}(k)=6(\pm2c_r)^{\frac{1}{2}}(c_r^2+2\xi)(k\mp c_r)^{\frac{1}{2}}+\left[3(\pm2c_r)^{-\frac{1}{2}}(c_r^2+2\xi)+4(\pm2c_r)^{\frac{3}{2}}\right](k\mp c_r)^{\frac{3}{2}}+\mathcal{O}((k\mp c_r)^{\frac{5}{2}}).
        \end{equation}
    \end{itemize}
\end{proposition}

A straightforward calculation shows that $\pm \eta_{r}(\xi)=\pm \sqrt{-3\xi-c_r^2/2}$ are the two real zeros
for $\xi\in\mathcal{T}_{\textup{\uppercase\expandafter{\romannumeral2}}}$.
Noticing that $\im g_{\textup{\uppercase\expandafter{\romannumeral2}},+}>0$ for $k\in(-c_r,-\eta_r)\cup(\eta_r, c_r)$,
then the signature table for $\im g_{\textup{\uppercase\expandafter{\romannumeral2}}}$ is illustrated in Figure \ref{fig:signs img RIII}.

\begin{figure}[htbp]
\begin{center}
    \tikzset{every picture/.style={line width=0.75pt}} %set default line width to 0.75pt
    \begin{tikzpicture}[x=0.75pt,y=0.75pt,yscale=-1,xscale=1]
    %Curve Lines [id:da23990346379567207]
    \draw    (372,52) .. controls (346,52) and (345,233) .. (378,233) ;
    %Straight Lines [id:da27525634770111207]
    \draw    (161,139) -- (477,139) ;
    %Curve Lines [id:da4970333595425316]
    \draw    (248,50) .. controls (273,51) and (273,231) .. (245,232) ;
    % Text Node
    \draw (245,143.4) node [anchor=north west][inner sep=0.75pt]  [font=\tiny]  {\footnotesize$-\eta_r$};
    \fill (266,139) circle (1.2pt);
    % Text Node
    \draw (356,144.4) node [anchor=north west][inner sep=0.75pt]  [font=\tiny]  {\footnotesize$\eta_r$};
    \fill (353,139) circle (1.2pt);
    % Text Node
    \draw (171,143.4) node [anchor=north west][inner sep=0.75pt]  [font=\tiny]  {\footnotesize$-c_{l}$};
    % Text Node
    \draw (430,144.4) node [anchor=north west][inner sep=0.75pt]  [font=\tiny]  {\footnotesize$c_{l}$};
    % Text Node
    \draw (202.18,143.4) node [anchor=north west][inner sep=0.75pt]  [font=\tiny,rotate=-1.56]  {\footnotesize$-c_{r}$};
    % Text Node
    \draw (400,144.4) node [anchor=north west][inner sep=0.75pt]  [font=\tiny]  {\footnotesize$c_{r}$};
    % Text Node
    \draw (443,89.4) node [anchor=north west][inner sep=0.75pt]    {$+$};
    % Text Node
    \draw (448,170.4) node [anchor=north west][inner sep=0.75pt]    {$-$};
    % Text Node
    \draw (296,77.4) node [anchor=north west][inner sep=0.75pt]    {$-$};
    % Text Node
    \draw (301,189.4) node [anchor=north west][inner sep=0.75pt]    {$+$};
    % Text Node
    \draw (172,172.4) node [anchor=north west][inner sep=0.75pt]    {$-$};
    % Text Node
    \draw (175,90.4) node [anchor=north west][inner sep=0.75pt]    {$+$};
    % Text Node

     \fill (177,139) circle (1.2pt);
    % Text Node
\fill (210,139) circle (1.2pt);
    % Text Node
\fill (435,139) circle (1.2pt);
    % Text Node
  \fill (405,139) circle (1.2pt);
    \end{tikzpicture}
\caption{ The signature table of $\im g_{\textup{\uppercase\expandafter{\romannumeral2}}}$ for $\xi\in\mathcal{T}_{\textup{\uppercase\expandafter{\romannumeral2}}}$,
where ``$+$'' and ``$-$'' denote $\im g_{\textup{\uppercase\expandafter{\romannumeral2}}} > 0$ and $\im g_{\textup{\uppercase\expandafter{\romannumeral2}}} < 0$ in the corresponding regions, respectively.}\label{fig:signs img RIII}
\end{center}
\end{figure}

\subsubsection{RH problem for $M^{(1)}$}
With the aid of  function $g_{\textup{\uppercase\expandafter{\romannumeral2}}}$ defined in \eqref{equ:g function TII},
we define $M^{(1)}$ by
\begin{equation}\label{def:M1 RIII}
 M^{(1)}(k)=   M^{(1)}(k;x,t):=M(k)e^{it\left(g_{\textup{\uppercase\expandafter{\romannumeral2}}}(k)-\theta(k)\right)\sigma_3}.
\end{equation}
Then RH problem for $M^{(1)}$ reads as follows:
\begin{RHP}\label{RHP:M1 RIII}
\hfill
\begin{itemize}
    \item $M^{(1)}(k)$ is holomorphic for $k\in\mathbb{C}\setminus\mathbb{R}$.
    \item $M^{(1)}(k)$ has continuous boundary values $M_{\pm}^{(1)}(k)$ on $\mathbb{R}$ with the jump condition
    \begin{equation*}
        M^{(1)}_{+}(k)=M^{(1)}_{-}(k)V^{(1)}(k), \quad k\in\mathbb{R},
    \end{equation*}
    where
    \begin{equation*}
        V^{(1)}(k)=
            \begin{cases}
            \begin{pmatrix}1-r(k)r^*(k) & -r^*(k)e^{-2itg_{\textup{\uppercase\expandafter{\romannumeral2}}}(k)}\\ r(k)e^{2itg_{\textup{\uppercase\expandafter{\romannumeral2}}}(k)} & 1\end{pmatrix}, &k\in(-\infty,-c_l)\cup(c_l,+\infty),  \\
            \begin{pmatrix}0 & -r_{-}^*(k)e^{-2itg_{\textup{\uppercase\expandafter{\romannumeral2}}}(k)}\\ r_{+}(k)e^{2itg_{\textup{\uppercase\expandafter{\romannumeral2}}}(k)} & 1 \end{pmatrix}, &k\in (-c_l,-c_{r})\cup(c_{r},c_l),\\
            \begin{pmatrix}0 & -1 \\ 1 & 0\end{pmatrix},  &k\in (-c_{r}, c_{r}).
            \end{cases}
    \end{equation*}
    \item As $k\rightarrow\infty$ in $\mathbb{C}\setminus\mathbb{R}$, we have $M^{(1)}(k)=I+\mathcal{O}(k^{-1})$.
    \item $M^{(1)}(k)$ admits the same singular behavior as $M(k)$ at branch points $\pm c_{r}$.
\end{itemize}
\end{RHP}

\subsection{Second transformation: $M^{(1)}\to M^{(2)}$}
In contrast to the previous sections, the signature table of $\im g_{\textup{\uppercase\expandafter{\romannumeral2}}}$ in  Figure \ref{fig:signs img RIII} indicates that it is not necessary to  introduce an auxiliary $D$ function
to open lenses. Nevertheless, we still need to perform an analytic approximation for the reflection coefficient $r(k)$.
\begin{proposition}[Analytic approximation of $r$]\label{analytic extension of r RIII}

There exist continuous functions
\begin{equation*}
    r_a: \left(\overline{U_1^{(2)}}\cup\overline{U_2^{(2)}}\right) \times \mathcal{T}_\textup{\uppercase\expandafter{\romannumeral2}} \rightarrow \mathbb{C} \; \text { and } \; r_r: \left((-\infty,-\eta_r)\cup(\eta_r,+\infty)\right) \times \mathcal{T}_\textup{\uppercase\expandafter{\romannumeral2}}\rightarrow \mathbb{C},
\end{equation*}
which satisfy the following properties:
\begin{itemize}
    \item [\rm (a)] $r(k)=r_a(k)+r_r(k)$, where $r_a(k)=r_a(k;\xi)$ and $r_r(k)=r_r(k;\xi)$ for all $(k;\xi) \in \{(-\infty,-\eta_r)\cup(\eta_r,+\infty)\} \times \mathcal{T}_\textup{\uppercase\expandafter{\romannumeral2}}$.
    \item [\rm (b)] $r_a(-k)=r_a^*(k)$ for $k\in\overline{U_1^{(2)}}\cup\overline{U_2^{(2)}}.$
    \item [\rm (c)]  For $\xi \in \mathcal{T}_{\textup{\uppercase\expandafter{\romannumeral2}}}$,  $r_a(k)$ is holomorphic in $U_1^{(2)}\cup U_2^{(2)}$. Moreover, for $(k;\xi)\in\left(\overline{U_1^{(2)}}\cup\overline{U_2^{(2)}}\right)\times \mathcal{T}_{\textup{\uppercase\expandafter{\romannumeral2}}}$, we have
\begin{equation*}
    \left|r_a(k)-r\left(\pm c_r\right)\right| \lesssim \left|k\mp c_r\right|^{1/2} e^{\frac{t}{4}|\operatorname{Im} g_{\textup{\uppercase\expandafter{\romannumeral2}}}(k)|},
\end{equation*}
and
\begin{equation*}
    \left|r_a(k)\right| \lesssim \frac{e^{t|\operatorname{Im} g_{\textup{\uppercase\expandafter{\romannumeral2}}}(k)|}}{1+|k|^2}.
\end{equation*}
\item [\rm (d)] For $\xi \in \mathcal{T}_{\textup{\uppercase\expandafter{\romannumeral2}}}$, the function $r_r\in L^p\left((-\infty,-c_r)\cup(c_r,+\infty)\right),\; p\in[1,+\infty]$, and as $t \rightarrow +\infty$,
\begin{equation*}
    \left\|r_r(k)\right\|_{L^p\left((-\infty,-c_r)\cup(c_r,+\infty)\right)}=  \mathcal{O}\left(t^{-1}\right).
\end{equation*}
\end{itemize}
\end{proposition}

Now we are ready to open lenses. Introduce a transformation
\begin{equation}\label{def:M2 RIII}
    M^{(2)}(k)= M^{(2)}(k;x,t):=M^{(1)}(k)G(k),
\end{equation}
where
\begin{equation*}
    G(k):=
        \begin{cases}
        \begin{pmatrix}1 & 0 \\ -r_a(k)e^{2itg_{\textup{\uppercase\expandafter{\romannumeral2}}}(k)} & 1\end{pmatrix}, & k\in U^{(2)}_1\cup U^{(2)}_2, \\
        \begin{pmatrix}1 & -r_a^{*}(k)e^{-2itg_{\textup{\uppercase\expandafter{\romannumeral2}}}(k)} \\ 0 & 1\end{pmatrix}, & k\in U_1^{(2)*}\cup U_2^{(2)*}, \\
        I, &\textnormal{elsewhere}.
        \end{cases}
\end{equation*}
Then the RH problem for $M^{(2)}$ is listed below.
\begin{RHP}
    \hfill
\begin{itemize}
    \item $M^{(2)}(k)$ is holomorphic for $k\in\mathbb{C}\setminus\Gamma^{(2)}$, where $\Gamma^{(2)}:=\cup_{j=1}^{2}(\Gamma^{(2)}_j\cup\Gamma_j^{(2)*})\cup\mathbb{R}$;
    see Figure \textup{\ref{fig:U_j domain RIII}} for an illustration.
    \item $M^{(2)}(k)$ has continuous boundary values $M_{\pm}^{(2)}(k)$ on $\Gamma^{(2)}$ with the jump condition
    \begin{equation*}
        M^{(2)}_{+}(k)=M^{(2)}_{-}(k)V^{(2)}(k),
    \end{equation*}
    where
    \begin{equation*}
        V^{(2)}(k)=
            \begin{cases}
            \begin{pmatrix} 1 & 0\\ r_a(k)e^{2itg_{\textup{\uppercase\expandafter{\romannumeral2}}}(k)} & 1\end{pmatrix}, &k\in\Gamma^{(2)}_1\cup\Gamma^{(2)}_2,  \\
            \begin{pmatrix} 1 & -r_a^*(k)e^{-2itg_{\textup{\uppercase\expandafter{\romannumeral2}}}(k)}\\ 0 & 1\end{pmatrix}, & k\in\Gamma^{(2)*}_1\cup\Gamma^{(2)*}_2,  \\
            \begin{pmatrix}
                1-r_r(k)r^*_r(k)&-r_r^*(k)e^{-2itg_{\textup{\uppercase\expandafter{\romannumeral2}}}(k)}\\r_r(k)e^{2itg_{\textup{\uppercase\expandafter{\romannumeral2}}}(k)}&1
            \end{pmatrix},&k\in(-\infty,-c_r)\cup(c_r,+\infty),\\
            \begin{pmatrix}0 & -1 \\ 1 & 0\end{pmatrix},  & k\in (-c_{r}, c_{r}).
            \end{cases}
    \end{equation*}
    \item As $k\rightarrow\infty$ in $\mathbb{C}\backslash\Gamma^{(2)}$, we have $M^{(2)}(k)=I+\mathcal{O}(k^{-1})$.
    \item As $k\rightarrow\pm c_{r}$, $M^{(2)}(k)=\mathcal{O}((k\mp c_{r})^{-1/4})$.
\end{itemize}
\end{RHP}

\begin{figure}[H]
\centering
\tikzset{every picture/.style={line width=0.75pt}}
\begin{tikzpicture}[x=0.75pt,y=0.75pt,yscale=-1,xscale=1]

% 水平主线（箭头向左，反向）
\draw[postaction={decorate}, decoration={markings, mark=at position 0.5 with {\arrow{latex reversed}}}]
      (481,152)--(171,150);

% 右上斜线（箭头反向）
\draw[postaction={decorate}, decoration={markings, mark=at position 0.5 with {\arrow{latex reversed}}}]
     (588,222)--(481,152) ;

% 右下斜线（箭头同向）
\draw[postaction={decorate}, decoration={markings, mark=at position 0.5 with {\arrow{latex}}}]
     (481,152)--(582,86) ;

% 左上斜线（箭头同向）
\draw[postaction={decorate}, decoration={markings, mark=at position 0.5 with {\arrow{latex}}}]
     (64.26,79.61)--(171,150) ;

% 左下斜线（箭头反向）
\draw[postaction={decorate}, decoration={markings, mark=at position 0.5 with {\arrow{latex reversed}}}]
     (171,150)--(69.76,215.63) ;

% 右端水平线（箭头反向）
\draw[postaction={decorate}, decoration={markings, mark=at position 0.5 with {\arrow{latex reversed}}}]
     (604,151)--(481,152) ;

% 左端水平线（箭头反向）
\draw[postaction={decorate}, decoration={markings, mark=at position 0.5 with {\arrow{latex reversed}}}]
    (171,150)--(48,151) ;

% 其余图形保持不变
\draw[dash pattern={on 0.84pt off 2.51pt}] (431,63) .. controls (405,63) and (404,244) .. (437,244);
\draw[dash pattern={on 0.84pt off 2.51pt}] (225,65) .. controls (250,66) and (250,246) .. (222,247);

% 文本与标记点
\node[font=\tiny, anchor=north west, inner sep=0.75pt] at (397,156.4) {\footnotesize$\eta_{r}$};
\fill (412,152) circle (1.2pt);
\node[font=\tiny, anchor=north west, inner sep=0.75pt] at (221,154.4) {\footnotesize$-\eta_{r}$};
\fill (244,150.5) circle (1.2pt);
\node[font=\tiny, anchor=north west, inner sep=0.75pt] at (173,154.4) {\footnotesize$-c_{r}$};
\node[font=\tiny, anchor=north west, inner sep=0.75pt] at (476,156.4) {\footnotesize$c_{r}$};
\node[font=\tiny, anchor=north west, inner sep=0.75pt] at (541,73.4) {\footnotesize$\Gamma_{1}^{(2)}$};
\node[font=\tiny, anchor=north west, inner sep=0.75pt] at (542,209.4) {\footnotesize$\Gamma_{1}^{(2)*}$};
\node[font=\tiny, anchor=north west, inner sep=0.75pt] at (106,78.4) {\footnotesize$\Gamma_{2}^{(2)}$};
\node[font=\tiny, anchor=north west, inner sep=0.75pt] at (95,204.4) {\footnotesize$\Gamma_{2}^{(2)*}$};
\node[font=\tiny, anchor=north west, inner sep=0.75pt] at (543,114.4) {\footnotesize$U_{1}^{(2)}$};
\node[font=\tiny, anchor=north west, inner sep=0.75pt] at (546,162.4) {\footnotesize$U_{1}^{(2)*}$};
\node[font=\tiny, anchor=north west, inner sep=0.75pt] at (85,112.4) {\footnotesize$U_{2}^{(2)}$};
\node[font=\tiny, anchor=north west, inner sep=0.75pt] at (85,163.4) {\footnotesize$U_{2}^{(2)*}$};

\fill (170,150) circle (1.2pt);
\fill (481,152) circle (1.2pt);

\end{tikzpicture}
\caption{The jump contours of RH problem for $M^{(2)}$ for $\xi\in\mathcal{T}_{\textup{\uppercase\expandafter{\romannumeral2}}}$.}
\label{fig:U_j domain RIII}
\end{figure}

\subsection{Analysis of RH problem for $M^{(2)}$}
Notice that $V^{(2)}(k)\rightarrow I$ as $t\rightarrow+\infty$ on the contours
$\Gamma^{(2)}_j\cup\Gamma_j^{(2)*}$ for $j=1,2$. Moreover, by item (d) in Proposition \ref{analytic extension of r RIII},
we have that, as $t\rightarrow +\infty$, $V^{(2)}(k)\to I$ for $k\in(-\infty,-c_r)\cup (c_r, +\infty)$.
Therefore, it follows that $M^{(2)}$ is approximated, to the leading order, by the global parametrix $M^{(\infty)}$ given below.

% Note that $V^{(2)}\rightarrow I$ as $t\rightarrow+\infty$ on the contours
% $\Gamma^{(2)}_j\cup\Gamma_j^{(2)*}$ for $j=1,2$, it follows that $M^{(2)}$ is approximated,
% to the leading order, by the global parametrix $M^{(\infty)}$ given below. The sub-leading
% contribution stems from the local behavior near the branch points $\pm c_{r}$, which is well
% approximated by the Bessel parametrix.

\subsubsection{Global parametrix}
The global parametrix satisfies the following RH problem.
%As $t$ large enough, the jump matrix $V^{(2)}$ approaches
%\begin{equation}\label{equ:jump Vinfty RIII}
 %   V^{(\infty)}=\begin{pmatrix} 0 & -1 \\ 1 & 0 \end{pmatrix}, \quad k\in(-c_{r},c_{r}).
%\end{equation}
%For $k\in\mathbb{C}\backslash[-c_{r},c_{r}]$, $V^{(2)}\rightarrow I$ as $t\rightarrow \infty$.
%Then it is naturally established the following parametrix for $M^{(\infty)}$.
\begin{RHP}\label{RHP:Minfty RIII}
\hfill
\begin{itemize}
    \item $M^{(\infty)}(k)$ is holomorphic for $k\in\mathbb{C}\setminus[-c_{r},c_{r}]$.
    \item $M^{(\infty)}(k)$ has continuous boundary values on $(-c_{r},c_{r})$ satisfying the following jump condition
    \begin{align*}
        M_{+}^{(\infty)}(k)=M_{-}^{(\infty)}(k)
        \begin{pmatrix}
            0 & -1\\
            1 & 0
        \end{pmatrix},\quad k\in (-c_r,c_r).
    \end{align*}
    \item As $k\rightarrow\infty$ in $\mathbb{C}\setminus[-c_{r}, c_{r}]$, we have $M^{(\infty)}(k)=I+\mathcal{O}(k^{-1})$.
    \item As $k\rightarrow \pm c_{r}$, $M^{(\infty)}(k)=\mathcal{O}((k\mp c_{r})^{-1/4})$.
\end{itemize}
\end{RHP}
Then the unique solution of $M^{(\infty)}(k)$ is given by
\begin{align}\label{equ:sol of Minfty RIII}
    M^{(\infty)}(k)=\Delta_r(k)
\end{align}
with $\Delta_{r}(k)$ defined by \eqref{equ:Delta_j} for the subscript $j$ being chosen as $r$.
For later use, the higher order asymptotic behavior of $M^{(\infty)}(k)$ as
$k\rightarrow c_r$ is
\begin{equation*}
    \begin{aligned}
    M^{(\infty)}(k)=&\frac{(2c_r)^{\frac{1}{4}}}{2(k-c_r)^{1/4}}\left[
        \begin{pmatrix}
            1 & -i\\
            i & 1
        \end{pmatrix}+\frac{(k-c_r)^{\frac{1}{2}}}{(2c_r)^{\frac{1}{2}}}
        \begin{pmatrix}
            1 & i\\
            -i & 1
        \end{pmatrix}+
        \frac{k-c_r}{8c_r}
        \begin{pmatrix}
            1 & -i\\
            i & 1
        \end{pmatrix}\right.\\
        &+\left.\frac{(k-c_r)^{\frac{3}{2}}}{4(2c_r)^{3/2}}
        \begin{pmatrix}
            -1 & -i\\
            i & -1
        \end{pmatrix}
        +\mathcal{O}\left(\left(k-c_r\right)^2\right)
    \right].
\end{aligned}
\end{equation*}
\subsubsection{Local parametrices near $\pm c_{r}$}
Let
\begin{align*}
    D_\varrho(c_r):=\left\{k: |k-c_r|<\varrho \right\}, \quad D_\varrho(-c_r):=\left\{k: |k+c_r|<\varrho \right\}
\end{align*}
be two small disks around $\pm c_r$ respectively, where
\begin{equation*}
\varrho<\frac{1}{3}{\rm min}\left\{ |c_r-\eta_r|, |c_l-c_r| ,t^{-\frac{2}{3}+\epsilon} \right\},\quad \frac{1}{2}<\epsilon<\frac{2}{3}.
\end{equation*}
For $j\in\{r,l\}$, we intend to solve the following local RH problem for $M^{(j)}$.
\begin{RHP}
\hfill
\begin{itemize}
    \item $M^{(j)}(k)$ is holomorphic for $k\in\overline{D_{\varrho}(\pm c_r)}\setminus\Gamma^{(j)}$ \textup{(}``$+$'' for $j=r$, and ``$-$'' for $j=l$\textup{)},
    where
    \begin{align*}
        \Gamma^{(r)}:=D_\varrho(c_r)\cap \Gamma^{(2)},\quad \Gamma^{(l)}:=D_\varrho(-c_r)\cap \Gamma^{(2)}.
    \end{align*}
    \item $M^{(j)}(k)$ has continuous boundary values $M_{\pm}^{(j)}(k)$ on $k\in \Gamma^{(j)}$ with the jump condition
    \begin{align*}
        M^{(j)}_{+}(k)=M^{(j)}_{-}(k)V^{(2)}(k)\big|_{\Gamma^{(j)}}.
    \end{align*}
    \item As $k\rightarrow\infty$ in $\mathbb{C}\setminus\Gamma^{(j)}$, we have $M^{(j)}(k)=I+\mathcal{O}(k^{-1})$.
\end{itemize}
\end{RHP}

In the remainder of this subsection, we focus on the construction of $M^{(r)}$ near $k=c_r$, and
the construction of $M^{(l)}$ near $-c_r$ follows by symmetry. Recall the asymptotics of $g_{\textup{\uppercase\expandafter{\romannumeral2}}}$ from \eqref{g_II asy cr}
\begin{equation*}
     g_{\textup{\uppercase\expandafter{\romannumeral2}}}(k)=g_{\textup{\uppercase\expandafter{\romannumeral2}}}^{(1)}(c_r)(k- c_r)^{\frac{1}{2}}+g_{\textup{\uppercase\expandafter{\romannumeral2}}}^{(2)}(c_r)(k- c_r)^{\frac{3}{2}}+\mathcal{O}((k- c_r)^{\frac{5}{2}}),\quad k\to c_r,
\end{equation*}
where
\begin{equation}\label{equ:coefficients of gII}
    g_{\textup{\uppercase\expandafter{\romannumeral2}}}^{(1)}(c_r)=6(2c_r)^{\frac{1}{2}}(c_r^2+2\xi),\quad g_{\textup{\uppercase\expandafter{\romannumeral2}}}^{(2)}(c_r)=3(2c_r)^{-\frac{1}{2}}(c_r^2+2\xi)+4(2c_r)^{\frac{3}{2}}.
\end{equation}

For $k \in D_{\varrho}(c_r)\setminus(-\infty,c_r)$, we set
\begin{equation}\label{vira cha 1 gII}
    \zeta_r(k)=-\left(\frac{3}{2}\right)^{\frac{2}{3}}t^{\frac{2}{3}}\left[g_{\textup{\uppercase\expandafter{\romannumeral2}}}-g_{\textup{\uppercase\expandafter{\romannumeral2}}}^{(1)}(c_r)(k-c_r)^{\frac{1}{2}}\right]^{\frac{2}{3}},
\end{equation}
which is a one-to-one conformal mapping and $\zeta_r^\prime(c_r)=-\left(\frac{3}{2}\right)^{\frac{2}{3}}t^{\frac{2}{3}}\left|g_{\textup{\uppercase\expandafter{\romannumeral2}}}^{(2)}(c_r)\right|^{\frac{2}{3}}<0$.
Then we define
\begin{equation*}
   S(k)= S(k;\xi):=\begin{cases}
        it\frac{g_{\textup{\uppercase\expandafter{\romannumeral2}}}^{(1)}(c_r)(k-c_r)^{\frac{1}{2}}}{\zeta_r(k)^{\frac{1}{2}}}, &k\in D_{\varrho}(c_r)\cap\mathbb{C}^+,\\
       - it\frac{g_{\textup{\uppercase\expandafter{\romannumeral2}}}^{(1)}(c_r)(k-c_r)^{\frac{1}{2}}}{\zeta_r(k)^{\frac{1}{2}}}, &k\in D_{\varrho}(c_r)\cap\mathbb{C}^-,
    \end{cases}
\end{equation*}
which implies that $S(k)$ is analytic in $D_{\varrho}(c_r)$ and
\begin{equation}\label{defSgII}
    s:=S(c_r)=-t^{\frac{2}{3}}\frac{g_{\textup{\uppercase\expandafter{\romannumeral2}}}^{(1)}(c_r)}{\left(\frac{3}{2}\right)^{\frac{1}{3}}\left|g_{\textup{\uppercase\expandafter{\romannumeral2}}}^{(2)}(c_r)\right|^{\frac{1}{3}}}.
\end{equation}
It can be concluded that
\begin{equation}
    \frac{4}{3}\zeta_r(k)^{\frac{3}{2}}+2S(k)\zeta_r(k)^{\frac{1}{2}}=\begin{cases}
        2itg_{\textup{\uppercase\expandafter{\romannumeral2}}}(k), &k\in D_{\varrho}(c_r)\cap\mathbb{C}^+,\\
        -2itg_{\textup{\uppercase\expandafter{\romannumeral2}}}(k), &k\in D_{\varrho}(c_r)\cap\mathbb{C}^-.
    \end{cases}
\end{equation}
Under the change of variable \eqref{vira cha 1 gII}, the construction of the local parametrix proceeds in a manner analogous to that described in Section \ref{subsubsec: local para pm cl}.
For brevity, we outline only the essential steps below.
Define the local parametrix $\widetilde{M}^{(r)}$ by
\begin{equation}\label{R3Localcr}
    \widetilde M^{(r)}(k)= \widetilde M^{(r)}(k;x,t):=P^{(r)}(k)\begin{pmatrix}
		1 & 0\\
		i a(s) & 1
		\end{pmatrix}M^{P_{34}}(\zeta_r; -\frac{\arg r(c_r)}{2\pi}, 0, s)e^{\hat\theta(\zeta_r)\sigma_3}Q^{(r)}(k),
\end{equation}
where $\hat\theta(\zeta_r)$ is defined as in \eqref{hattheta} and $M^{P_{34}}$ is given in Appendix \ref{p34}.
Here $P^{(r)}$ is the matching factor given by
\begin{equation}\label{equ:Pr RIII}
    P^{(r)}(k):=M^{(\infty)}(k){Q^{(r)}}(k)^{-1} \frac{1}{\sqrt{2}}(I-i\sigma_1)\zeta_{r}(k)^\frac{\sigma_3}{4},
\end{equation}
with $M^{(\infty)}$ defined in \eqref{equ:sol of Minfty RIII}, and
\begin{equation}\label{def:Q^{(r)} III}
    Q^{(r)}(k):=
        \begin{cases}
        \sigma_3,   &k\in\mathbb{C}^{+}\cap \overline{D_{\varrho}\left(c_r\right)}, \\
       \sigma_1,   &k\in\mathbb{C}^{-}\cap \overline{D_{\varrho}\left(c_r\right)}.\\
        \end{cases}
\end{equation}
The function $\widetilde{M}^{(r)}$ serves as an approximation to $M^{(r)}$ for sufficiently large $t$.
From the definition \eqref{R3Localcr}, its jump matrix $\widetilde V^{(r)}(k)$ is given by
\begin{equation*}
        \widetilde V^{(r)}(k)=
            \begin{cases}
            \begin{pmatrix} 1 & 0\\ e^{\arg r(c_r) i}e^{2\hat\theta(\zeta_r)} & 1\end{pmatrix}, &k\in\Gamma^{(2)}_1\cap D_\varrho(c_r),  \\
            \begin{pmatrix} 1 & -e^{-\arg r(c_r)i}e^{-2\hat\theta(\zeta_r)}\\ 0 & 1\end{pmatrix}, & k\in\Gamma^{(2)*}_1\cap D_\varrho(c_r),  \\
            \begin{pmatrix}0 & -1 \\ 1 & 0\end{pmatrix},  & k\in (c_{r}-\varrho, c_{r}),\\
             I,&k\in(c_r,c_r+\varrho).
            \end{cases}
\end{equation*}
It can be verified that $P^{(r)}(k)$ defined in \eqref{equ:Pr RIII} is analytic in $D_{\varrho}\left(c_r\right)$.
Moreover, following the steps to obtain \eqref{asy for Pr}, we can also derive that as $k\to c_r$,
\begin{equation}\label{asymptotic for P^R III}
    P^{(r)}(k)=\begin{pmatrix}
        \frac{1-i}{2}(2c_r)^{\frac{1}{4}}\left(\frac{3}{2}\right)^{\frac{1}{6}}t^{\frac{1}{6}}\left|g_{\textup{\uppercase\expandafter{\romannumeral2}}}^{(2)}(c_r)\right|^{\frac{1}{6}}&\frac{1-i}{2}(2c_r)^{-\frac{1}{4}}\left(\frac{3}{2}\right)^{-\frac{1}{6}}t^{-\frac{1}{6}}\left|g_{\textup{\uppercase\expandafter{\romannumeral2}}}^{(2)}(c_r)\right|^{-\frac{1}{6}}\\\frac{1+i}{2}(2c_r)^{\frac{1}{4}}\left(\frac{3}{2}\right)^{\frac{1}{6}}t^{\frac{1}{6}}\left|g_{\textup{\uppercase\expandafter{\romannumeral2}}}^{(2)}(c_r)\right|^{\frac{1}{6}}&\frac{-1-i}{2}(2c_r)^{-\frac{1}{4}}\left(\frac{3}{2}\right)^{-\frac{1}{6}}t^{-\frac{1}{6}}\left|g_{\textup{\uppercase\expandafter{\romannumeral2}}}^{(2)}(c_r)\right|^{-\frac{1}{6}}
    \end{pmatrix}+\mathcal{O}(|k-c_r|).
\end{equation}

For $k \in D_{\varrho}(-c_r)\setminus(-c_r,+\infty)$, we set
\begin{equation}\label{vira cha 2 gII}
    \zeta_l(k)=-\left(\frac{3}{2}\right)^{\frac{2}{3}}t^{\frac{2}{3}}\left[g_{\textup{\uppercase\expandafter{\romannumeral2}}}(k)-g_{\textup{\uppercase\expandafter{\romannumeral2}}}^{(1)}(c_r)(-k-c_r)^{\frac{1}{2}}\right]^{\frac{2}{3}},
\end{equation}
which is also a one-to-one conformal mapping and $\zeta_l^\prime(-c_r)=\left(\frac{3}{2}\right)^{\frac{2}{3}}t^{\frac{2}{3}}\left|g_{\textup{\uppercase\expandafter{\romannumeral2}}}^{(2)}(c_r)\right|^{\frac{2}{3}}>0$.

Similarly, the RH problem for $\widetilde M^{(l)}$ (an approximation of $M^{(l)}$) can be solved by the symmetry ($\widetilde M^{(l)}(k)=\sigma_1\widetilde M^{(r)}(-k)\sigma_1$).
Then, we obtain the local parametrix
\begin{equation}\label{R3Localcl}
    \widetilde M^{(l)}(k)=P^{(l)}(k)\sigma_1\begin{pmatrix}
		1 & 0\\
		i a(s) & 1
		\end{pmatrix}M^{P_{34}}(\zeta_l; -\frac{\arg r(c_r)}{2\pi}, 0, s)e^{\hat\theta(\zeta_r)\sigma_3}\sigma_1Q^{(l)}(k),
\end{equation}
where
\begin{align}
    &P^{(l)}(k)=\sigma_1P^{(r)}(-k)\sigma_1,&& k \in D_{\varrho}(-c_r),\label{def:Pl TII}\\
    &Q^{(l)}(k)=\sigma_1Q^{(r)}(-k)\sigma_1, && k \in D_{\varrho}(-c_r)\setminus\mathbb{R}.\nonumber
\end{align}
From \eqref{def:Pl TII}, $P^{(l)}(k)$ is also analytic in $D_{\varrho}(-c_r)$ with the following asymptotics as $k\to -c_r$,
\begin{equation}\label{asymptotic for P^l III}
     P^{(l)}(k)= \sigma_1P^{(r)}(c_r)\sigma_1+\mathcal{O}(|k+c_r|),\quad k\to -c_r.
\end{equation}
Now, we introduce the following lemma, which is similar to Lemma \ref{Lemma v3-vr-II}, and we omit its proof.
\begin{lemma}\label{Lemma v3-vr-III}
    For $\xi\in\mathcal{T}_{\textup{\uppercase\expandafter{\romannumeral2}}}$ and $t>0$, the function $\widetilde M^{(j)}$ has the following properties:  For $j\in\{l,r\}$,
    \begin{itemize}
        \item [\rm (a)] As $t\to\infty$,
        \begin{equation*}
            \|\widetilde M^{(j)}(k)M^\infty(k)^{-1}-I\|_{L^\infty(\partial D_\varrho(\pm c_r))}=\mathcal{O}(t^{1/3-\epsilon}).
        \end{equation*}
        Here, the script ``$\pm$'' stands for ``$+$'' when $j=r$, and ``$-$'' when $j=l$.
 \item [\rm (b)] Across $\Gamma^{(j)}$, the jump matrix $\widetilde V^{(j)}$ of $\widetilde M^{(j)}(k)$ satisfies
 \begin{align*}
    & \|V^{(2)}-\widetilde V^{(j)}\|_{L^1(\Gamma^{(j)})}=\mathcal{O}(t^{-1}),\\
    & \|V^{(2)}-\widetilde V^{(j)}\|_{L^2(\Gamma^{(j)})}=\mathcal{O}(t^{-2/3}),\\
     & \|V^{(2)}-\widetilde V^{(j)}\|_{L^\infty(\Gamma^{(j)})}=\mathcal{O}(t^{-1/3}).
    % &\|V^{(3)}-\tilde V^{(j)}\|_{(L^1\cap L^2\cap L^\infty)(\mathbb{R}\cap D_\varrho(\pm \eta_l))}=\mathcal{O}(t^{-1}).
 \end{align*}
\end{itemize}
\end{lemma}

\subsection{Small norm RH problem for $M^{(err)}$}\label{subsec:small norm RH RIII}
Define
\begin{align}\label{def:Merr RIII}
    M^{(err)}(k) =M^{(err)}(k;x,t):=
            \begin{cases}
             M^{(2)}(k)\left(M^{(\infty)}\left(k\right)\right)^{-1}, & k\in \mathbb{C}\setminus\left(D_{\varrho}\left(c_r\right)\cup D_{\varrho}\left(-c_r\right)\right),\\
             M^{(2)}(k)\left(\widetilde M^{(r)}\left(k\right)\right)^{-1}, & k\in D_{\varrho}\left(c_r\right),\\
             M^{(2)}(k)\left(\widetilde M^{(l)}\left(k\right)\right)^{-1}, & k\in D_{\varrho}\left(-c_r\right).
            \end{cases}
    \end{align}
 It is readily seen that $M^{(err)}$ satisfies the following RH problem.
    \begin{RHP}\label{RHP: Merr RIII}
    \hfill
    \begin{itemize}
        \item $M^{(err)}(k)$ is holomorphic for $k\in\mathbb{C}\setminus\Gamma^{(err)}$, where
        \begin{align*}
        \Gamma^{(err)}:=\partial D_{\varrho}\left(c_r\right)\cup\partial D_{\varrho}\left(-c_r\right)\cup \Gamma^{(3)}.
        \end{align*}
        \item $M^{(err)}(k)$ has continuous boundary values $M^{(err)}_\pm(k)$ on $k\in \Gamma^{(err)}$ with the jump condition
        \begin{equation*}
            M^{(err)}_{+}(k)=M^{(err)}_{-}(k)V^{(err)}(k),
        \end{equation*}
        where
         \begin{align}\label{equ:jump Verr RIII}
           V^{(err)}(k)=
                \begin{cases}
                 M^{(\infty)}_-(k)V^{(2)}(k){M^{(\infty)}_+(k)}^{-1}, & k\in\Gamma^{(3)}\backslash \left(\overline{D_\varrho(c_r)}\cup \overline{D_\varrho(-c_r)}\right),\\
                 \widetilde M^{(r)}(k){M^{(\infty)}(k)}^{-1}, & k\in\partial D_\varrho(c_r),\\
               \widetilde  M^{(l)}(k){M^{(\infty)}(k)}^{-1}, & k\in\partial D_\varrho(-c_r),\\
                 \widetilde M^{(r)}_-(k)V^{(2)}(k){\widetilde M^{(r)}_+(k)}^{-1}, & k\in\Gamma^{(3)}\cap D_\varrho(c_r),\\
                \widetilde M^{(l)}_-(k)V^{(2)}(k){\widetilde M^{(l)}_+(k)}^{-1}, & k\in\Gamma^{(3)}\cap D_\varrho(-c_r).
                \end{cases}
        \end{align}
        \item As $k\rightarrow\infty$ in $k\in\mathbb{C}\setminus\Gamma^{(err)}$, we have $M^{(err)}(k)=I+\mathcal{O}(k^{-1})$.
        \item As $k\rightarrow\pm c_r$, we have $M^{(err)}(k)=\mathcal{O}(1)$.
    \end{itemize}
    \end{RHP}

Using items (c)--(d) in Proposition \ref{analytic extension of r RIII} and Lemma \ref{Lemma v3-vr-III},
a straightforward calculation yields the following proposition.
\begin{proposition}\label{pro:est VE-II}
    For $\xi\in\mathcal{T}_\textup{\uppercase\expandafter{\romannumeral2}}$, the function $V^{(err)}(\cdot)-I:\Gamma^{(err)}\to\mathbb{C}^{2\times2}$ lies in $L^p\left(\Gamma^{(err)}\right)$ for $p=1,2,\infty$, and uniformly for $\xi\in\mathcal{T}_\textup{\uppercase\expandafter{\romannumeral2}}$, it holds that
    \begin{align*}
    &\Vert V^{(err)}(k)-I \Vert_{L^p\left(\Gamma^{(err)}\setminus\left(\overline{D_\varrho(c_r)}\cup \overline{D_\varrho(-c_r)}\cup\mathbb{R}\right)\right)}=
    \mathcal{O}(e^{-ct}), \\
    &\Vert V^{(err)}(k)-I \Vert_{L^p\left(\partial D_\varrho(c_r)\cup \partial D_\varrho(-c_r)\right)}=\mathcal{O}(t^{\gamma_p}),\\
    &\Vert V^{(err)}(k)-I \Vert_{L^p\left(\mathbb{R}\setminus\left(\overline{D_\varrho(c_r)}\cup \overline{D_\varrho(-c_r)}\right)\right)}=\mathcal{O}(t^{-1}),\\
    &\Vert V^{(err)}(k)-I \Vert_{L^p\left(\Gamma^{(err)}\cap \left(D_\varrho(c_r)\cup D_\varrho(-c_r)\right)\right)}=\mathcal{O}(t^{\omega_p}).
\end{align*}
where $c>0$,
\begin{align*}
    \gamma_1=-1/3,\;\gamma_2=-\frac{\epsilon}{2},\gamma_\infty=\frac{1}{3}-\epsilon,\quad
\omega_1=-1,\;\omega_2=-\frac{2}{3},\omega_\infty=-\frac{1}{3}.
\end{align*}
In particular, uniformly for $\xi\in\mathcal{T}_\textup{\uppercase\expandafter{\romannumeral2}}$, it holds that
\begin{equation}
    \Vert V^{(err)}(k)-I \Vert_{L^p\left(\Gamma^{(err)}\right)}=\mathcal{O}(t^{\gamma_p}), \quad p=1,2,\infty.
\end{equation}
\end{proposition}
\begin{proof}
    The proof is similar to that of Proposition \ref{pro:est VE-I}.
\end{proof}

The solution for $M^{(err)}$ can be given by
\begin{equation}\label{BCsolforErrorR3}
   M^{(err)}(k) = I + \frac{1}{2\pi i} \int_{\Gamma^{(err)}} \frac{(I + \varpi(z))(V^{(err)}(z) - I)}{z - k} \dif z,
\end{equation}
where $\varpi\in L^{2}(\Gamma^{(err)})$ is the unique solution of $(1 - C_{V^{(err)}})\varpi = C_{V^{(err)}} I$.
Analogous to the estimate \eqref{CverrL2}, we have
$\|C_{V^{(err)}}\|_{L^2(\Gamma^{(err)})} = \mathcal{O}(t^{1/3-\epsilon})$,
which implies that $1 - C_{V^{(err)}}$ is invertible for large positive $t$ in this case.

Moreover, by \eqref{BCsolforErrorR3}, it follows that
%as $k\to\infty$,
%For later use, we conclude this section with behavior of $M^{(err)}(k)$ at $k=\infty$.
%By \eqref{equ:BCsolforError}, it follows that
\begin{equation*}
    M^{(err)}(k)=I+\frac{M^{(err)}_1}{k}+\mathcal{O}(k^{-2}), \quad k\rightarrow\infty,
\end{equation*}
where
\begin{equation*}
    M^{(err)}_1=-\frac{1}{2\pi i}\int_{\Gamma^{(err)}}\left(I + \varpi(z)\right)(V^{(err)}(z)-I)\dif z.
\end{equation*}
Then we obtain the following proposition.

\begin{proposition}\label{prop: result of Merr_1 TII}
   For $\xi\in\mathcal{T}_\textup{\uppercase\expandafter{\romannumeral2}}$, as $t\rightarrow+\infty$,
    \begin{equation*}\label{equ:Merr_1 TII}
    \begin{aligned}
         M^{(err)}_1=&\left(\frac{3}{2}\right)^{-\frac{2}{3}}\left|g_{\textup{\uppercase\expandafter{\romannumeral2}}}^{(2)}(c_r)\right|^{-\frac{2}{3}}\begin{pmatrix}-it^{-\frac{1}{3}}(2c_r)^{\frac{1}{2}}\left(\frac{3}{2}\right)^{\frac{1}{3}}\left|g_{\textup{\uppercase\expandafter{\romannumeral2}}}^{(2)}(c_r)\right|^{\frac{1}{3}}a(s)&-it^{-\frac{2}{3}}\left[2^{\frac{1}{3}}a(s)q(-2^{\frac{1}{3}}s)+u(s)+\frac{s}{2}\right]\\it^{-\frac{2}{3}}\left[2^{\frac{1}{3}}a(s)q(-2^{\frac{1}{3}}s)+u(s)+\frac{s}{2}\right]&it^{-\frac{1}{3}}(2c_l)^{\frac{1}{2}}\left(\frac{3}{2}\right)^{\frac{1}{3}}\left|g_{\textup{\uppercase\expandafter{\romannumeral2}}}^{(2)}(c_r)\right|^{\frac{1}{3}} a(s)
  \end{pmatrix}\\&+\mathcal{O}(t^{\frac{1}{3}-2\epsilon}),
    \end{aligned}
    \end{equation*}
    where $g_{\textup{\uppercase\expandafter{\romannumeral2}}}^{(2)}(c_r)$ and $s$ are given by \eqref{equ:coefficients of gII} and \eqref{defSgII} respectively. Here, $a(s)$ is defined as in \eqref{p34entry-1}, $q(s)$ and $u(s)$ satisfy Painlev\'e \textup{\uppercase\expandafter{\romannumeral2}} equation in \eqref{painleve2} and Painlev\'e \textup{\uppercase\expandafter{\romannumeral34}} equation in \eqref{Bp34} respectively.
\end{proposition}
\begin{proof}
The definition of  $M^{(err)}_1$ in \eqref{equ:Merr_1 TI} implies that
   \begin{equation*}
    \begin{aligned}
   M^{(err)}_1=&-\frac{1}{2\pi i}\oint_{\partial D_\varrho(c_r)\cup\partial D_\varrho(c_r)}\left(V^{(err)}(z)-I\right)\dif z+\mathcal{O}(t^{\frac{1}{3}-2\epsilon})\\
    =&-\frac{1}{2\pi i}\oint_{\partial D_\varrho(c_r)}\frac{1}{\zeta_r}P^{(r)}(z)\begin{pmatrix}
        \left(M_1^{P_{34}}(s)\right)_{11}&t^{1/3}\left(M_1^{P_{34}}(s)\right)_{12}\\0&\left(M_1^{P_{34}}(s)\right)_{22}
    \end{pmatrix} P^{(r)}(z)^{-1}\dif z\\
    &-\frac{1}{2\pi i}\oint_{\partial D_\varrho(-c_r)}\frac{1}{\zeta_l}\sigma_1P^{(r)}(z)\begin{pmatrix}
        \left(M_1^{P_{34}}(s)\right)_{11}&t^{1/3}\left(M_1^{P_{34}}(s)\right)_{12}\\0&\left(M_1^{P_{34}}(s)\right)_{22}
    \end{pmatrix} P^{(r)}(z)^{-1}\sigma_1\dif z\\
    &+\mathcal{O}(t^{\frac{1}{3}-2\epsilon}).
\end{aligned}
\end{equation*}
Using the residue theorem and the definition of $\zeta_j$ for $j\in\{l,r\}$ in \eqref{vira cha 1 gII} and \eqref{vira cha 2 gII} completes the proof.
\end{proof}
\subsection{Proof of the part ($\textup{\uppercase\expandafter{\romannumeral2}}$) of Theorem \ref{thm:mainthm}}
By tracing back the transformations \eqref{def:M1 RIII}, \eqref{def:M2 RIII}, and \eqref{def:Merr RIII},
we conclude that, for $k\in\mathbb{C}\setminus\Gamma^{(2)}$,
\begin{equation}
    M(k)=M^{(err)}(k)M^{(\infty)}(k)e^{-it(g_{\textup{\uppercase\expandafter{\romannumeral2}}}(k)-\theta(k))\sigma_3}.
\end{equation}
Together with the reconstruction formula stated in \eqref{equ:recovering formula}, we  obtain
\begin{align*}
q(x,t)=2i\left[(M^{(err)}_1)_{12}+\lim_{k\rightarrow\infty}k\left(\Delta_{r}(k)\right)_{12}\right].
\end{align*}
Based on \eqref{equ:sol of Minfty RIII} and \eqref{equ:Merr_1 TII},  we obtain part (\textup{\uppercase\expandafter{\romannumeral2}}) of Theorem \ref{thm:mainthm}.

\section*{Acknowledgments}
The work of Fan is partially supported by NSFC under grant number 12271104.
\hspace*{\parindent}

\appendix	
\renewcommand\thefigure{\Alph{section}\arabic{figure}}
\setcounter{figure}{0}
\renewcommand{\theequation}{\thesection.\arabic{equation}}
\setcounter{equation}{0}
\section{The Painlev\'e XXXIV parametrix}\label{p34}

\begin{figure}[h]
	\centering
	\begin{tikzpicture}[scale=1.66]
	\draw [thick](-2,0)node[left]{\footnotesize$\Sigma_3$}--(0,0) ;
	\draw [thick](-1,1)node[above]{\footnotesize$\Sigma_2$}--(0,0);
	\draw [thick] (0,0)--(2,0)node[right]{\footnotesize$\Sigma_1$};
	\draw [thick] (-1,-1)node[below]{\footnotesize$\Sigma_4$}--(0,0) ;
	\draw[thick, -latex](-2,0)--(-1,0);
	\draw[thick, -latex](-1,1)--(-0.5,0.5);
	\draw[thick, -latex](-1,-1)--(-0.5,-0.5);
	\draw[thick, -latex](0,0)--(1,0);	
	\draw (0.3,0.2)node{\footnotesize$\Omega_1$} (-0.5,0.2)node{\footnotesize$\Omega_2$} (-0.5,-0.2)node{\footnotesize$\Omega_3$} (0.3,-0.2)node{\footnotesize$\Omega_4$};
	\end{tikzpicture}
	\caption{The jump contour of $M^{P_{34}}$.}
	\label{figp34}
\end{figure}
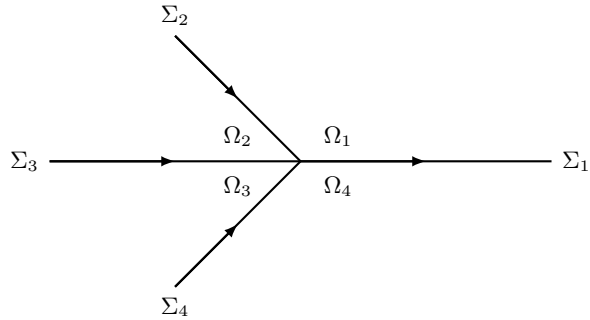

The Painlev\'e XXXIV parametrix satisfies the following RH problem.

\begin{RHP}\label{RHPp34}
	Find a $2\times 2$ matrix valued function $M^{P_{34}}(\zeta):=M^{P_{34}}(\zeta; b, \omega, s)$ with the following properties:
	\begin{itemize}
		\item $M^{P_{34}}(\zeta)$ is analytic in $
		\mathbb{C} \setminus \{\cup_{j=1}^4\Sigma_j\cup\{0\}\}$, where
		\begin{align*}
		\Sigma_1=\R^+, ~~ \Sigma_2=\E^{\frac{2 \pi i}{3}}\R^+, ~~\Sigma_3=\E^{   \pi i   }\R^+,~~\Sigma_4=\E^{-\frac{2 \pi i}{3}}\R^+
		\end{align*}
		with the orientations as shown in Figure \textup{\ref{figp34}}.
		
		\item $M^{P_{34}}(\zeta)$  satisfies the jump condition
		\begin{equation*}
		M^{P_{34}}_+ (\zeta)=M^{P_{34}}_- (\zeta)
		\left\{ \begin{array}{ll}
		\begin{pmatrix}
		1 & \omega
		\\
		0 & 1
		\end{pmatrix}, &\quad \zeta \in {\Sigma}_1,
		\\[9pt]
		\begin{pmatrix}
		1 & 0 \\
		e^{2b\pi i} & 1
		\end{pmatrix}, &\quad \zeta \in {\Sigma}_2,
		\\[9pt]
		\begin{pmatrix}
		0 & 1 \\
		-1 & 0
		\end{pmatrix},& \quad
		\zeta \in {\Sigma}_3,
		\\[9pt]
		\begin{pmatrix}
		1 & 0 \\
		e^{-2b\pi i} & 1
		\end{pmatrix},  & \quad \zeta \in \Sigma_4.
		\end{array}
		\right .
		\end{equation*}
		
		\item As $\zeta \to \infty$, there exists a function $a(s):=a(s;b,\omega)$ such that
		\begin{align} \label{eq:Psi-infinity}
		M^{P_{34}}(\zeta) = \begin{pmatrix}
		1 & 0\\
		-i a(s) & 1
		\end{pmatrix}
		\left(I+\frac {M^{P_{34}}_1(s)}{\zeta}
		+\mathcal{O} \left( \zeta^{-2} \right) \right)
		\frac{\zeta^{-\frac{1}{4}\sigma_3}}{\sqrt{2}}
		\begin{pmatrix}
		1 & i
		\\
		i & 1
		\end{pmatrix} e^{-(\frac{2}{3}\zeta^{\frac{3}{2}}+s\zeta^{\frac{1}{2}}) \sigma_3},
		\end{align}
		where we take the  principle branch for the fractions and
		\begin{align}
		\left(M^{P_{34}}_1(s)\right)_{12}&=i 2^{-\frac{2}{3}}(H-q)(-2^{\frac{1}{3}}s):=ia(s),\label{p34entry-1}\\
        \left(M^{P_{34}}_1(s)\right)_{11}&=2^{-\frac{7}{3}}(H^2-q^2)(-2^{\frac{1}{3}}s)-2^{-\frac{4}{3}}(q_s+qH)(-2^{\frac{1}{3}}s),\label{p34entry-2}\\
        \left(M^{P_{34}}_1(s)\right)_{22}&=2^{-\frac{7}{3}}(H^2-q^2)(-2^{\frac{1}{3}}s)+2^{-\frac{4}{3}}(q_s+qH)(-2^{\frac{1}{3}}s),\label{p34entry-3}
		\end{align}
		with $q(s)$ is the generalized Hastings-Mcleod solution of the Painlev\'e \textup{II} equation
        \begin{equation}\label{painleve2}
            q''(s)=sq(s)+2q^3-\nu,\quad \nu=2b+\frac{1}{2},
        \end{equation}
        which is characterized by the following asymptotics
        \begin{equation*}
            q(s)=\begin{cases}
                \frac{\nu}{s}+\mathcal{O}(s^{-4}), &s\to+\infty,\\
               \sqrt{-\frac{s}{2}}+\mathcal{O}(s^{-1}), &s\to-\infty.
            \end{cases}
        \end{equation*}
        Moreover, $H(s)$ is the Hamiltonian for Painlev\'e \rm{II} equation
        \begin{equation}\label{hamil for p2}
            H(s)=\left(q'(s)\right)^2-sq^2(s)-q^4(s)+2\nu q(s),\quad  H'(s)=-q^2.
        \end{equation}
		\item As $\zeta \to 0$, we have, if  $-\frac{1}{2} < b < 0$,
		\begin{equation*}
		M^{P_{34}}(\zeta)=\bigo{\zeta^{b}},
		\end{equation*}
		and if $b \geq 0$,
		\begin{equation*}
		M^{P_{34}}(\zeta)=\left\{ \begin{array}{ll}
		\begin{pmatrix}
		\bigo{\zeta^{b}} & \bigo{\zeta^{-b}}
		\\
		\bigo{\zeta^{b}} & \bigo{\zeta^{-b}}
		\end{pmatrix}, &\quad \zeta \in \Omega_1\cup\Omega_4,
		\\
		\bigo{\zeta^{-b}}, &\quad \zeta \in \Omega_2\cup\Omega_3,
		\end{array}
		\right .
		\end{equation*}
		 where the domains $\Omega_j$, $j=1,2,3,4$ are shown in Figure \ref{figp34}.
		
	\end{itemize}
\end{RHP}

By \cite{ikj2008,its2009}, the above RH problem is uniquely solvable for $b>-\frac{1}{2}$, $\omega \in \mathbb{C}\setminus (-\infty,0)$, and $s\in \mathbb{R}$. Moreover, with $a(s)$ given in
\eqref{eq:Psi-infinity}, the function
$$
u(s)=u(s;b,\omega):=a'(s;b,\omega)-\frac s2
$$
satisfies the  Painlev\'{e}  XXXIV equation
\begin{equation}\label{Bp34}
u''(s)=4u(s)^2+2su(s)+\frac{u'(s)^2-(2b)^2}{2u(s)}.
\end{equation}
%and is pole-free on the real axis.
Particularly, one has
\begin{equation*}
u(s;b, 0)=\left\{ \begin{array}{ll}
\frac{b}{\sqrt{s}}+\bigo{s^{-2}}, &\qquad s\to +\infty,
\\
-\frac{s}{2}+\bigo{s^{-2}}, &\qquad s\to -\infty.
\end{array}\right .
\end{equation*}
This, together with the fact that $a(s;b,0)\to 0$ as $s\to -\infty$, implies that
\begin{equation*}
a(s;b,0)=\int_{-\infty}^s \left( u(z;b, 0) + \frac z2\right) \dd z.
\end{equation*}

Moreover, for $\xi \in \mathcal{T}_\textup{\uppercase\expandafter{\romannumeral1}}$, we consider the special case of the RH problem \ref{RHPp34} with $b=0$, which by \eqref{painleve2} implies that  $\nu=1/2$. It can be inferred from \cite[Section 4.1]{ikj2008} that $q(s)$ admits a special solution that can be expressed in terms of the Airy function as
\begin{equation}\label{equ:spe p2}
    q(s)=-2^{-\frac{1}{3}}\frac{\textup{Ai} '(-2^{-\frac{1}{3}}s)}{\textup{Ai} (-2^{-\frac{1}{3}}s)}.
\end{equation}
Indeed, from the asymptotics of Airy functions, we have
\begin{equation*}
    q(s)\sim\frac{1}{2}\sqrt{2}(-s)^{\frac{1}{2}},\quad s\to-\infty.
\end{equation*}
%Moreover, it can be calculated that for $\zeta\in\Omega_3$, $M^{P_{34}}(\zeta)$ takes the form
%\begin{equation*}
 %   M^{P_{34}}(\zeta;0,\omega,s)=\sqrt{2\pi}\begin{pmatrix}
 %       e^{\pi i/3}\textup{Ai} \left( e^{-2\pi i/3}(\zeta+s)\right)&-e^{-\pi i/3}\textup{Ai} \left( e^{2\pi i/3}(\zeta+s)\right)\\-ie^{-\pi i/3}\textup{Ai}' \left( e^{-2\pi i/3}(\zeta+s)\right)&ie^{\pi i/3}\textup{Ai} '\left( e^{2\pi i/3}(\zeta+s)\right)
%    \end{pmatrix}.
%\end{equation*}
Substituting the formula \eqref{equ:spe p2} into $H(s)$ in \eqref{hamil for p2}, and using the properties of the Airy function ($\textup{Ai}'' (s)=s\textup{Ai} (s)$), we  obtain
\begin{align}
		\left(M^{P_{34}}_1(s)\right)_{12}&= \frac{s^2i}{4},\label{speentry-1}\\
        \left(M^{P_{34}}_1(s)\right)_{11}&=\frac{s^4}{32}-\frac{s}{4},\label{speentry-2}\\
        \left(M^{P_{34}}_1(s)\right)_{22}&=-\frac{s^2}{4}\frac{\textup{Ai} '(s)}{\textup{Ai} (s)}+\frac{s^4}{32}+\frac{s}{4}.\label{speentry-3}
		\end{align}

\section*{Acknowledgements}
\addcontentsline{toc}{section}{Acknowledgements}
This work is supported by the National Natural Science Foundation of China   {(Grant No. 12271104)}.

\end{document}